\documentclass[11pt]{amsart}

\usepackage[margin=1in]{geometry}

\usepackage{amsmath,amsthm,amsfonts}
\usepackage{mathtools}
\usepackage{mathrsfs}

\usepackage{tikz}
\usepackage{hyperref}

\usepackage{comment}

\usepackage[english]{babel}

\theoremstyle{plain}

\newtheorem{thm}{Theorem}
\newtheorem{cor}[thm]{Corollary}
\newtheorem{prop}[thm]{{\bf Proposition}}
\newtheorem{lem}[thm]{{\bf Lemma}}

\newcounter{hyp_counter}
\theoremstyle{definition}
\newtheorem{claim}{Claim}
\newtheorem{defn}[thm]{Definition}

\theoremstyle{remark}
\newtheorem{rem}{Remark}

\newcommand{\CP}{\mathbb{C}\operatorname{P}}
\newcommand{\Det}{\operatorname{Det}}

\newcommand{\End}{\operatorname{End}}

\newcommand{\Id}{\operatorname{Id}}
\newcommand{\inj}{\operatorname{inj}}
\newcommand{\Isom}{\operatorname{Isom}}

\newcommand{\Diff}{\operatorname{Diff}}
\newcommand{\GL}{\operatorname{GL}}

\newcommand{\Gr}{\operatorname{Gr}}

\newcommand{\HP}{\mathbb{H}\operatorname{P}}
\newcommand{\len}{\operatorname{len}}

\newcommand{\N}{\mathbb{N}}
\newcommand{\PSU}{\operatorname{PSU}}
\newcommand{\R}{\mathbb{R}}
\newcommand{\RP}{\mathbb{R}\operatorname{P}}
\newcommand{\SO}{\operatorname{SO}}
\newcommand{\SU}{\operatorname{SU}}
\newcommand{\Sp}{\operatorname{Sp}}
\newcommand{\Spin}{\operatorname{Spin}}
\newcommand{\stab}{\operatorname{stab}}

\newcommand{\Tr}{\operatorname{Tr}}

\newcommand{\vol}{\operatorname{vol}}
\newcommand{\Vect}{\operatorname{Vect}}
\newcommand{\Z}{\mathbb{Z}}

\newcommand{\wt}[1]{\widetilde{#1}}

\newcommand{\abs}[1]{\left| #1\right|}
\newcommand{\mb}[1]{\mathbb{#1}}

\newcommand{\mc}[1]{\mathcal{#1}}
\newcommand{\mf}[1]{\mathfrak{#1}}

\newcommand{\pez}[1]{\left( #1\right)}

\newcommand{\epOne}{\theta}
\newcommand{\epTwo}{\eta}
\newcommand{\epThree}{\kappa}

\makeatletter
\def\blfootnote{\xdef\@thefnmark{}\@footnotetext}
\makeatother

\title{Simultaneous Linearization of Diffeomorphisms of Isotropic Manifolds}
\author{Jonathan DeWitt}
\address{Department of Mathematics, The University of Chicago, Chicago, IL 60637, USA}
\email{dewitt@uchicago.edu}
\date{\today}

\begin{document}

\maketitle

\begin{abstract}Suppose that $M$ is a closed isotropic Riemannian manifold and that $R_1,...,R_m$ generate the isometry group of $M$. Let $f_1,...,f_m$ be smooth perturbations of these isometries. We show that the $f_i$ are simultaneously conjugate to isometries if and only if their associated uniform Bernoulli random walk has all Lyapunov exponents zero. This extends a linearization result of Dolgopyat and Krikorian \cite{dolgopyat2007simultaneous} from $S^n$ to real, complex, and quaternionic projective spaces. In addition, we identify and remedy an oversight in that earlier work.
\end{abstract}

\section{Introduction}

A \blfootnote{This material is based upon work supported by the National Science Foundation Graduate Research Fellowship under Grant No. DGE-1746045.}basic problem in dynamics is determining whether two dynamical systems are equivalent. A standard notion of equivalence is conjugacy: if $f$ and $g$ are two diffeomorphisms of a manifold $M$, then $f$ and $g$ are \emph{conjugate} if there exists a homeomorphism $h$ of $M$ such that $hfh^{-1}=g$. Some classes of dynamical systems are distinguished up to conjugacy by a small amount of dynamical information.  One of the most basic examples of this is Denjoy's theorem: a $C^2$ orientation preserving circle diffeomorphism with irrational rotation number is conjugate to a rotation \cite[\textsection 12.1]{katok1997introduction}. In the case of Denjoy's theorem, the rotation number is all the information needed to determine the topological equivalence class of the diffeomorphism under conjugacy.

Rigidity theory focuses on identifying dynamics that are distinguished up to conjugacy by particular kinds of dynamical information such as the rotation number. There are finer dynamical invariants than rotation number which require a finer notion of equivalence to study. For instance, one obtains a finer notion of equivalence if one insists that the conjugacy be a $C^1$ or even $C^{\infty}$ diffeomorphism. A smoother conjugacy allows one to consider invariants such as Lyapunov exponents, which may not be preserved under conjugacy by homeomorphisms. For a single volume preserving Anosov diffeomorphism, the Lyapunov exponents with respect to volume are invariant under conjugation by $C^1$ volume preserving maps. Consequently, one is naturally led to ask, ``If two volume preserving Anosov diffeomorphisms have the same Lyapunov exponents are the two $C^1$ conjugate?" In some circumstances the answer is ``yes". Such situations where knowledge about Lyapunov exponents implies systems are conjugate by a $C^1$ diffeomorphism are instances of a phenomenon called ``Lyapunov spectrum rigidity". See \cite{gogolev2019rigidity} for examples and discussion of this type of rigidity. For recent examples, see \cite{butler2017characterizing}, \cite{dewitt2019local},\cite{gogolev2019smooth},\cite{gogolev2018local}, and \cite{saghin2019lyapunov}.

In rigidity problems related to isometries, it is often natural to consider a family of isometries. A collection of isometries may have strong rigidity properties even if the individual elements of the collection do not. For example, Fayad and Khanin \cite{fayad2009smooth} proved that a collection of commuting diffeomorphisms of the circle whose rotation numbers satisfy a simultaneous Diophantine condition are smoothly simultaneously conjugated to rotations. Their result is a strengthening of an earlier result of Moser \cite{moser1990commuting}. A single diffeomorphism in such a collection might not satisfy the Diophantine condition on its own.

Although the two types of rigidity described above occur in the dissimilar hyperbolic and elliptic settings, a result of Dolgopyat and Krikorian combines the two. They introduce a notion of a Diophantine set of rotations of a sphere and use this notion to prove that certain random dynamical systems with all Lyapunov exponents zero are conjugated to isometric systems \cite{dolgopyat2007simultaneous}. Our result is a generalization of this result to the setting of isotropic manifolds. We now develop the language to state both precisely.

Let $(f_1,...,f_m)$ be a tuple of diffeomorphisms of a manifold $M$. Let $(\omega_i)_{i\in \mathbb{N}}$ be a sequence of independent and identically distributed random variables with uniform distribution on $\{1,...,m\}$. Given an initial point $x_0\in M$, define $x_n=f_{\omega_n}x_{n-1}$. This defines a Markov process on $M$. We refer to this process as the random dynamical system associated to the tuple $(f_1,...,f_m)$. Let $f_{\omega}^n$ be defined to equal $f_{\omega_n}\circ\cdots \circ  f_{\omega_1}$. We say that a probability measure $\mu$ on $M$ is a \emph{stationary measure} for this process if $m^{-1}\sum_{i=1}^m (f_i)_*\mu=\mu$.  A stationary measure is \emph{ergodic} if it is not a non-trivial convex combination of two distinct stationary measures. Fix an ergodic stationary measure $\mu$. For $\mu$-almost every $x$, almost surely for any $v\in T_xM\setminus\{0\}$, the following limit exists
\begin{equation}\label{eqn:lyapunov_exponent_definition}
\lim_{n\to \infty} \frac{1}{n}\ln \|D_xf^n_\omega v\|
\end{equation}
and takes its value in a fixed finite list of numbers depending only on $\mu$:
\begin{equation}\label{eqn:list_of_lyapunov_exponents}
\lambda_1(\mu)\ge \lambda_2(\mu)\ge \cdots \ge\lambda_{\dim M}(\mu).
\end{equation}
These are the \emph{Lyapunov exponents} with respect to $\mu$. In fact, for almost every $\omega$ and $\mu$-a.e.~$x$ there exists a flag $V_1\subset \cdots \subset V_j$ inside $T_xM$ such that if $v\in V_i\setminus V_{i-1}$ then the limit in \eqref{eqn:list_of_lyapunov_exponents} is equal to $\lambda_{\dim M-\dim V_i}$. The number of times a particular exponent appears in \eqref{eqn:list_of_lyapunov_exponents} is given by $\dim V_i- \dim V_{i-1}$; this number is referred to as the multiplicity of the exponent. For more information, see \cite{kifer1986ergodic}.

Our result holds for isotropic manifolds. By definition, an \emph{isotropic manifold} is a Riemannian manifold whose isometry group acts transitively on its unit tangent bundle. The closed isotropic manifolds are $S^n$, $\RP^n$, $\CP^n$, $\HP^n$, and the Cayley projective plane. In the following we write $G^{\circ}$ for the identity component of a Lie group $G$.

\begin{thm}\label{thm:KAM_converges}
Let $M^d$ be a closed isotropic Riemannian manifold other than $S^1$. There exists $k_0$ such that if $(R_1,...,R_m)$ is a tuple of isometries of $M$ such that the subgroup of $\Isom(M)$ generated by this tuple contains $\Isom(M)^{\circ}$, then there exists $\epsilon_{k_0}>0$ such that the following holds. Let $(f_1,...,f_m)$ be a tuple of $C^{\infty}$ diffeomorphisms satisfying $\max_{i} d_{C^{k_0}}(f_i,R_i)<\epsilon_{k_0}$. Suppose that there exists a sequence of ergodic stationary measures $\mu_n$ for the random dynamical system generated by $(f_1,...,f_m)$ such that $\abs{\lambda_d(\mu_n)}\to 0$, then there exists $\psi\in \Diff^{\infty}(M)$ such that for each $i$ the map $\psi f_i\psi^{-1}$ is an isometry of $M$ and lies in the subgroup of $\Isom(M)$ generated by $(R_1,\ldots,R_m)$.
\end{thm}

\noindent Dolgopyat and Krikorian proved Theorem \ref{thm:KAM_converges} in the case of $S^n$ \cite[Thm. 1]{dolgopyat2007simultaneous}. 

Dolgopyat and Krikorian also obtained a Taylor expansion of the Lyapunov exponents of the stationary measure of the perturbed system \cite[Thm. 2]{dolgopyat2007simultaneous}. 
Fix $(R_1,\ldots,R_m)$ generating $\Isom(S^n)^{\circ}$. Let $(f_1,\ldots, f_m)$ be a $C^{k_0}$ small perturbation of $(R_1,\ldots,R_m)$ and $\mu$ be any ergodic stationary measure for $(f_1,\ldots,f_m)$. Let $\Lambda_r=\lambda_1+\cdots+\lambda_r$ denote the sum of the top $r$ Lyapunov exponents.
In \cite[Thm. 2]{dolgopyat2007simultaneous}, it is shown that the Lyapunov exponents of $\mu$ satisfy
\begin{equation}\label{eqn:DK_taylor_expansion}
\lambda_r(\mu)=\frac{\Lambda_d}{d}+\frac{d-2r+1}{d-1}\left(\lambda_1-\frac{\Lambda_d}{d}
\right)+o(1)|\lambda_d(\mu)|,
\end{equation}
where $o(1)$ goes to zero as $\max_i d_{C^{k_0}}(f_i,R_i)\to 0$. Using this formula Dolgopyat and Krikorian obtain an even stronger dichotomy for systems on even dimensional spheres: either $(f_1,\ldots,f_m)$ is simultaneously conjugated to isometries or the Lyapunov exponents of every ergodic stationary measure of the perturbation are uniformly bounded away from zero. By using this result they show if $(R_1,\ldots,R_m)$ generates $\Isom(S^{2n})^{\circ}$ and $(f_1,\ldots,f_m)$ is a $C^{k_0}$ small perturbation such that each $f_i$ preserves volume, then volume is an ergodic stationary measure for $(f_1,\ldots, f_m)$ \cite[Cor. 2]{dolgopyat2007simultaneous}.

It is natural to ask if a similar Taylor expansion can be obtained in the setting of isotropic manifolds. Proposition \ref{prop:taylor_expansion_lambda_k} shows that $\Lambda_r$ may be Taylor expanded assuming that $(R_1,...,R_m)$ generates $\Isom(M)^{\circ}$ and the induced action of $\Isom(M)^{\circ}$ on $\Gr_r(M)$, the Grassmanian bundle of $r$-planes in $TM$, is transitive.

In Theorem \ref{thm:taylor_expansion_of_top_and_bottom}, we give a Taylor expansion relating $\lambda_1$ and $\lambda_d$ which holds for isotropic manifolds.
However, we cannot Taylor expand every Lyapunov exponent as in equation \eqref{eqn:DK_taylor_expansion} because if a manifold does not have constant curvature then its isometry group cannot act transitively on the two-planes in its tangent spaces. The argument of Dolgopyat and Krikorian requires that the isometry group act transitively on the space of $k$-planes in $TM$ for $0\le k\le d$. 

It is natural to ask why the proof of Theorem \ref{thm:KAM_converges} does not work in the case of $S^1$ even though $S^1$ is isotropic. 
As Proposition \ref{prop:decay_estimate} shows, for a tuple $(R_1,\ldots,R_m)$ as in the theorem, uniformly small perturbations of $(R_1,\ldots, R_m)$ are uniformly Diophantine in a sense explained below. This uniformity is used crucially in the proof when we change the tuple of isometries that we are working with. The same uniformity of Diophantineness does not hold for tuples of isometries of $S^1$: a small perturbation may lose all Diophantine properties. The reason that the proof of Proposition \ref{prop:decay_estimate} does not work for $S^1$ is that the isometry group of $S^1$ is not semi-simple.

There are not many other results like Theorem \ref{thm:KAM_converges}. In addition to the aforementioned result of Dolgopyat and Krikorian, there are some results of Malicet. In \cite{malicet2012simultaneous}, a similar linearization result is obtained that applies to a particular type of map of $\mathbb{T}^2$ that fibers over a rotation on $S^1$. In a recent work, Malicet  obtained a Taylor expansion of the Lyapunov exponent for a perturbation of a Diophantine random dynamical system on the circle \cite{malicet2020lyapunov}.

\

\noindent{{\bf Acknowledgements.}} The author thanks Aaron Brown and Amie Wilkinson for their critical comments during all parts of this project. The author also thanks Dmitry Dolgopyat for his generosity in answering the author's questions about \cite{dolgopyat2007simultaneous}. The author is also grateful to the anonymous referees for carefully reading the manuscript and providing many useful comments and suggestions.

\subsection{Outline}

The proof of Theorem \ref{thm:KAM_converges} follows the general argument of \cite{dolgopyat2007simultaneous}. For readability, the argument in this paper is self-contained. While a number of the results below appear in \cite{dolgopyat2007simultaneous}, we have substantially reformulated many of them and in many places offer a different proof. Doing so is not merely a courtesy to the reader: the results in \cite{dolgopyat2007simultaneous} are stated in too narrow a setting for us to use. Simply stating more general reformulations would unduly burden the reader's trust. In addition, as will be discussed below, there are some oversights in \cite{dolgopyat2007simultaneous} which we explain in subsection \ref{subsec:oversight_and_remedy} and that we have remedied in Section \ref{sec:diffs_of_small_strain}. We have also stated intermediate results and lemmas in more generality than is needed for the proof of Theorem \ref{thm:KAM_converges} so that they may be used by others. Below we sketch the general argument of the paper and emphasize some differences with the approach in \cite{dolgopyat2007simultaneous}.

The proof of Theorem \ref{thm:KAM_converges} is by an iterative KAM convergence scheme. Fix a closed isotropic manifold $M$. We start with a tuple of diffeomorphisms $(f_1,\ldots,f_m)$ nearby to a tuple of isometries $(R_1,\ldots,R_m)$.  We must find some smooth diffeomorphism $\psi$ such that $\wt{f}_i\coloneqq \psi f_i\psi^{-1}\in \Isom(M)$. To do this we produce a conjugacy $\psi$ that brings each $f_i$ closer to being an isometry. To judge the distance from being an isometry, we define a strain tensor that vanishes precisely when a diffeomorphism is an isometry. By solving a particular coboundary equation and using that the Lyapunov exponents are zero,  we can construct $\psi$ so that $\wt{f}_i$ has small strain tensor. In our setting, a diffeomorphism with small strain is near to an isometry, so $(\wt{f}_1,\ldots,\wt{f}_m)$ is near to a tuple of isometries $(R_1',\ldots,R_m')$. We then repeat the procedure using these new tuples as our starting point. The results of performing a single step of this procedure comprise Lemma \ref{lem:KAM_step}. Once Lemma \ref{lem:KAM_step} is proved, the rest of the proof of Theorem \ref{thm:KAM_converges} is bookkeeping that checks that the procedure converges. Most of the paper is in service of the proof Lemma \ref{lem:KAM_step}, which gives the result of a single step in the convergence scheme. 

 Proofs of technical and basic facts are relegated to a significant number of appendices. This has been done to focus the main exposition on the important ideas in the proof of Theorem \ref{thm:KAM_converges} and not on the technical details. The appendices that might be most beneficial to look at before they are referenced in the text are appendices \ref{sec:ck_calculus} and \ref{sec:interpolation}. These appendices concern $C^k$ calculus and interpolation inequalities. Both contain estimates that are common in KAM arguments. The organization of the main body of the paper reflects the order of the steps in the proof of Lemma \ref{lem:KAM_step}. There are several important results in the proof of Lemma \ref{lem:KAM_step}, which we now describe.

 The first part of the proof of Lemma \ref{lem:KAM_step} requires that a particular coboundary equation can be tamely solved. The solution to this equation is one of the main subjects of Section \ref{sec:diophantine}. The equation is solved in Proposition \ref{prop:tame_coboundaries}. This proposition is essential in the work of Dolgopyat and Krikorian \cite{dolgopyat2007simultaneous} and its proof follows from the appendix to \cite{dolgopyat2002mixing}; it relies on a Diophantine property of the tuple of isometries $(R_1,\ldots, R_m)$. This property is formulated in subsection \ref{subsec:diophantine_sets}. The stability of this property under perturbations is crucial in the proof and an essential feature of our setting. In addition, the argument in Section \ref{sec:diophantine} is different from Dolgopyat's earlier argument because we we use the Solovay-Kitaev algorithm (Theorem \ref{thm:real_solovay_kitaev}), which is more efficient than the procedure used in the appendix to \cite{dolgopyat2002mixing}.

Section \ref{sec:stationary_measures} considers stationary measures for perturbations of $(R_1,\ldots,R_m)$. Suppose $M$ is a quotient of its isometry group, its isometry group is semisimple, and  $(R_1,\ldots,R_m)$ is a Diophantine subset of $\Isom(M)$. Suppose $(f_1,\ldots,f_m)$ is a small smooth perturbation of $(R_1,\ldots,R_m)$. There is a relation between a stationary measure $\mu$ for the perturbed system and the Haar measure.
Proposition \ref{prop:taylor_expansion_of_haar} relates integration against $\mu$ with integration against the Haar measure. Lyapunov exponents are calculated by integrating the $\log $ Jacobian against a stationary measure of an extended dynamical system on a Grassmannian bundle over $M$. Consequently, this proposition relates stationary measures and their Lyapunov exponents to the volume on a Grassmannian bundle.

The relationship between Lyapunov exponents and stationary measures is explained in Section \ref{sec:strain_and_Lyapunov_exponents}. Proposition \ref{prop:taylor_expansion_lambda_k} provides a Taylor expansion of the sum of the top $r$ Lyapunov exponents of a stationary measure $\mu$. Three terms appear in the Taylor expansion.  The first two terms have a direct geometric meaning, which we interpret in terms of strain tensors introduced in subsection \ref{subsec:strain}.  The final term in the Taylor expansion depends on a quantity $\mc{U}(\psi)$. This quantity does not have a direct geometric interpretation. However, in the proof of Lemma \ref{lem:KAM_step}, we show that by solving the coboundary equation from Proposition \ref{prop:tame_coboundaries} the quantity $\mc{U}(\psi)$ can be made to vanish. Once $\mc{U}(\psi)$ vanishes, then we have an equation directly relating Lyapunov exponents to the strain. This equation then allows us to conclude that a diffeomorphism with small Lyapunov exponents also has small strain. We reformulate in a Riemannian geometric setting some arguments of \cite{dolgopyat2007simultaneous} by using the strain tensor. This gives coordinate-free expressions that are easier to interpret. 

Section \ref{sec:diffs_of_small_strain} contains the most important connection between the strain tensor and isometries: diffeomorphisms of small strain on isotropic manifolds are near to isometries. The basic geometric fact proved in Section \ref{sec:diffs_of_small_strain} is Theorem \ref{thm:near_to_identity}, which is true on any manifold. Theorem \ref{thm:near_to_identity} is then used to prove Proposition \ref{lem:H_0_nearby_isometry}, which is a more technical result adapted for use in the KAM scheme. Proposition \ref{lem:H_0_nearby_isometry} then allows us to prove that our conjugated tuple is near to a new tuple of isometries, which allows us to repeat the process. 

All of the previous sections combine in Section \ref{sec:KAM_scheme} to prove Lemma \ref{lem:KAM_step}. We then obtain the main theorem, Theorem \ref{thm:KAM_converges}, and prove an additional theorem that relates the top and bottom Lyapunov exponents of a perturbation, Theorem \ref{thm:taylor_expansion_of_top_and_bottom}.

\subsection{An oversight and its remedy}\label{subsec:oversight_and_remedy}

Section \ref{sec:diffs_of_small_strain} is entirely new and different from anything appearing in \cite{dolgopyat2007simultaneous}. Consequently, the reader may wonder why it is needed. Section \ref{sec:diffs_of_small_strain} provides a method of finding a tuple of isometries $(R_1',\ldots,R_m')$ near to the tuple $(\wt{f}_1,\ldots,\wt{f}_m)$ of diffeomorphisms. In \cite{dolgopyat2007simultaneous}, the new diffeomorphisms $R_m$ are found in the following manner. As in equation \eqref{eqn:linearized_error}, one may find vector fields $Y_i$ such that 
\[
\exp_{R_i(x)}Y_i(x)=f_i(x).
\]
If $Z$ is a vector field on $M$, we define $\psi_Z$, as in equation \eqref{eq:psi_defn} to be the map $x\mapsto \exp_xZ(x)$. 
There is a certain operator, the Casimir Laplacian, which acts on vector fields. This operator is defined and discussed in more detail in subsection \ref{subsec:diophantine_sets}. Dolgopyat and Krikorian then project the vector fields $Y_i$ onto the kernel of the Casimir Laplacian, to obtain a vector field $Y_i'$. They then define $R_i'$ to equal $\psi_{Y_i'}\circ R_i$.  This happens in the line immediately below equation (19) in \cite{dolgopyat2007simultaneous}. 

One difficulty is establishing that the maps $(R_1',\ldots,R_m')$ are close to the $(\wt{f_i},\ldots,\wt{f_m})$. 
The argument for their nearness hinges on part (d) of Proposition 3 in \cite{dolgopyat2007simultaneous},  which essentially says that, up to a third order error, the magnitude of the smallest Lyapunov exponent is a bound on the distance. As written, the argument in \cite{dolgopyat2007simultaneous} suggests that part (d) is an easy consequence of part (c) of \cite[Prop. 3]{dolgopyat2007simultaneous}. However, part (d) does not follow. Here is a simplification of the problem. Suppose that $f\colon \R^n\to \R^n$ is a diffeomorphism. Pick a point $x\in \R^n$ and write $D_xf=A+B+C$, where $A$ is a multiple of the identity, $B$ is symmetric with trace zero, and $C$ is skew-symmetric. The results in part (c) imply that $A$ and $B$ are small, but they offer no information about $C$.\footnote{For those comparing with the original paper, $A$ and $B$ correspond to the terms $q_1$ and $q_2$, respectively, which appear in part (c) of \cite[Prop. 3]{dolgopyat2007simultaneous}.} Concluding that the norm of $Df$ is small requires that $C$ be small as well. As $C$ is skew-symmetric it is natural to think of it as the germ of an isometry. Our modification to the argument is designed to accommodate the term $C$ by recognizing it as the ``isometric" part of the differential. Pursuing this perspective leads to the strain tensor and our Proposition \ref{lem:H_0_nearby_isometry}.  Conversation with Dmitry Dolgopyat confirmed that there is a problem in the paper on this point and that part (d) of Proposition 3 does not follow from part (c).

\section{A Diophantine Property and Spectral Gap}\label{sec:diophantine}

Fix a compact connected semisimple Lie group $G$ and let $\mf{g}$ denote its Lie algebra. Endow $G$ with the bi-invariant metric arising from the negative of the Killing form on $\mf{g}$. We denote this metric on $G$ by $d$. We endow a subgroup $H$ of $G$ with the pullback of the Riemannian metric from $G$ and denote the distance on $H$ with respect to the pullback metric by $d_H$. We use the manifold topology on $G$ unless explicitly stated otherwise. Consequently, whenever we say that a subset of $G$ is dense, we mean this with respect to the manifold topology on $G$. We say that a subset $S$ of $G$ \emph{generates} $G$ if the smallest closed subgroup of $G$ containing $S$ is $G$. In other words, if $\langle S\rangle$ denotes the smallest subgroup of $G$ containing $S$, then $S$ generates if $\overline{\langle S\rangle}=G$.

Suppose that $S\subset G$ generates $G$. We begin this section by discussing how long a word in the elements of $S$ is needed to approximate an element of $G$. Then using this approximation we obtain quantitative estimates for the spectral gap of certain operators associated to $S$. Finally, those spectral gap estimates allow us to obtain a ``tameness" estimate for a particular operator that arises from $S$. This final estimate, Proposition \ref{prop:tame_coboundaries}, will be crucial in the KAM scheme that we use to prove Theorem \ref{thm:KAM_converges}.

The content of this section is broadly analogous to Appendix A in \cite{dolgopyat2002mixing}. However, our development follows a different approach and in some places we are able to obtain stronger estimates.

\subsection{The Solovay-Kitaev algorithm}

Suppose that $S$ is a subset of $G$. We say that $S$ is \emph{symmetric} if $s\in S$ implies $s^{-1}\in S$. For a natural number $n$, let $S^n$ denote the $n$-fold product of $S$ with itself. Let $S^{-1}$ be $\{s^{-1}: s\in S\}$. For $n<0$, define $S^n$ to equal $(S^{-1})^{-n}$. 
 The following theorem says that any sufficiently dense symmetric subset $S$ of a compact semisimple Lie group is a generating set. More importantly, it also gives an estimate on how long a word in the generating set $S$ is needed to approximate an element of $G$ to within error $\epsilon$. If $w=s_1\cdots s_n$ is a word in the elements of the set $S$, then we say that $w$ is \emph{balanced} if for each $s\in S$, $s$ appears the same number of times in $w$ as $s^{-1}$ does.

\begin{thm}\label{thm:real_solovay_kitaev}
\cite[Thm. 1]{dawson2006solovay}(Solovay-Kitaev Algorithm)
Suppose that $G$ is a compact semisimple Lie group. There exists $\epsilon_0(G)>0$ and $\alpha>0$ and $C>0$ such that if $S$ is any symmetric $\epsilon_0$-dense subset of $G$ then the following holds. For any $g\in G$ and any $\epsilon>0$, there exists a natural number $l_{\epsilon}$ such that $d(g,S^{l_{\epsilon}})<\epsilon$. Moreover, $l_{\epsilon}\le C\log^\alpha (1/\epsilon)$. Further, there is a balanced word of length $l_{\epsilon}$ within distance $\epsilon$ of $g$.
\end{thm}

Later, we use a version of this result that does not require that the set $S$ be symmetric. Using a non-symmetric generating set  significantly increases the word length obtained in the conclusion of the theorem. It is unknown if there exists a version of the Solovay-Kitaev algorithm that does not require a symmetric generating set and keeps the $O(\log^\alpha(1/\epsilon))$ word length. See \cite{bouland2018trading} for a partial result in this direction.

\begin{prop}\label{prop:solovay-kitaev_without_inverses}
Suppose that $G$ is a compact semisimple Lie group endowed with a bi-invariant metric. There exists $\epsilon_0(G)>0$, $\alpha>0$, and $C\ge 0$  such that if $S$ is any $\epsilon_0$-dense subset of $G$ then the following holds. For any $g\in G$ and any $\epsilon>0$, there exists a natural number $l_{\epsilon}$ such that $d(g,S^{l_{\epsilon}})<\epsilon$. Moreover, $l_{\epsilon}\le C\epsilon^{-\alpha}$.
\end{prop}

Our weakened version of the Solovay-Kitaev algorithm relies on the following lemma, which allows us to approximate the inverse of an element $h$ by some positive power of $h$.

\begin{lem}\label{lem:approximate_inverse}
Suppose that $G$ is a compact $d$-dimensional Lie group with a fixed bi-invariant metric. Then there exists a constant $C$ such that for all $\epsilon>0$ and any $h\in G$ there exists a natural number $n<C/\epsilon^d$ such that $d(h^{-1},h^n)<\epsilon$.
\end{lem}

\begin{proof}
This follows from a straightforward pigeonhole argument. We cover $G$ with sets of diameter $\epsilon$. There exists a constant $C$ so that we can cover $G$ with at most $C\vol(G)/\epsilon^d$ such sets, where $d$ is the dimension of $G$.  Consider now the first $\lceil C\vol(G)/\epsilon^d\rceil $ iterates of $h^2$.  By the pigeonhole principle, two of these must fall into the same set in the covering, and so there exist natural numbers $n_i$ and $n_j$ such that  $0<n_i<n_j<\lceil C\vol(G)/(\epsilon^d)\rceil$ and $h^{2n_i}$ and $h^{2n_j}$ lie in the same set in the covering. Thus $d(h^{2n_i},h^{2n_j})<\epsilon$.  As $h$ is an isometry it follows that $d(e,h^{2n_j-2n_i})<\epsilon$ and hence $d(h^{-1},h^{2n_j-2n_i-1})<\epsilon$ as well. This finishes the proof.
\end{proof}

We now prove the proposition.

\begin{proof}[Proof of Proposition \ref{prop:solovay-kitaev_without_inverses}]
Let $\hat{S}=S\cup S^{-1}$. Note that as $\hat{S}$ is a symmetric generating set of $G$ that by Theorem \ref{thm:real_solovay_kitaev} for any $\epsilon>0$, there exists a number $l_{\epsilon/2}=O(\log^\alpha(1/\epsilon))$ such that for any $g\in G$ there exists an element $h$ in $\hat{S}^{l_{\epsilon/2}}$ such that $d(h,g)<\epsilon/2$. Further, by the statement of Theorem \ref{thm:real_solovay_kitaev}, we know that $h$ is represented by a balanced word $w$ in $\hat{S}^{l_{\epsilon/2}}$. 

To finish the proof, we replace each element of $w$ that is in $S^{-1}$ by a word in $S^j$ for some uniform $j>0$. To do this we show that there exists a fixed $j$ so that the elements of $S^j$ approximate well the inverses of the elements of $S$. 
Write $S=\{s_1,\cdots,s_m\}$ and consider the element $(s_1,...,s_m)$ in the group $G\times \cdots\times G$, where there are $m$ terms in the product. By applying Lemma \ref{lem:approximate_inverse} to the group $G\times \cdots \times G$ and the element $(s_1,...,s_m)$, we obtain that there exists a uniform constant $C'$ and $j<C'2^{dm}l_{\epsilon/2}^{dm}/\epsilon^{dm}$ such that any $s\in S^{-1}$ may be approximated to distance $\epsilon/(2l_{\epsilon/2})$ by an element in $S^j$.

We now replace each element of $S^{-1}$ appearing in $w$ with a word in $S^j$ that is at distance $\epsilon/(2l_{\epsilon/2})$ away from it. Call this new word $w'$. Because $w$ is balanced, we replace exactly half of the terms in $w$. Thus $w'$ is a word of length $jl_{\epsilon/2}/2+l_{\epsilon/2}/2$ as we have replaced half the entries of $w$, which has length $l_{\epsilon/2}$, with words of length $j$. Let $h'$ be the element of $G$ obtained by multiplying together the terms in $w'$.

Note that multiplication of any number of elements of $G$ is $1$-Lipschitz in each argument. Hence as we have modified the expression for $h$ in exactly $l_{\epsilon/2}/2$ terms and each modification is of size $\epsilon /(2l_{\epsilon/2})$, $h'$ is distance at most $\epsilon/2$ from $h$ and hence at most distance $\epsilon$ from $g$. Thus $S^{jl_{\epsilon/2}/2+l_{\epsilon/2}/2}$ is  $\epsilon$ dense in $G$ and 
\[
jl_{\epsilon/2}/2+l_{\epsilon/2}/2<C''l_{\epsilon/2}^{dm+1}/\epsilon^{dm}=O(\log^{(dm+1)\alpha}(1/\epsilon)\epsilon^{-dm}), 
\]
which establishes the proposition as $m$ depends only on $\abs{S}$.
\end{proof}

We record one final result that asserts that if $S\subseteq G$ generates, then the powers of $S$ individually become dense in $G$.

\begin{prop}\label{prop:generating_set_gets_dense}
Suppose that $G$ is a compact connected Lie group. Suppose that $S\subseteq G$ generates $G$. Then for all $\epsilon>0$ there exists a natural number $n_{\epsilon}$ such that $S^{n_{\epsilon}}$ is $\epsilon$-dense in $G$.
\end{prop}

\begin{proof}
 Let $\{g_1,...,g_m\}$ be an $\epsilon/2$-dense subset of $G$. Because $S$ generates, for each $g_i$ there exists $n_i$ and $w_i\in S^{n_i}$ such that $d(g_i,w_i)<\epsilon/2$.
By a pigeonhole argument similar to the proof of Lemma \ref{lem:approximate_inverse}, it holds that for all $\epsilon>0$ there exists a natural number $N$ such that for all $n\ge N$, $d(S^n,e)<\epsilon$.  Thus there exists $N$ such that for all $n\ge N$, $S^n$ contains elements within distance $\epsilon/2$ of the identity. Thus $S^{N+\max_i\{n_i\}}$ is $\epsilon$-dense in $G$.
\end{proof}

\subsection{Diophantine Sets}\label{subsec:diophantine_sets}

We will now introduce a notion of a Diophantine subset of a compact connected semisimple Lie group $G$.  Write $\mf{g}$ for the Lie algebra of $G$. We recall the definition of the standard quadratic Casimir inside of $U(\mf{g})$, the universal enveloping algebra of $\mf{g}$. Write $B$ for the Killing form on $\mf{g}$ and let $X_i$ be an orthonormal basis for $\mf{g}$ with respect to $B$. We will also denote the inner product arising from the Killing form by $\langle \cdot,\cdot\rangle$. Then the \emph{Casimir}, $\Omega$, is the element of $U(\mf{g})$ defined by
\[
\Omega=\sum_{i} X_i^2.
\]
The element $\Omega$ is well-defined and central in $U(\mf{g})$. Elements of $U(\mf{g})$ act on the smooth vectors of representations of $G$. Consequently, as $\Omega$ is central and every vector in an irreducible representation $(\pi,V)$ is smooth, $\pi(\Omega)$ acts by a multiple of the identity. Given an irreducible unitary representation $(\pi,V)$, Define $c(\pi)$ by 
\begin{equation}\label{eqn:defn_of_cpi}
c(\pi)\Id=-\pi(\Omega).
\end{equation}
The quantity $c(\pi)$ is positive in non-trivial representations. Further, as $\pi$ ranges over all non-trivial representations, $c(\pi)$ is uniformly bounded away from $0$. For further information see \cite[{}5.6]{wallach2018harmonic}.

\begin{defn}\label{defn:defn_of_diophantine}
Let $G$ be a compact, connected, semisimple Lie group. We say that a subset $S\subset G$ is $(C,\alpha)$-\emph{Diophantine} if the following holds for each non-trivial, irreducible, finite dimensional unitary representation $(\pi, V)$ of $G$. For all non-zero $v\in V$ there exists $g\in S$ such that
\[
\|v-\pi(g)v\|\ge Cc(\pi)^{-\alpha}\|v\|,
\]
where $c(\pi)$ is defined in \eqref{eqn:defn_of_cpi}. We say that $S$ is \emph{Diophantine} if $S$ is $(C,\alpha)$-Diophantine for some $C,\alpha>0$. If $(g_1,...,g_m)$ is a tuple of elements of $G$, the we say that this tuple is $(C,\alpha)$-Diophantine if the underlying set is $(C,\alpha)$-Diophantine.
\end{defn}

Our formulation of Diophantine is slightly different from the definition in \cite{dolgopyat2002mixing} as we refer directly to irreducible representations. We choose this formulation because it allows for a unified analysis of the action of $\Omega$ in diverse representations of $G$.

It is useful to compare Definition \ref{defn:defn_of_diophantine} with the simultaneous Diophantine condition used when studying translations on tori, such as is considered in \cite{damjanovic2019local} or \cite{petkovic2021classification}. The condition for tori is a generalization of the simultaneous Diophantine condition considered considered by Moser \cite{moser1990commuting} for circle diffeomorphisms.  Denote by $\langle \cdot,\cdot\rangle$ denote the standard inner product in $\R^d$. A tuple of vectors $(\theta_1,\ldots,\theta_m)$ in $\R^d$ defines a tuple of translations of $\mathbb{T}^d$. We say that this tuple is $(C,\alpha)$-Diophantine if for every non-zero $k\in \Z^d$,
\begin{equation}\label{eqn:diophantine_torus_defn}
\max_{1\le i\le m }\min_{l\in \Z}\abs{\langle \theta_i,k\rangle-l} \ge \frac{C}{\|k\|^{\alpha}}.
\end{equation}
One can see the relationship between this definition and the one for compact semisimple groups when we think of  $\Z^d$ as indexing the unitary representations of $\mathbb{T}^d$. Although these definitions apply to different types of groups, one can check that the estimates at their core are equivalent: for a given unitary representation defined by $k\in \Z^d$,  use the $\theta_i$ that achieves the maximum in \eqref{eqn:diophantine_torus_defn}  to act on the representation defined by $k$.

We now give a useful characterization of Diophantine subsets of compact semisimple groups.
\begin{prop}\label{prop:diophantine_equivalence}
\cite[Thm. A.3]{dolgopyat2002mixing} Suppose that $S$ is a finite subset of a compact connected semisimple Lie group $G$. Then $S$ is Diophantine if and only if $\overline{\langle S\rangle}=G$. Moreover, there exists $\epsilon_0(G)$ such that any $\epsilon_0$-dense subset of $G$ is Diophantine.
\end{prop}

Before proceeding to the proof we will show two preliminary results.

\begin{lem}\label{lem:lipschitz_rep}
Suppose that $G$ is a compact connected semisimple Lie group. Suppose that $(\pi,V)$ is an irreducible unitary representation of $G$. Then for any $v\in V$ of unit length, any $X\in \mf{g}$ of unit length, and $t\ge 0$,
\[
\|\pi(\exp(tX))v-v\|\le t\sqrt{c(\pi)}.
\]
\end{lem}
\begin{proof}
A similar argument to the following appears in \cite[{}5.7.13]{wallach2018harmonic}.
	There exists an orthonormal basis $\{X_1,...,X_n\}$ of $\mf{g}$ such that $X_1=X$. Observe that
\[
\pi(\exp(tX))v-v=td\pi(X)v+O(t^2).
\]
	The transformation $d\pi(X)$ is skew symmetric with respect to the inner product. Thus $d\pi(X)^2$ is positive semidefinite. Consequently:
\[
\langle d\pi(X)v,d\pi(X)v\rangle=-\langle d\pi(X)^2v,v\rangle \le -\langle \pi(\Omega) v,v\rangle=c(\pi)\|v\|^2.
\]
 Hence
\[
\|\pi(\exp(tX)v)-v\|\le t\sqrt{c(\pi)}+O(t^2).
\]
For $0\le i\le n$, let $t_i=\frac{i}{n}t$. Then
\begin{align*}
\|\pi(\exp(tX))v-v\|&\le \sum_{i=1}^n \|\pi(\exp(t_{i}X))v-\pi(\exp(t_{i-1}X))v\|\\
&\le \sum_{i=1}^n \|\pi(\exp(tX/n))v-v\|\\
&\le n\pez{\frac{t}{n}\sqrt{c(\pi)}+O((t/n)^2)}.
\end{align*}
Taking the liminf of the right hand side as $n\to \infty$ gives the result.
\end{proof}

The following lemma will be of use in the proof of Proposition \ref{prop:dense_implies_diophantine}.
\begin{lem}\label{lem:irrep_orthogonal}
Suppose that $(\pi,V)$ is a non-trivial, irreducible, finite dimensional, unitary representation of a compact, connected, semisimple group $G$. Then for any $v\in V$, there exists $g$ such that $\langle \pi(g)v,v\rangle =0$.
\end{lem}

\begin{proof}
If such a $g$ does not exist, then for all $g\in G$, $\pi(g)v$ lies in the same half-space as $v$. But then $\int_G\pi(g)v\,dg\neq 0$ and is a $G$ invariant vector, which contradicts the irreducibility of $\pi$.
\end{proof}

\begin{prop}\label{prop:dense_implies_diophantine}
Suppose that $G$ is a compact connected semisimple Lie group. Then there exist $\epsilon_0,C,\alpha>0$ such that any $\epsilon_0$-dense subset of $G$ is $(C,\alpha)$-Diophantine. If $S$ is a subset of $G$ such that $S^{n_0}$ is $\epsilon_0$-dense in $G$, then $S$ is $(C/n_0,\alpha)$ Diophantine.
\end{prop}

\begin{proof}
Let $\epsilon_0$ equal the $\epsilon_0(G)$ in Theorem \ref{thm:real_solovay_kitaev}, the Solovay-Kitaev algorithm. In the case that $S$ is already $\epsilon_0$-dense, let $n_0=1$. By Theorem \ref{thm:real_solovay_kitaev}, there exist $C$ and $\alpha$ such that for each $\epsilon$ there exists $l_{\epsilon}\le C\log^\alpha(\epsilon^{-1})$ such that $S^{n_0l_{\epsilon}}$ is $\epsilon$-dense in $G$. Suppose that $(\pi, V)$ is a non-trivial irreducible unitary representation of $G$ and suppose that $v\in V$ is a unit vector. By Lemma \ref{lem:irrep_orthogonal} there exists $g\in G$ such that $\langle \pi(g)v,v\rangle=0$.  
Now fix $\epsilon=1/(100\sqrt{c(\pi)})$. Then there exists an element $w\in S^{n_0l_{\epsilon}}$ such that $d(g,w)<\epsilon$. Thus by Lemma \ref{lem:lipschitz_rep},
\[
\|\pi(g)v-\pi(w)v\|\le \epsilon \sqrt{c(\pi)}<\frac{1}{100}.
\]
By the triangle inequality, this implies that 
\[
\|\pi(w)v -v\|\ge 1.
\]
Write $w=g_1^{\sigma_1}\cdots g_{n_0l_{\epsilon}}^{\sigma_{n_0l_{\epsilon}}}$ where each $\sigma_i\in \{\pm 1\}$ and each $g_i\in S$. Let $w_i=g_1^{\sigma_1}\cdots g_i^{\sigma_i}$. 
Let $w_0=e$. By applying the triangle inequality $n_0l_{\epsilon}$ times, we see that 
\[
\sum_{i=0}^{n_0l_{\epsilon}-1} \|\pi(w_i)v-\pi(w_{i+1})v\|\ge\|v-\pi(w)v\| \ge 1.
\]
Thus there exists some $i$ such that 
\[
\|\pi(w_i)v-\pi(w_{i+1})v\|\ge \frac{1}{n_0l_{\epsilon}}.
\]
Applying $\pi(w_i^{-1})$ and noting by our choice of $\epsilon$ that $l_{\epsilon}\le C\log^\alpha(c(\pi))$, we obtain that
\begin{equation}\label{eqn:stronger_diophantineness_estimate}
\|v-\pi(g_i^{\sigma_i})v\|\ge \frac{1}{n_0C'\log^\alpha(c(\pi))}.
\end{equation}
Thus we are done as we have obtained an estimate that is stronger than the required lower bound of $C/c(\pi)^{\alpha}$.
\end{proof}

We now prove the equivalence of the Diophantine property appearing in Proposition \ref{prop:dense_implies_diophantine} with that in Definition \ref{defn:defn_of_diophantine}.

\begin{proof}[Proof of Proposition \ref{prop:diophantine_equivalence}.]
To begin, suppose that $S$ is Diophantine. For the sake of contradiction, suppose that $H\coloneqq\overline{\langle S\rangle}\neq G$. Consider the action of $G$ on $L^2(G/H)$ by left translation. Note that $H$ acts trivially. However, $L^2(G/H)$ contains non-trivial representations of $G$. Thus $S\subset H$ cannot be Diophantine, which is a contradiction.

For the other direction, suppose that $\overline{\langle S\rangle}=G$. Then by Proposition \ref{prop:generating_set_gets_dense} there exists $n$ such that $S^n$ is $\epsilon_0(G)$-dense and, hence $S$ is Diophantine by Proposition \ref{prop:dense_implies_diophantine}. 
\end{proof}

The stronger bound in equation \eqref{eqn:stronger_diophantineness_estimate} gives an equivalent characterization of Diophantineness.

\begin{cor}\label{cor:diophantine_characterization}
Let $G$ be a compact, connected, semisimple Lie group. A subset $S$ of $G$ is Diophantine if and only if there exist $C,\alpha>0$ such that the following holds for each non-trivial, irreducible, finite dimensional, unitary representation $(\pi, V)$ of $G$. For all $v\in V$  there exists $g\in S$ such that 
\[
\|v-\pi(g)v\|\ge \frac{\|v\|}{C\log^\alpha(c(\pi))}.
\]
\end{cor}

Diophantine subsets of a group are typical in the following sense.

\begin{prop}
Suppose that $G$ is a compact connected semisimple Lie group. Let $U\subset G\times G$ be the set of ordered pairs $(u_1,u_2)$ such that $\{u_1,u_2\}$ is a Diophantine subset of $G$. Then $U$ is Zariski open and hence open and dense in the manifold topology on $G\times G$.
\end{prop}

\begin{proof}
Let $U\subset G\times G$ be the set of points $(u_1,u_2)$ such that $\{u_1,u_2\}$ generates a dense subset of $G$. Theorem 1.1 in \cite{field1999generating} gives that $U$ is Zariski open and non-empty. By Proposition \ref{prop:diophantine_equivalence}, this implies that $\{u_1,u_2\}$ is Diophantine. As $U$ is non-empty, the final claim follows.
\end{proof}

\subsection{Polylogarithmic spectral gap}

In this subsection, we study spectral properties of an averaging operator associated to a tuple of elements of $G$. Consider a tuple $(g_1,...,g_m)$ of elements of $G$. Let $\R[G]$ denote the group ring of $G$ over $\R$. From this tuple we form $\mc{L}\coloneqq (g_1+\cdots+g_m)/m\in \R[G]$. The element $\mc{L}$ acts in representations of $G$ in the natural way. If $(\pi,V)$ is a representation of $G$, then we write $\mc{L}_{\pi}$ for the action of $\mc{L}$ on $V$. The main result of this subsection is the following proposition, which gives some spectral properties of $\mc{L}_{\pi}$ under the assumption that $\{g_1,...,g_m\}$ is Diophantine.

\begin{prop}\label{prop:decay_estimate}
Suppose that $G$ is a compact connected semisimple Lie group, $(g_1,...,g_m)$ is a tuple of elements of $G$, and that $\{g_1,...,g_m\}$ generates $G$. Then there exists a neighborhood $N$ of $(g_1,...,g_m)$ in $G\times \cdots \times G$ and constants $D_1,D_2,\alpha>0$ such that if $(g_1',....,g_m')\in N$, then $\{g_1',\ldots,g_m'\}$ is Diophantine and its associated averaging operator $\mc{L}$ satisfies  
\[
\|\mc{L}_\pi^n\|\le D_1\pez{1-\frac{1}{D_2\log^{\alpha}(c(\pi))}}^n,
\]
for each non-trivial irreducible unitary representation $(\pi,V)$.
\end{prop}

The proof of Proposition \ref{prop:decay_estimate} uses the following lemma, which is a sharpening the triangle inequality for vectors that are not colinear.
\begin{lem}\label{lem:sharpened_triangle_inequality}
Suppose that $v,w$ are two vectors in an inner product space. Suppose that $\|v\|\le \|w\|$ and let $\hat{v}=v/\|v\|$ and $\hat{w}=w/\|w\|$. If
\[
\|\hat{v}-\hat{w}\|\ge \epsilon,
\]
then 
\[
\|v+w\|\le (1-\epsilon^2/10)\|v\|+\|w\|.
\]
\end{lem}

\begin{proof}
We begin by considering the following estimate for unit vectors.

\begin{claim} Suppose that the angle between two unit vectors $\hat{v}$ and $\hat{w}$ is $\theta\in [0,\pi]$, then 
\[
\|\hat{v}+{w}\|\le \|\hat{v}\|+(1-\theta^2/10)\|\hat{w}\|.
\]
\end{claim}
\begin{proof}
It suffices to consider the two vectors $\hat{v}=(1,0)$ and $\hat{w}=(\cos \theta,\sin \theta)$ in $\R^2$. It then suffices to show:
\[
\|\hat{v}+\hat{w}\|^2\le \left(\|\hat{v}\|+\left(1-\frac{\theta^2}{10}\right)\|\hat{w}\|\right)^2.
\]
From the definitions,
\[
\|\hat{v}+\hat{w}\|^2=2+2\cos\theta
\]
and
\[
\left(\|\hat{v}\|+\left(1-\frac{\theta^2}{10}\right)\|\hat{w}\|\right)^2 = 4-4\frac{\theta^2}{10}+\frac{\theta^4}{100}\ge 4-4\frac{\theta^2}{10}.
\]
Thus it suffices to show for $\theta\in [0,\pi]$ that
\[
2+2\cos \theta\le 4-4\frac{\theta^2}{10},
\]
which follows because for $\theta\in [0,\pi]$ we have the estimate $\cos\theta\le 1-\theta^2/5$.
\end{proof}

We may prove the lemma once we have one more observation. Note that if $\hat{v}$ and $\hat{w}$ are two unit vectors, then $\|\hat{v}-\hat{w}\|=\epsilon$ is less than the angle $\theta$ between $\hat{v}$ and $\hat{w}$ because the distance between $\hat{v}$ and $\hat{w}$ along a unit circle they lie on is precisely $\theta$. Thus we see that $\epsilon\le \theta$ for $0\le \theta\le \pi$.

We now compute. Note that without loss of generality we may assume that $\|w\|=1$, which we do in the following. By the triangle inequality,
\begin{align*}
\|v+w\|\le \|v\|\|\hat{v}+\hat{w}\|+(1-\|v\|)\|\hat{w}\|.
\end{align*}
By the claim it then follows that
\[
\|v+w\|\le \|v\|((1-\theta^2)\|\hat{v}\|+\|\hat{w}\|)+(1-\|v\|)\|\hat{w}\|.
\]
Noting from before that $0\le \epsilon \le \theta$ for $\theta\in [0,\pi]$, we then conclude:
\[
\|v+w\|\le \|v\|((1-\epsilon^2/10)\|\hat{v}\|+\|\hat{w}\|)+(1-\|v\|)\|\hat{w}\|=(1-\epsilon^2/10)\|v\|+\|w\|.
\]
\end{proof}

\begin{proof}[Proof of Proposition \ref{prop:decay_estimate}.]
For convenience, let $W=(g_1,...,g_m)$ and let $S=\{g_1,...,g_m\}$. Let $\epsilon_0(G)$ be as in Proposition \ref{prop:dense_implies_diophantine}. By Proposition \ref{prop:generating_set_gets_dense}, because $\overline{\langle S\rangle}=G$ there exists some $n_0$ such that $S^{n_0}$ is $\epsilon_0/2$-dense in $G$. Then let $N$ be the neighborhood of $(g_1,...,g_m)$ in $G\times\cdots \times G$ such that if $p=(g_1',...,g_m')\in N$ then $\{g_1',...,g_m'\}^{n_0}$ is at least $\epsilon_0$-dense in $G$. It now suffices to obtain the given estimate for the set $W=(g_1,...,g_m)$ using only the assumption that $S^{n_0}$ is $\epsilon_0$-dense. Below, $W^{n_0}$ is the tuple of the $m^{n_0}$ words of length $n_0$ with entries in $W$.

 By Proposition \ref{prop:dense_implies_diophantine}, there exist $(C,\alpha)$ such that any $\epsilon_0$-dense set is $(C,\alpha)$-Diophantine. As $S^{n_0}$ is $\epsilon_0$-dense, so is $S^{n_0}S^{-n_0}$, and hence  $S^{n_0}S^{-n_0}$ is $(C,\alpha)$-Diophantine. 

Consider now a non-trivial irreducible finite dimensional unitary representation $(\pi, V)$ of $G$.
Since $S^{n_0}S^{-n_0}$ is $(C,\alpha)$-Diophantine, Corollary \ref{cor:diophantine_characterization} implies that for any unit length $v\in V$ there exist $w_1,w_2\in S^{n_0}$ such that 
\[
\|v-\pi(w_1^{-1}w_2)v\|\ge \frac{1}{C\log^\alpha(c(\pi))},
\]
and so
\[
\|\pi(w_1)v-\pi(w_2)v\|\ge \frac{1}{C\log^\alpha(c(\pi))}.
\]
Hence by Lemma \ref{lem:sharpened_triangle_inequality}, since $\pi$ is unitary
\[
\|\pi(w_1)v+\pi(w_2)v\|\le \pez{1-\frac{1}{10C^2\log^{2\alpha}(c(\pi))}}\|\pi(w_1)v\| +\|\pi(w_2)v\|\le \pez{2-\frac{1}{10C^2\log^{2\alpha}(c(\pi))}}\|v\|.
\]
Then by the triangle inequality:
\begin{align*}
\|\mc{L}_{\pi}^{n_0}v\|&=\left\|\frac{1}{\abs{W}^{n_0}}\sum_{w\in W^{n_0}} \pi(w)v\right\|\\
&\le \frac{1}{\abs{W}^{n_0}}\left(\|\pi(w_1)v+\pi(w_2)v\|+\sum_{w\in W^{n_0}\setminus \{w_1,w_2\}}\|\pi(w)v\|\right)\\
&\le \frac{1}{\abs{W}^{n_0}}\pez{2-\frac{1}{10C^2\log^{2\alpha}(c(\pi))}}\|v\| +\frac{\abs{W}^{n_0}-2}{\abs{W}^{n_0}}\|v\|\\
&\le \pez{1-\frac{1}{10C^2|W^{n_0}|\log^{2\alpha}(c(\pi))}}\|v\|.
\end{align*}

Interpolating gives that for all $n\ge 0$,
\[
\|\mc{L}_{\pi}^n\|\le \pez{1-\frac{1}{10C^2|W^{n_0}|\log^{2\alpha}(c(\pi))}}^{-1}\pez{1-\frac{1}{10C^2|W^{n_0}|\log^{2\alpha}(c(\pi))}}^{n/n_0}.
\]
As $(\pi,V)$ ranges over all non-trivial representations, $c(\pi)$ is uniformly bounded away from $0$; see \cite[{}5.6.7]{wallach2018harmonic}. This implies that the first term above is uniformly bounded by some $D>0$ independent of $\pi$. Applying the estimate $(1+x)^{\epsilon}\le 1+\epsilon x$ to the second term then gives the proposition.
\end{proof}

Notice that in Proposition \ref{prop:decay_estimate} that we obtain an entire neighborhood of our initial set $S$ on which we have the same estimates for $\mc{L}_{\pi}$. Consequently, because these estimates remain true under small perturbations, we think of them as being stable. We will use the term ``stable" in the following precise sense.

\begin{defn}\label{defn:stable}
Suppose that $T$ is some property of a tuple $W=(g_1,...,g_m)$ with elements in a Lie group $G$.
We say that $T$ is \emph{stable} at $W=(g_1,...,g_m)$ if there exists a neighborhood $N$ of $(g_1,...,g_m)$ in $G\times \cdots \times G$ such that if $(g_1',...,g_m')\in N$ then $T$ holds for $(g_1',...,g_m')$. We will also say that $T$ is \emph{stable} without reference to a subset when the relevant tuples that $T$ is stable on are evident.
\end{defn}

A crucial aspect of the Diophantine property in compact semisimple Lie groups is that by Proposition \ref{prop:dense_implies_diophantine} there is a stable lower bound on $(C,\alpha)$. This stability will be essential during the KAM scheme.

\subsection{Diophantine sets and tameness}\label{subsec:diophantine_sets_tameness}

Consider a smooth vector bundle $E$ over a closed manifold $M$. We may consider the space $C^{\infty}(M,E)$ of smooth sections of $E$. Consider a linear map $L\colon C^{\infty}(M,E)\to C^{\infty}(M,E)$. We say that $L$ is \emph{tame} if there exists $\alpha$ such that for all $k$ there exists $C_k$, such that for all $s\in C^{\infty}(M,E)$,
\[
\|Ls\|_{C^k}\le C_k\|s\|_{C^{k+\alpha}}.
\]
See \cite[{}II.2.1]{hamilton1982inverse} for more about tameness.
The main result of this section is to show such estimates for certain operators related to $\mc{L}$.

Though $\mc{L}$ acts in any representation of $G$, we are most interested in the action of $G$ on the sections of certain vector bundles, which we now describe. Suppose that $K$ is a closed subgroup of $G$ and that $E$ is a smooth vector bundle over $G/K$.
 We say that $E$ is a \emph{homogeneous vector bundle} over $G/K$ if $G$ acts on $E$ by bundle maps and this action projects to the action of $G$ on $G/K$ by left translation. We now give an explicit description of all homogeneous vector bundles over $G/K$ via the Borel construction.  See \cite[Ch. 5]{wallach2018harmonic} for more details about this topic and what follows. 
 Suppose that $(\tau, E_0)$ is a finite dimensional unitary representation of $K$. Form the trivial bundle $G\times E_0$. Then $K$ acts on this bundle by $(g,v)\mapsto (gk,\tau(k)^{-1}v)$. Then $(G\times E_0)/K$ is a vector bundle over $G/K$ that we denote by $G\times_{\tau} E_0$.  Note, for instance, that $C^{\infty}(G,\R)$ is the space of sections of the homogeneous vector bundle obtained from the trivial representation of $\{e\}<G$. The left action of $G$ on $G\times E_0$ descends to $G\times_{\tau} E_0$, and hence this is a homogeneous vector bundle.

In order to do analysis in a homogeneous vector bundle, we must introduce some additional structures. Suppose that $E=G\times_{\tau}E_0$ is a homogeneous vector bundle. The base $G/K$ comes equipped with the projection of the Haar measure on $G$. As the action of $K$ on $G\times E_0$ is isometric on fibers, the fibers of $E$ are naturally endowed with an inner product. We may then consider the space $L^2(E)$, the space of all $L^2$ sections of $E$. In addition, we will write $C^{\infty}(E)$ for the space of all smooth sections of $E$. The action of $G$ on $E$ preserves $L^2(E)$ and $C^{\infty}(E)$. 

We recall briefly how one may do harmonic analysis on sections of such bundles.
As before, let $\Omega$ be the Casimir operator, which is an element of $U(\mf{g})$. Then $\Omega$ acts on the $C^{\infty}$ vectors of any representation of $G$. Denote by $\Delta$ the differential operator obtained by the action of $-\Omega$ on $C^{\infty}(E)$. Then $\Delta$ is a hypoelliptic differential operator on $E$. We then use the spectrum of $\Delta$ to define for any $s\ge 0$ the Sobolev norm $H^s$ in the following manner. $L^2(E)$ may be decomposed as the Hilbert space direct sum of finite dimensional irreducible unitary representations $V_{\pi}$.  Write $\phi=\sum_{\pi} \phi_\pi$ for the decomposition of an element $\phi\in L^2(E)$. Then the $s$-Sobolev norm is defined by 
\[
\|\phi\|_{H^s}^2=\sum_{\pi} (1+c(\pi))^s \|\phi_{\pi}\|_{L^2}^2.
\]
We write $\|f\|_{C^s}$ for the usual $C^s$ norm of a function or section of a vector bundle.  It is not always necessary to work with the decomposition of $L^2(E)$ into irreducible subspaces, but instead use a coarser decomposition as follows. We let $H_{\lambda}$ denote the subspace of $L^2(E)$ on which $\Delta$ acts by multiplication by $\lambda>0$. There are countably many such subspaces $H_{\lambda}$ and each is finite dimensional. In the sequel, those functions that are orthogonal to the trivial representations in $L^2(E)$ will be of particular importance. We denote by $L^2_0(E)$ the orthogonal complement of the trivial representations in $L^2(E)$, and $C^{\infty}_0(E)$ the subspace $L^2_0(E)\cap C^{\infty}(E)$.

We now consider the action of $\mc{L}$ on the sections of a homogeneous vector bundle.
\begin{prop}\label{prop:tame_coboundaries}
\cite[Prop 1.]{dolgopyat2007simultaneous} (Tameness) Suppose that $(g_1,...,g_m)$ is a Diophantine tuple with elements in a compact connected semisimple Lie group $G$. Suppose that $E$ is a homogeneous vector bundle that $G$ acts on. Then there exist constants $C_1,\alpha_1,\alpha_2>0$ such that for any $s\ge 0$ there exists $C_s$ such that for any nonzero $\phi\in C^{\infty}_0(G/K,E)$ the following holds:
\[
\|(I-\mc{L})^{-1}\phi\|_{H^s}\le C_1\|\phi\|_{H^{s+\alpha_1}}
\]
and
\[
\|(I-\mc{L})^{-1}\phi\|_{C^s}\le C_s\|\phi\|_{C^{s+\alpha_2}}.
\]
Moreover, these estimates are stable.
\end{prop}

\begin{proof}
As before, let $H_{\lambda}$ be the $\lambda$-eigenspace of $\Delta$ acting on sections of $E$. Let $\mc{L}_{\lambda}$ denote the action of $\mc{L}$ on $H_{\lambda}$. From Proposition \ref{prop:decay_estimate}, we see that there exist $D_1,D_2$ and $\alpha_3$ such that for all $\lambda>0$, $\|\mc{L}_{\lambda}^n\|_{H^0}\le D_1(1-1/(D_2\log^{\alpha_3}(\lambda))^n$. Thus there exists $C_3$ such that  $\|(I-\mc{L}_{\lambda})^{-1}\|_{H^0}\le C_3\log^{\alpha_3}(\lambda)$.  Now observe, that in the following sum that $\lambda\neq 0$ by our assumption that $\phi$ is orthogonal to the trivial representations contained in $L^2(E)$:
\begin{align*}
\|(I-\mc{L})^{-1}\phi\|_{H^s}^2&=\sum_{\lambda>0} (1+\lambda)^s\|(I-\mc{L}_{\lambda})^{-1}\phi_{\lambda}\|^2_{L^2}\\
&\le \sum_{\lambda>0} (1+\lambda)^s\|(I-\mc{L}_{\lambda})^{-1}\|^2\|\phi_{\lambda}\|^2_{L^2}\\
&\le \sum_{\lambda>0}C_3^2\log^{2\alpha_3}(\lambda)(1+\lambda)^s\|\phi_{\lambda}\|^2_{L^2}\\
&\le \sum_{\lambda>0}C_4^2(1+\lambda)^{s+\alpha_1}\|\phi_{\lambda}\|^2_{L^2}\\
&\le C_4^2\|\phi\|_{H^{s+\alpha_1}}^2,
\end{align*}
for any $\alpha_1>0$ and sufficiently large $C_4$. The second estimate in the proposition then follows from the first by applying the Sobolev embedding theorem.
\end{proof}

\subsection{Application to isotropic manifolds}\label{subsec:isotropic_manifolds}

We now introduce the class of isotropic manifolds, which are the subject of this paper and whose isometry groups may be studied along the above lines. We say that $M$ is \emph{isotropic} if $\Isom(M)$ acts transitively on the unit tangent bundle of $M$, $T^1M$. This is equivalent to $\Isom(M)^{\circ}$ acting transitively on $T^1M$. There are not many isotropic manifolds.  In fact, all are globally symmetric spaces. The following is the complete list of all compact isotropic manifolds:
\begin{enumerate}
\item
$S^n=\SO(n+1)/\SO(n)$, sphere,
\item
$\RP^n=\SO(n+1)/O(n)$, real projective space,
\item
$\CP^n=\SU(n+1)/U(n)$, complex projective space,
\item
$\HP^n=\Sp(n+1)/(\Sp(n)\times \Sp(1))$, quaternionic projective space,
\item
$F_4/\Spin(9)$, Cayley projective plane.
\end{enumerate}
A proof of this classification may be found in \cite[Thm. 8.12.2]{wolf1972spaces}. 

Though $S^1$ is an isotropic manifold, we will exclude it in all future statements because its isometry group is not semisimple. The reason that we study isotropic manifolds is that if $M$ is an isotropic manifold, $\Isom(M)$ is semisimple.
\begin{lem}\label{lem:isotropic_has_diophantine}
Suppose that $M$ is a compact connected isotropic manifold other than $S^1$, then $\Isom(M)$ is semisimple. The same is true for $\Isom^0(M)$, the connected component of the identity.
\end{lem}
For a proof of this Lemma, see \cite{shankar2001isometry}, which computes the isometry groups for each of these spaces explicitly. In fact, these isometry groups all have simple Lie algebras.

One minor issue with applying what we have developed so far to isotropic manifolds is that $\Isom(M)$ need not be connected. Even in the case of $S^2$, $\Isom(M)$ is disconnected. In fact, Dolgopyat and Krikorian assume that the isometries in their theorem all lie in the identity component of $\Isom(M)$ and hence are rotations. Here, we consider the full isometry group. Hence Theorem \ref{thm:KAM_converges} is a generalization even in the case of $S^n$. That said, the generalization is minor: the identity component is index $2$ in the full isometry group.

 Although connectedness of $\Isom(M)$ has not been the crux of previous arguments, if $\Isom(M)\neq \Isom(M)^{\circ}$, then there are ``extra" representations of $\Isom(M)$ that appear in the definition of Diophantineness that would need to be dealt with slightly differently. For this reason we give the following definition, which is adapted to the case where $\Isom(M)$ is not connected.
\begin{defn}
We say that a tuple $(g_1,\ldots,g_m)$ with each $g_i\in \Isom(M)$ is \emph{Diophantine} if there exists $n$ such that if $S=\{g_1,\ldots,g_m\}$ then $S^{n}\cap \Isom(M)^{\circ}$ is $(C,\alpha)$-Diophantine for some $C,\alpha>0$. We say that such a tuple is $(C,\alpha,n)$-Diophantine.
\end{defn}

It follows from Proposition \ref{prop:diophantine_equivalence} that if a tuple is Diophantine, then there exists a neighborhood of that tuple such that the constants $C,\alpha,n$ may be taken to be uniform over that neighborhood. Thus Diophantineness in this more general sense is a stable property. The following analogue of Proposition \ref{prop:diophantine_equivalence2} is then immediate.
\begin{prop}\label{prop:diophantine_equivalence2}
Let $M$ be a closed isotropic manifold of dimension at least $2$ and $S$ be a finite subset of $\Isom(M)$. The set $S$ is Diophantine if and only if $\Isom(M)^{\circ}\subseteq \overline{\langle S\rangle}$. Moreover, there exists $\epsilon_0(M),C,\alpha,n>0$ such that any subset of $\Isom(M)$ that is $\epsilon_0$-dense in $\Isom(M)^{\circ}$ is stably $(C,\alpha,n)$-Diophantine.
\end{prop}

We will show a tameness result in this setting.
The important point is that $\Isom(M)^{\circ}$ is a semisimple connected Lie group and $TM$ is a homogeneous vector bundle that $\Isom(M)^{\circ}$ acts on. Further, due to $M$ being isotropic $L^2(M,TM)$ contains no trivial representations of $\Isom(M)^{\circ}$.  Thus we are almost in a position where we can apply Proposition \ref{prop:tame_coboundaries}. There is one small issue: there may be representations of $\Isom(M)$ that are trivial on $\Isom(M)^{\circ}$ and hence the previous arguments do not apply directly to these representations. However, for the purpose of studying sections of $TM$, studying representations of $\Isom(M)^{\circ}$ suffices.  The following Proposition explains how one may get around this issue to recover the appropriate analog of Proposition \ref{prop:decay_estimate}. It is important to note that there are many choices of a ``Laplacian" acting on vector fields over a manifold, and they may not all be the same. In our case, we are choosing to work with the Casimir Laplacian, which arises from viewing $TM$ as a homogeneous vector bundle. Given a tuple $(g_1,\ldots,g_m)$ of isometries of $M$, the associated operator $\mc{L}$ that acts on $L^2(M,TM)$ is defined for a vector field $V$ by $V\mapsto m^{-1}\sum_{i=1}^m (Dg_i)_*V$.

\begin{prop}\label{prop:decay_lemma2}
Suppose that $M$ is a closed isotropic manifold with $\dim M\ge 2$. Suppose that $(g_1,\ldots,g_m)$ is a Diophantine tuple with elements in $\Isom(M)$. There exists a neighborhood $\mc{N}$ of $(g_1,\ldots,g_m)$ in $\Isom(M)\times \cdots \times \Isom(M)$ and constants $D_1,D_2,\alpha>0$ such that if $(g_1',\ldots,g_m')\in \mc{N}$, then $\{g_1',\ldots,g_m'\}$ is Diophantine. Let $H_{\lambda}$ denote the $\lambda$-eigenspace of $\Delta$ acting on sections of $TM$. For any tuple in this neighborhood, the associated operator $\mc{L}$ acts on $L^2(M,TM)$ and preserves the $H_{\lambda}$-eigenspaces. In fact, writing $\mc{L}_{\lambda}$ for this induced action we have that:
\[
\|\mc{L}^n_{\lambda}\|\le D_1\left(1-\frac{1}{D_2\log^{\alpha}(\lambda)}\right)^n.
\]
The same holds for the eigenspaces $H_{\lambda}$ of $\Delta$ acting on other bundles over $M$ assuming that $\Isom(M)$ acts isometrically on the space of sections of those bundles. In cases where there is a trivial representation, we must also assume $\lambda>0$. Examples of such bundles are $L^2(M,\R)$ as well as $L^2(\Gr_r(M),\R)$ in the case that $\Isom(M)^{\circ}$ acts transitively on the $r$-planes in $TM$.
\end{prop}

\begin{proof}
The key steps in the proof are substantially similar to those in Proposition \ref{prop:decay_estimate}, once we show that the elements of $\Isom(M)$ all preserve the spaces $H_{\lambda}$. Let $\Gamma$ be a bundle as in the statement of the proposition that $\Isom(M)$ acts on isometrically.

\begin{claim}
Suppose that $V\subset \Gamma$ is an irreducible representation of $\Isom(M)^{\circ}$ isomorphic to $(\pi, W)$.  Then for any $k\in \Isom(M)^{\circ}$, $kV$ is an irreducible representation of $V$ isomorphic to $(\pi\circ \alpha, W)$ for some automorphism $\alpha$ of $\Isom(M)^{\circ}$. In particular, $c(\pi\circ \alpha)=c(\pi)$.
\end{claim}
\begin{proof}

Let $g^k=k^{-1}gk$ as usual. We claim that for any $k\in \Isom(M)$ that $kV$ is a representation of $\Isom(M)^{\circ}$. To see this note that for $v\in V$, that $gkv=kg^kv$, but $g^k\in \Isom(M)^{\circ}$, so $kg^kv\in kV$. Moreover, it is straightforward to see that the representation of $\Isom(M)^{\circ}$ on $kV$ is isomorphic to the representation $(\pi\circ \alpha,W)$ where $\alpha$ is the automorphism $g\mapsto g^k$.

We now claim that $c(\pi\circ \alpha)=c(\pi)$. Because $\alpha$ is an automorphism, it preserves the Killing form, and hence we see that we can write the Casimir element as $\sum_i(d\alpha^{-1}(X_i))^2$. Now note that if one traces through the computation of what the value $c(\pi\circ \alpha)$ for the representation $\pi\circ \alpha$, that the $\alpha^{-1}$ we have introduced cancels with the $\alpha$. Thus the computation reduces to the computation of $c(\pi)$ with the original expression $\sum_i X_i^2$. Hence $c(\pi\circ \alpha)=c(\pi)$.
\end{proof}

To conclude from this point, one does the same argument as in Proposition \ref{prop:decay_estimate}, except we start with the set $S^{n_0}$ and only make use of the elements in $S^{n_0}\cap \Isom(M)^{\circ}$. No issues arise because any terms that do not lie in $\Isom(M)^{\circ}$ are isometries of $H_{\lambda}$ as we have now shown.
\end{proof}

Having established the previous proposition the following is immediate and may be shown by repeating the argument of Proposition \ref{prop:tame_coboundaries}. 
\begin{prop}\label{prop:tame_coboundaries2}
Suppose that $M$ is a closed isotropic manifold with $\dim M\ge 2$. Suppose that $(g_1,\ldots,g_m)$ is a Diophantine tuple with elements in $\Isom(M)$. There exist constants $C_1,\alpha_1,\alpha_2>0$ such that for any $s\ge 0$ there exists $C_s$ such that for any $\phi\in C^{\infty}(M,TM)$ the following holds:
\[
\|(I-\mc{L})^{-1}\phi\|_{H^s}\le C_1\|\phi\|_{H^{s+\alpha_1}},
\]
and
\[
\|(I-\mc{L})^{-1}\phi\|_{C^s}\le C_s\|\phi\|_{C^{s+\alpha_2}}.
\]
Moreover these estimates are stable. The same holds for the action of $\mc{L}$ on any of the sections of any of the bundles that Proposition \ref{prop:decay_lemma2} applies to.
\end{prop}

\section{Approximation of Stationary Measures}\label{sec:stationary_measures}

In this section, we introduce the notion of a stationary measure associated to a random dynamical system. We consider stationary measures of certain random dynamical systems associated to a Diophantine subset of a compact semisimple Lie group as well as perturbations of these systems. We begin by introducing these systems and some associated transfer operators. In Proposition \ref{prop:taylor_expansion_of_haar}, we give an asymptotic expansion of the stationary measures of a perturbation.

\subsection{Random dynamical systems and their transfer operators}\label{subsec:rds_defns}

We now give some basic definitions concerning random dynamical systems. For general treatments of random dynamical systems and their basic properties, see \cite{kifer1986ergodic} or \cite{arnold2013random}. If $(f_1,...,f_m)$ is a tuple of maps of a standard Borel space $M$, then these maps generate a uniform Bernoulli random dynamical system on $M$. This dynamical system is given by choosing an index $1\le i\le m$ uniformly at random and then applying the function $f_i$ to $M$. To iterate the system further, one chooses additional independent uniformly distributed indices and repeats. We always use the words \emph{random dynamical system} to mean uniform Bernoulli random dynamical system in the sense just described. 

Associated to this random dynamical system are two operators. The first operator is called the \emph{averaged Koopman operator}. It acts on functions and is defined by
\begin{equation}
\mc{M}\phi\coloneqq \frac{1}{m}\sum_{i=1}^m \phi\circ f_i.
\end{equation}
The second operator is called the \emph{averaged transfer operator}. It acts on measures and is defined by
\begin{equation}
\mc{M}^*\mu\coloneqq \frac{1}{m}\sum_{i=1}^m (f_i)_*\mu.
\end{equation}
Depending on the space $M$, we may restrict the domains of these operators to a suitable subset of the spaces of functions and measures on $M$. We say that a measure is \emph{stationary} if $\mc{M}^*\mu=\mu$. We assume that stationary measures have unit mass.

In this paper, we take $M$ to be a compact homogeneous space $G/K$. If $g\in G$, then left translation by $g$ gives an isometry of $G/K$ that we also call $g$. As before, a tuple $(g_1,...,g_m)$ with each $g_i\in G$ generates a random dynamical system on $G/K$. We will also consider perturbations of this random dynamical system. Consider a tuple $(f_1,...,f_m)$ where each $f_i\in\Diff^{\infty}(G/K)$. This collection also generates a random dynamical system on $G/K$.  The indices $1,...,m$ give a natural way to compare the two systems. We refer to the initial system as homogeneous or linear and to the latter system as non-homogeneous or non-linear. 

We will simultaneously work with a homogeneous and non-homogeneous systems, so we now introduce notation to distinguish the transfer operators of each.
We write $\mc{M}$ for the averaged Koopman operator associated to the system generated by the tuple $(g_1,...,g_m)$ and we write $\mc{M}_{\epsilon}$ for the averaged Koopman operator associated to the tuple $(f_1,...,f_m)$. Analogously we use the notation $\mc{M}^*$ and $\mc{M}_{\epsilon}^*$.

Later we will compare the homogeneous system given by a tuple $(g_1,...,g_m)$ and a non-homogeneous perturbation $(f_1,...,f_m)$.  We thus introduce the notation
\begin{equation}
\varepsilon_k\coloneqq\max_{i} d_{C^k}(f_i,g_i),
\end{equation}
for describing how large a perturbation is.
In addition, it will be useful to have a linearization of the difference between $f_i$ and $g_i$. The standard way to do this is via a chart on the Fréchet manifold $\Diff^{\infty}(G/K)$.  If $d_{C^0}(f_i,g_i)<\inj G/K$, then we associate $f_i$ with the vector field $Y_i$ defined at $g_i(x)\in G/K$ by
\begin{equation}\label{eqn:linearized_error}
Y_i(g_i(x))\coloneqq \exp_{g_i(x)}^{-1} f_i(x),
\end{equation}
where we choose the minimum length preimage of $f_i(x)$ in $T_{g_i(x)}G/K$ under the map $\exp^{-1}_{g_i(x)}$. In addition, if $Y$ is a vector field on $M$, then we define $\psi_Y\colon M\to M$ to be the map that sends 
\begin{equation}\label{eq:psi_defn}
\psi_Y:x\mapsto \exp_x(Y(x)).
\end{equation}

The following theorem asserts the existence of Lyapunov exponents for random dynamical systems.

\begin{thm}
\cite[Ch. 3, Thm. 1.1]{kifer1986ergodic}. Suppose that $E$ is measurable vector bundle over a Borel space $M$. Suppose that $F_1,F_2,...$ is a sequence of independent and identically distributed bundle maps of $E$ with common distribution $\nu$ and suppose that $\nu$ has finite support. Suppose that $\mu$ is an ergodic  $\nu$-stationary measure on $M$ for the random dynamics on $M$ induced by those on $E$.

Then there exists a list of numbers, the Lyapunov exponents,
\[
-\infty<\lambda^s<\lambda^{s-1}<\cdots<\lambda^1<\infty,
\]
such that for $\mu$ a.e.~$x\in M$ and almost every realization of the sequence, there exists a filtration of linear subspaces 
\[
0\subset V^s\subset \cdots \subset V^1\subset E_x
\]
such that, for that particular realization of the sequence, if $\xi\in V^{i+1}\setminus V^{i}$, where $V^i\equiv \{0\}$ for $i>s$, then
\[
\lim_{n\to \infty} \frac{1}{n} \log \|F^n\circ \cdots \circ F^1 \xi\|=\lambda^i.
\]
\end{thm}

\subsection{Approximation of stationary measures}

Let $dm$ denote the push-forward of Haar measure to $G/K$. Note that Haar measure is stationary for the homogeneous random dynamical system given by $(g_1,...,g_m)$.
The following proposition compares the integral against a stationary measure $\mu$ for a perturbation $(f_1,...,f_m)$ and the Haar measure. Up to higher order terms, the difference between integrating against Haar and against $\mu$ is given by the integral of a particular function $\mc{U}(\phi)$. We obtain an explicit expression for $\mc{U}(\phi)$, which is useful because we can tell when $\mc{U}(\phi)$ vanishes and thus when $\mu$ is near to Haar. Compare the following with \cite[Prop. 2]{dolgopyat2007simultaneous}.
\begin{prop}\label{prop:taylor_expansion_of_haar}
Suppose that $S=(g_1,...,g_m)$ is a Diophantine tuple with elements in a compact connected semisimple group $G$ or elements in $\Isom(M)$ for an isotropic manifold $M$ with $\dim M\ge 2$. Let $G/K$ be a quotient of $G$ in the former case or a space $\Isom(M)^{\circ}$ acts transitively on in the latter. There exist constants $k$ and $C$ such that if $(f_1,...,f_m)$ is a tuple with elements in $\Diff^{\infty}(G/K)$ with $\varepsilon_0=\max_i d_{C^0}(f_i,g_i)<\inj G/K$, then the following holds for each stationary measure $\mu$ for the uniform Bernoulli random dynamical system generated by the $f_i$. Let $Y_i=\exp^{-1}_{g_i(x)}f_i(x)$. Then for any $\phi\in C^{\infty}(G/K)$, we have
\begin{equation}\label{eq:taylor_expansion_of_haar_1}
\int_{G/K}\phi\,d\mu=\int_{G/K}\phi\,dm+\int_{G/K}\mc{U}(\phi)\,dm+O(\varepsilon_k^2\|\phi\|_{C^{k}}),
\end{equation}
where $dm$ denotes the normalized push-forward of Haar measure to $G/K$ and 
\begin{equation}
\mc{U}(\phi)\coloneqq\frac{1}{m}\sum_{i=1}^m\nabla_{Y_i}(I-\mc{M})^{-1}(\phi-\smallint\phi\,dm).
\end{equation}
Moreover,
\begin{equation}\label{eqn:error_term_control_estimate}
\abs{\int \mc{U}(\phi)\,dm}\le C\|\phi\|_{C^k}\left\|\sum_{i=1}^m Y_i\right\|_{C^k},
\end{equation}
and the constants, including the constant in the big-$O$ in equation \eqref{eq:taylor_expansion_of_haar_1}, are stable in $S$. 
\end{prop}

\begin{proof} The proof is similar to the proof of \cite[Prop.\ 4]{malicet2012simultaneous}. We write the proof for the connected group $G$; the proof for $\Isom(M)$ is identical with us using Proposition \ref{prop:tame_coboundaries2} instead of Proposition \ref{prop:tame_coboundaries}. 

Note that a smooth real valued function defined on $G/K$ is naturally viewed as a section of the trivial bundle over $G/K$. If we view the averaged Koopman operator $\mc{M}$ associated to $(g_1,...,g_m)$ as acting on the sections of the trivial bundle $G/K\times \R$, then $\mc{M}$ satisfies the hypotheses of Proposition \ref{prop:tame_coboundaries}. Thus there exists $\alpha$ and constants $C_s$ such that for any $\phi\in C^{\infty}_0(G/K)$, the space of integral $0$ smooth functions on $G/K$, 
\begin{equation}\label{eq:tameness1}
\|(I-\mc{M})^{-1}\phi\|_{C^s}\le C_s\|\phi\|_{C^{s+\alpha}}.
\end{equation}
Observe that for any $i$:
\[
\abs{\phi\circ f_i(x)-\phi\circ g_i(x)}\le \varepsilon_0\|\phi\|_{C^1}.
\]
Since $\mu$ is $\mc{M}_{\epsilon}^*$ invariant, this implies that
\[
\abs{\int \phi-\mc{M}\phi\,d\mu}=\abs{\int \mc{M}_{\varepsilon}\phi-\mc{M}\phi\,d\mu}\le \varepsilon_0\|\phi\|_{C^1}.
\]
Substituting $(I-\mc{M})^{-1}(\phi-\int \phi\,dm)$ for the function $\phi$ in the previous line and using equation \eqref{eq:tameness1} yields a first order approximation:
\begin{equation}\label{eq:first_order_haar_est}
\abs{\int \phi\,d\mu-\int\phi\,dm}\le \varepsilon_0 C_1\|\phi\|_{C^{1+\alpha}}.
\end{equation}

We now use this first order approximation to obtain a better estimate. Note the Taylor expansion:
\[
\phi\circ f_i(x)-\phi\circ g_i(x)=(\nabla_{Y_i} \phi)(g_i(x))+O(\varepsilon_0^2\|\phi\|_{C^2}).
\]
Integrating against $\mu$ yields 
\[
\int \phi-\mc{M}\phi\,d\mu=\int \mc{M}_{\varepsilon} \phi-\mc{M}\phi\,d\mu=\int \frac{1}{m}\sum_{i=1}^m \nabla_{Y_i} \phi(g_i(x))\,d\mu+O(\varepsilon_0^2\|\phi\|_{C^2}).
\]
We now plug in $(I-\mc{M})^{-1}(\phi-\smallint \phi\,dm)$ for $\phi$ in the previous line and use the estimate in equation \eqref{eq:tameness1} to obtain:
\[
\int \phi\,d\mu-\int\phi\,dm=\int \frac{1}{m}\sum_{i=1}^m\pez{\nabla_{Y_i}(I-\mc{M})^{-1}(\phi-\smallint \phi\,dm)}(g_i(x))\,d\mu+O(\varepsilon_0^2\|\phi\|_{C^{2+\alpha}}).
\]
Using equation \eqref{eq:first_order_haar_est} on the first term on the right hand side above yields
\begin{align}\label{eq:prop_approx_by_haar_halfway}
\int \phi\,d\mu-\int\phi\,dm=&\int \frac{1}{m}\sum_{i=1}^m\pez{\nabla_{Y_i}(I-\mc{M})^{-1}(\phi-\smallint \phi\,dm)}(g_i(x))\,dm\\&+O\pez{\varepsilon_0\left\|\sum_{i=1}^m \nabla_{Y_i}(I-\mc{M})^{-1}\phi\right\|_{C^{1+\alpha}}}+O(\varepsilon_0^2\|\phi\|_{C^{2+\alpha}}).\nonumber
\end{align}
Note that
\[
\left\|\sum_{i=1}^m \nabla_{Y_i}(I-\mc{M})^{-1}\phi\right\|_{C^{1+\alpha}}
=O(\varepsilon_{2+\alpha}\|(I-\mc{M})^{-1}\phi\|_{C^{2+\alpha}}).
\]
The application of equation \eqref{eq:tameness1} to $\|(I-\mc{M})^{-1}\phi\|_{C^{2+\alpha}}$then gives that the first big $O$-term in \eqref{eq:prop_approx_by_haar_halfway} is $O(\varepsilon_0\varepsilon_{2+\alpha}\|\phi\|_{C^{2+2\alpha}})$. Thus,  
\[
\int \phi\,d\mu-\int\phi\,dm=\int \frac{1}{m}\sum_{i=1}^m\pez{\nabla_{Y_i}(I-\mc{M})^{-1}(\phi-\smallint \phi\,dm)}(g_i(x))\,dm+O(\varepsilon_{2+\alpha}^2\|\phi\|_{C^{2+2\alpha}}).
\]
Now, by translation invariance of the Haar measure we may remove the $g_i$'s:
\[
\int \phi\,d\mu-\int\phi\,dm=\int \frac{1}{m}\sum_{i=1}^m\nabla_{Y_i}(I-\mc{M})^{-1}(\phi-\smallint \phi\,dm)\,dm+O(\varepsilon_{2+\alpha}^2\|\phi\|_{C^{2+2\alpha}}).
\]
This proves everything except equation \eqref{eqn:error_term_control_estimate}.

We now estimate the integral of
\begin{align*}
\mc{U}(\phi)&=\frac{1}{m}\sum_{i=1}^m\nabla_{Y_i}(I-\mc{M})^{-1}(\phi-\smallint\phi\,dm),\\
&=\nabla_{\frac{1}{m}\sum_{i=1}^m Y_i}(I-\mc{M})^{-1}(\phi-\smallint\phi\,dm),\\
\end{align*}
against Haar.
By equation \eqref{eq:tameness1} there exists $C_1$ such that  
\[
\left\|(I-\mc{M})^{-1}(\phi-\int\phi\,dm)\right\|_{C^1}\le C_1\|\phi\|_{C^{1+\alpha}},
\]
which establishes equation \eqref{eqn:error_term_control_estimate} by a similar argument to the estimate of the big-O term occurring in the previous part of this proof.
\end{proof}

\section{Strain and Lyapunov Exponents}\label{sec:strain_and_Lyapunov_exponents}

In this section we study the Lyapunov exponents of perturbations of isometric systems. The main result is Proposition \ref{prop:taylor_expansion_lambda_k}, which gives a Taylor expansion of the Lyapunov exponents of a perturbation. The terms appearing in the Taylor expansion have a particular geometric meaning. We explain this meaning in terms of two ``strain" tensors associated to a diffeomorphism. These tensors measure how far a diffeomorphism is from being an isometry. After introducing these tensors, we prove Proposition \ref{prop:taylor_expansion_lambda_k}. The Lyapunov exponents of a random dynamical system may be calculated by integrating against a stationary measure of a certain extension of the original system. By using Proposition \ref{prop:taylor_expansion_of_haar}, we are able to approximate such stationary measures by the Haar measure and thereby obtain a Taylor expansion.

\subsection{Norms on Tensors}\label{subsec:tensor_norms}

Throughout this paper we use the pointwise $L^2$ norm on tensors, which we now describe. For a more detailed discussion, see the discussion surrounding \cite[Prop. 2.40]{lee2018introduction}. If $V$ is an inner product space with orthonormal basis $[e_1,\ldots,e_n]$, then $V^{\otimes k}$ has a basis of tensors of the form
\[
e_{i_1}\otimes \cdots \otimes e_{i_k}
\]
where $1\le i_j\le n$ for each $1\le j \le k$. We declare the vectors of this basis to be orthonormal for the inner product on $V^{\otimes k}$. This norm is independent of the choice of orthonormal basis. For a continuous tensor field $T$ on a closed Riemannian manifold $M$, we write $\|T\|$ for $\max_{x\in M}\|T(x)\|$.  If $T$ is a tensor on a Riemannian manifold $M$, we then define its $L^2$ norm in the expected way by integrating the norm of $T(x)$ as a tensor on $T_xM$ over all points $x\in M$, i.e.
\[
\|T\|_{L^2}=\left(\int_M \|T(x)\|^2\,d\vol(x)\right)^{1/2}.
\]

\subsection{Strain}\label{subsec:strain}

If a diffeomorphism of a Riemannian manifold is an isometry, then it pulls back the metric tensor to itself. Consequently, if we are interested in how near a diffeomorphism is to being an isometry, it is natural to consider the difference between the metric tensor and the pullback of the metric tensor. This leads us to the following definition.

\begin{defn}
Suppose that $f$ is a diffeomorphism of a Riemannian manifold $(M,g)$. We define the \emph{Lagrangian strain tensor} associated to $f$ to be 
\[
E^f\coloneqq \frac{1}{2}\left(f^*g-g\right).
\]
\end{defn}
This definition is consonant with the definition of the Lagrangian strain tensor that appears in continuum mechanics, c.f. \cite{lai2009introduction}. 

The strain tensor will be useful for two reasons. First, it naturally appears in the Taylor expansion in Proposition \ref{prop:taylor_expansion_lambda_k}, which will allow us to conclude that a random dynamical system with small Lyapunov exponents has small strain. Secondly, we prove in Theorem \ref{thm:near_to_identity} that for certain manifolds that a diffeomorphism with small strain is near to an isometry. The combination of these two things will be essential in the proof of our main linearization result, Theorem \ref{thm:KAM_converges}, which shows that perturbations with all Lyapunov exponents zero are conjugate to isometric systems.

We now introduce two refinements of the strain tensor that will appear in the Taylor expansion in Proposition \ref{prop:taylor_expansion_lambda_k}. Note that $E^f$ is a $(0,2)$-tensor. Consequently, we may take its trace with respect to the ambient metric $g$.

\begin{defn}
Suppose that $f$ is a diffeomorphism of a Riemannian manifold $(M,g)$. We define the \emph{conformal strain tensor}  by
\[
E_C^f\coloneqq \frac{\Tr(f^*g-g)}{2d}g.
\]
We define the \emph{nonconformal strain tensor} by
\[
E_{NC}^f\coloneqq E^f-E_C^f=\frac{1}{2}\pez{f^*g-g-\frac{\Tr(f^*g-g)}{d}g}.
\]
\end{defn}

\subsection{Taylor expansion of Lyapunov exponents}

Suppose that $M$ is a manifold and that $f$ is a diffeomorphism of $M$. Let $\Gr_r(M)$ denote the Grassmannian bundle comprised of $r$-planes in $TM$. When working with $\Gr_r(M)$ we write a subspace of $T_xM$ as $E_x$ to emphasize the basepoint. Then $f$ naturally induces a map $F\colon \Gr_r(M)\to \Gr_r(M)$ by sending a subspace $E_x\in \Gr_r(T_xM)$ to $D_xfE_x\in \Gr_r(T_{f(x)}M)$. If we have a random dynamical system on $M$, then by this construction we naturally obtain a random dynamical system on $\Gr_r(M)$.  The following Proposition should be compared with \cite[Prop. 3]{dolgopyat2007simultaneous}.

\begin{prop}\label{prop:taylor_expansion_lambda_k}
Suppose that $M$ is a compact connected Riemannian manifold such that $\Isom(M)$ is semisimple and that $\Isom(M)^{\circ}$ acts transitively on $\Gr_r(M)$. Suppose that $S=(g_1,...,g_m)$ is a Diophantine tuple of elements of $\Isom(M)$. Then there exists $\epsilon>0$ and $k>0$ such that if $(f_1,...,f_m)$ is a tuple with elements in $\Diff^{\infty}(M)$ such that $d_{C^k}(f_i,g_i)<\epsilon$, then the following holds. Suppose that $\mu$ is an ergodic stationary measure for the random dynamical system obtained from the $(f_1,...,f_m)$. Let $\Lambda_r$ be the sum of the top $r$ Lyapunov exponents of $\mu$. Then
\begin{align}\label{eqn:lambda_r_taylor_expansion}
\Lambda_r(\mu)=&-\frac{r}{2dm}\sum_{i=1}^m \int_M \|E_C^{f_i}\|^2\,d\vol+\frac{r(d-r)}{(d+2)(d-1)m}\sum_{i=1}^m \int_M \|E_{NC}^{f_i}\|^2\,d\vol\\
&+\int_{\Gr_r(M)}\mc{U}(\psi)\,d\vol+O(\varepsilon_k^3).\nonumber
\end{align}
where $\psi=\frac{1}{m}\sum_{i=1}^m \ln \det(Df_i\mid E_x)$, $\varepsilon_k=\max_i\{d_{C^k}(f_i,g_i)\}$, $\mc{U}$ is defined as in Proposition \ref{prop:taylor_expansion_of_haar}, and $\det$ is defined in Appendix \ref{sec:determinants}.

\end{prop}

\begin{proof}
Given the random dynamical system on $M$ generated by the tuple $(f_1,...,f_m)$, there is the induced random dynamical system on $\Gr_r(M)$ generated by the tuple $(F_1,...,F_m)$. The Lyapunov exponents of the system on $M$ may be obtained from the system on $\Gr_r(M)$ in the following way. By \cite[Ch. III, Thm 1.2]{kifer1986ergodic}, given an ergodic stationary measure $\mu$ on $M$, there exists a stationary measure $\overline{\mu}$ on $\Gr_r(M)$ such that
\[
\Lambda_r(\mu)=\frac{1}{m}\sum_{i=1}^m \int_{\Gr_r(M)} \ln \det(Df_i\mid E_x)\,d\overline{\mu}(E_x).
\]
Reversing the order of summation, this is equal to
\begin{equation}\label{eq:int_against_mu_bar}
\int_{\Gr_r(M)}\frac{1}{m}\sum_{i=1}^m  \ln \det(Df_i\mid E_x)\,d\overline{\mu}(E_x).
\end{equation}
As $\Isom(M)$ acts transitively on $\Gr_r(M)$, $\Gr_r(M)$ is a homogeneous space of $\Isom(M)$. Thus as $(g_1,...,g_m)$ is Diophantine, we may apply Proposition \ref{prop:taylor_expansion_of_haar} to approximate the integral in equation \eqref{eq:int_against_mu_bar}.  Letting $\mc{U}$ be as in that proposition, there exists $k$ such that 
\begin{align}\label{eq:lambda_r_first}
\Lambda_r(\mu)=&\int_{\Gr_r(M)} \frac{1}{m}\sum_{i=1}^m  \ln \det(Df_i\mid E_x)\,d\vol(E_x)+\int_{\Gr_r(M)}\mc{U}\pez{\frac{1}{m}\sum_{i=1}^m  \ln \det(Df_i\mid E_x)}\,d\vol\\
&+O\left((\max_i\{d_{C^k}(F_i,G_i)\})^2\left\|\sum_{i=1}^m \ln\det(Df_i\mid E_x)\right\|_{C^k}\right)\nonumber
\end{align}
We now estimate the error term. The following two estimates follow by working in a chart on $\Gr_r(M)$. If $f,g$ are two maps of $M$ and $F,G$ are the induced maps on $\Gr_r(M)$, then $d_{C^k}(F,G)=O(d_{C^{k+1}}(f,g))$. In addition, by Lemma \ref{lem:log_det_on_grassmannian}
we have that
\begin{equation}\label{eqn:est_on_ck_log_det}
\left\|\sum_{i=1}^m \ln\det(Df_i\mid E_x)\right\|_{C^k}=O(\varepsilon_{k+1}).
\end{equation}

Thus the error term in \eqref{eq:lambda_r_first} is small enough to conclude \eqref{eqn:lambda_r_taylor_expansion}.

To finish, we apply the Taylor expansion in Proposition \ref{prop:taylor_expansion_of_log_jacobian}, which is in Appendix \ref{sec:taylor_expansions}, to 
\[
\int_{\Gr_r(M)} \ln \det(Df_i\mid E_x)\,d\vol(E_x),
\]
which gives precisely the first two terms on the right hand side of equation \eqref{eqn:lambda_r_taylor_expansion} and error that is $O(\varepsilon_1^3)$.
\end{proof}

\section{Diffeomorphisms of Small Strain: Extracting an Isometry in the KAM Scheme}\label{sec:diffs_of_small_strain}

In this section we prove Proposition \ref{lem:H_0_nearby_isometry}, which gives that a diffeomorphism of small strain on an isotropic manifold is near to an isometry. In the KAM scheme, we will see that diffeomorphisms with small Lyapunov exponents are low strain and hence conclude by Proposition \ref{lem:H_0_nearby_isometry} that they are near to isometries. 
Proposition \ref{lem:H_0_nearby_isometry} follows from Theorem \ref{thm:near_to_identity}, which shows that certain diffeomorphisms with small strain of a closed Riemannian manifold are $C^0$ close to the identity.

\begin{thm}\label{thm:near_to_identity}
Suppose that $(M,g)$ is a closed Riemannian manifold. Then there exists $1>r>0$ and $C>0$ such that if $f\in \Diff^{2}(M)$ and
\begin{enumerate}
\item
there exists $x\in M$ such that $f(x)=x$ and $\|D_xf-\Id\|=\epOne<r,$
\item
$\|f^*g-g\|=\epTwo<r$, and
\item
$d_{C^2}(f,\Id)=\epThree<r$,
\end{enumerate}
then for all $\gamma\in (0,r)$,
\[
d_{C^0}(f,\Id)\le C(\epOne+\epThree\gamma+\epTwo\gamma^{-1}).
\]
\end{thm}

Theorem \ref{thm:near_to_identity} is the main ingredient in the proof of our central technical result.

\begin{prop}\label{lem:H_0_nearby_isometry}
Suppose that $(M,g)$ is a closed isotropic Riemannian manifold. Then for all $\sigma>0$ and all integers $\ell>0$, there exist $k$ and $C,r>0$ such that for every $f\in \Diff^k(M)$, if there exists an isometry $I\in \Isom(M)$ such that
\begin{enumerate}
\item
$d_{C^k}(I,f)<r$, and
\item
$\|f^*g-g\|_{H^0}<r$,
\end{enumerate}
then there exists an isometry $R\in \Isom(M)$ such that 
\begin{align}
d_{C^0}(R,I)&<C(d_{C^2}(f,I)+\|f^*g-g\|_{H^0}^{1-\sigma}),\text{ and}\label{eqn:H_0_lemma_low_reg_est} \\
d_{C^{\ell}}(f,R)&<C(\|f^*g-g\|_{H^0}^{1/2-\sigma}d_{C^2}(f,I)^{1/2-\sigma}).\label{eqn:H_0_lemma_high_reg_est}
\end{align}
\end{prop}

Though the statement of Proposition \ref{lem:H_0_nearby_isometry} is technical, its use in the proof of Theorem \ref{thm:KAM_converges} is fairly transparent: the proposition produces an isometry near to a diffeomorphism  with small strain, which is the essence of iterative step in the KAM scheme. This remedies the gap in \cite{dolgopyat2007simultaneous}.

\subsection{Low strain diffeomorphisms on a general manifold: proof of Theorem \ref{thm:near_to_identity}}

The main geometric idea in the proof of Theorem \ref{thm:near_to_identity} is to study distances by intersecting spheres. In order to show that a diffeomorphism $f$ is close to the identity, we must show that it does not move points far. As we shall show, a diffeomorphism of small strain distorts distances very little. Consequently, a diffeomorphism of small strain nearly carries spheres to spheres. If we have two points $x$ and $y$ that are fixed by $f$, then the unit spheres centered at $x$ and $y$ are carried near to themselves by $f$. Consequently, the intersection of those spheres will be nearly fixed by $f$. By considering the intersection of spheres in this way, we may take a small set on which $f$ nearly fixes points and enlarge that set until it fills the whole manifold.

Before the proof of the theorem we prove several lemmas.

\begin{lem}\label{lem:lipschitz_constant}
Let $M$ be a closed Riemannian manifold. There exists $C>0$ such that the following holds. If $f\in \Diff^1(M)$ and $\|f^*g-g\|\le \eta$ then for all $x,y\in M$, 
\[
(1-C\eta)d(x,y)\le d(f(x),f(y))\le (1+C\eta)d(x,y).
\]
\end{lem}

\begin{proof}

If $\gamma$ is a path between $x$ and $y$ parametrized by arc length, then $f\circ \gamma$ is a path between $f(x)$ and $f(y)$. The length of $f\circ \gamma$ is equal to 
\begin{align*}
\len(f\circ \gamma)&=\int_0^{\len(\gamma)}\sqrt{g(Df\dot \gamma,Df\dot\gamma)}\,dt\\
&=\int_0^{\len(\gamma)} \sqrt{f^*g(\dot\gamma,\dot\gamma)}\,dt\\
&=\int_0^{\len(\gamma)} \sqrt{g(\dot\gamma,\dot\gamma)+[f^*g-g](\dot\gamma,\dot\gamma)}\,dt\\
&=\int_0^{\len(\gamma)} \sqrt{1+[f^*g-g](\dot\gamma,\dot\gamma)}\,dt.
\end{align*}
By our assumption on the norm of $f^*g-g$, there exists $C$ such that $\abs{[f^*g-g](\dot\gamma,\dot\gamma)}\le C\eta$. Then using that $\sqrt{1+x}\le 1+x$ for $x\ge 0$, we see that 
\[
\len(f\circ \gamma)\le \int_0^{\len(\gamma)} 1+\abs{[f^*g-g](\dot\gamma,\dot\gamma)}\,dt\le \len(\gamma)+C\eta \len(\gamma).
\] 
The lower bound follows similarly by using that $1+x\le \sqrt{1+x}$ for $-1\le x\le 0$.
\end{proof}

\begin{lem}\label{lem:fixed_neighborhood}
Let $M$ be a closed Riemannian manifold. Then there exist $r,C>0$ such that for all $f\in \Diff^2(M)$, if
\begin{enumerate}
\item
there exists $x\in M$ such that $f(x)=x$ and $\|D_xf-\Id\|=\epOne<r$, and
\item
$d_{C^2}(f,\Id)= \epThree<r$,
\end{enumerate}
then for all $0<\gamma<r$ and $y$ such that $d(x,y)<\gamma$ 
\[
d(y,f(y))\le C(\gamma \epOne+\gamma^2\epThree).
\]
\end{lem}

\begin{proof}
Let $r=\inj M/2$. We work in a fixed exponential chart centered at $x$, so that $x$ is represented by $0$ in the chart. Write
\[
f(y)=0+D_0fy+R(y)=y+(D_0f-\Id)y+R(y).
\]
As the $C^2$ distance between $f$ and the identity is at most $\epThree$, by Taylor's Theorem $R(y)$ is bounded in size by $C\epThree\abs{y}^2$ for a uniform constant $C$. Thus 
\[
\abs{f(y)-y}\le \epOne\abs{y}+C\epThree\abs{y}^2.
\]
In particular, for all $y$ such that $\abs{y}\le \gamma < r$,
\[
\abs{f(y)-y}\le C'(\gamma \epOne+\gamma^2\epThree).
\]
But the distance in such a chart is uniformly bi-Lipschitz with respect to the metric on $M$, so the lemma follows.
\end{proof}

The following geometric lemma produces points on two spheres in a Riemannian manifold that are further apart than the centers of the spheres. 

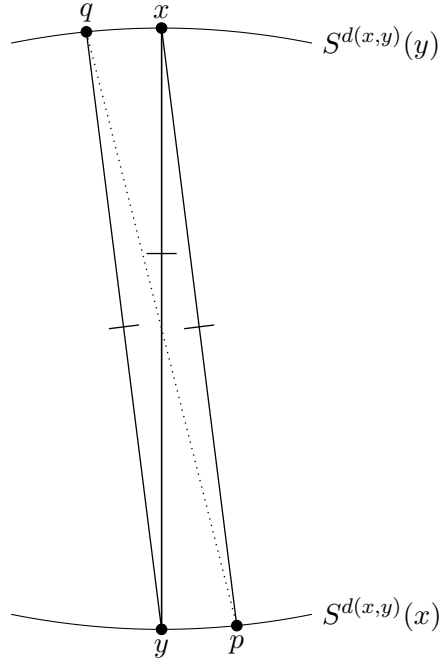
\begin{figure}\label{fig:points}
\begin{center}
\begin{tikzpicture}
\draw [line width = .5pt] (0,0) -- (0,8);
\draw [line width = .5pt] (0,0) -- (0,8);
\draw [line width = .5pt] (0,8) -- (1,1/20);
\draw [line width = .5pt] (-1,8-1/20) -- (0,0);
\draw [line width = .5pt, dotted] (-1,8-1/20) -- (1,1/20);
\draw [line width = .5pt] (-.2,5) -- (.2,5);
\draw [line width = .5pt] (.3,4+1/40-.2/8) -- (.7,4+1/40+.2/8);
\draw [line width = .5pt] (-.7,4+1/40-.2/8) -- (-.3,4+1/40+.2/8);

\draw[scale=1, domain=-2:2, smooth, variable=\x, black] plot ({\x}, {\x*\x/20});
\draw[scale=1, domain=-2:2, smooth, variable=\x, black] plot ({\x}, {8-\x*\x/20});

\filldraw (0,0) circle (2pt) node[align=left,   below] {$y$};
\filldraw (0,8) circle (2pt) node[align=left,   above] {$x$};
\filldraw (1,1/20) circle (2pt) node[align=left,   below] {$p$};
\filldraw (-1,8-1/20) circle (2pt) node[align=left,   above] {$q$};

\filldraw (2,4/20) node[align=right,   right]{$S^{d(x,y)}(x)$};
\filldraw (2,8-4/20) node[align=right,   right]{$S^{d(x,y)}(y)$};
\end{tikzpicture}
\caption{The four points $x,y,p,q$ appearing in Lemma \ref{lem:geometric_lemma}. Given $x,y,p$, the lemma produces the point $q$ and gives an estimate on the length of the dotted line,  which is longer than $d(x,y)$.}
\end{center}
\end{figure}

\begin{lem}\label{lem:geometric_lemma}
Let $M$ be a closed Riemannian manifold. There exist $C,r>0$ such that for all $\beta\in (0,r)$, if $x,y\in M$ satisfy $\frac{\inj M}{3}<d(x,y)<\frac{\inj M}{2}$, and
 there is a fixed $p\in M$ such that $d(x,p)=d(y,x)$ and $d(p,y)<r$, then there exists $q\in M$ depending on $p$ such that:
\begin{enumerate}
\item
$d(q,y)=d(y,x)$,
\item
$d(q,x)<\beta$, and
\item
$d(q,p)\ge d(x,y)+Cd(y,p)\beta$.
\end{enumerate}

\end{lem}

In order to prove Lemma \ref{lem:geometric_lemma}, we recall the following form of the second variation of length formula. For a proof of this and related discussion, see \cite[Ch.\ 1,§6]{cheeger1975comparison}.
\begin{lem}\label{lem:second_variation_of_length}
Let $M$ be a Riemannian manifold and $\gamma$ be a unit speed geodesic. Let $\gamma_{v,w}$ be a two parameter family of constant speed geodesics parametrized by $\gamma_{v,w}\colon [a,b]\times (-\epsilon,\epsilon)\times (-\epsilon,\epsilon)\to M$ such that $\gamma_{0,0}=\gamma$. Suppose that $\frac{\partial\gamma_{v,w}}{\partial v}=V$ and $\frac{\partial \gamma_{v,w}}{\partial w}=W$ are both normal to $\dot{\gamma}_{0,0}$, which we denote by $T$.  Then
\[
\frac{\partial^2 \len(\gamma_{v,w})}{\partial v\partial w}=\langle \nabla_W V,T\rangle \vert_{a}^b+\langle V,\nabla_T W\rangle \vert_a^b.
\]
\end{lem}

\begin{proof}[Proof of Lemma \ref{lem:geometric_lemma}.]

We will give a geometric construction using the points $x$ and $y$ and then explain how this construction may be applied to the particular point $p$ to produce a point $q$.

Let $Q$ be a unit tangent vector based at $y$ that is tangent to $S^{d(x,y)}(x)$, the sphere of radius $d(x,y)$ centered at $x$. Let $\gamma_t\colon [a,b]\to M$ be a one-parameter family of geodesics parametrized by arc length so that $\gamma_{0}$ is the unit speed geodesic from $x$ to $y$, $\partial_t\gamma_t(b)\vert_{t=0}=Q$, $\gamma_t(b)$ is a path in $S^{d(x,y)}(x)$, and $\gamma_t(a)=x$ for all $t$. The variation $\gamma_t$ gives rise to a Jacobi field $Y$. Note that $Y(a)=0$, $Y(b)=Q$, and $Y$ is a normal Jacobi field.

Next, let $X$ be the Jacobi field along $\gamma_0$ defined by $X(b)=0$ and $\nabla_T X\vert_b=Y(b)$, where $T$ denotes $\dot{\gamma_0}$, i.e. the tangent to the curve $\gamma_0$. Such a field exists and has uniformly bounded norms because $\gamma_0$ is shorter than the injectivity radius. Let $\eta_t\colon [a,b]\to M$ be a one-parameter family of geodesics tangent to the field $X$ such that $\eta_t(b)=y$, $\eta_t$ is arc length parametrized, and $\eta_0=\gamma_0$. Note that each $\eta_t$ has length $d(x,y)$. Let $T$ now denote $\dot{\gamma}_{s,t}$, which give the tangent direction to each curve $\gamma_{s,t}$ in the variation.

Define $\gamma_{s,t}\colon [a,b]\to M$ to be the arc length parametrized geodesic between $\eta_s(a)$ and $\gamma_t(b)$. The variation $\gamma_{s,t}$ is a two parameter variation satisfying the hypotheses of Lemma \ref{lem:second_variation_of_length}. Consequently, we see that 
\begin{equation}\label{eqn:second_deriv_length}
\frac{d^2 \len(\gamma_{s,t})}{dsdt}=\langle \nabla_X Y,T\rangle\vert_a^b +\langle Y,\nabla_T X\rangle\vert_a^b.
\end{equation}

The first term may be rewritten as 
\begin{equation}\label{eqn:second_variation_term1}
\langle \nabla_X Y,T\rangle\vert_a^b=\nabla_X\langle Y,T\rangle\vert_a^b -\langle Y,\nabla_X T\rangle\vert_a^b.
\end{equation}
As $Y(a)=0$ and $X(b)=0$, the second term in \eqref{eqn:second_variation_term1} is zero. Similarly $\nabla_X\langle Y,T\rangle\vert_b=0$. We claim that $\nabla_X\langle Y,T\rangle\vert_a=0$ as well. To see this we claim that $Y=\partial_t\gamma_{s,t}\vert_{a}=0$ for all $s$. This is the case because $\gamma_{s,t}(a)$ is constant in $t$ as $\gamma_{s,t}(a)$ depends only on $s$. Thus $\langle Y,T\rangle\vert_a=0$. When we differentiate by $X$, we are differentiating along the path $\gamma_{s,0}(a)$. Thus $\nabla_X\langle Y,T\rangle\vert_a=0$ as $\langle Y,T\rangle$ is $0$ along this path. Thus $\langle \nabla_XY,T\rangle\vert_a^b=0$. Noting in addition that $Y(a)=0$, equation \eqref{eqn:second_deriv_length} simplifies to
\[
\frac{d^2 \len(\gamma_{s,t})}{dsdt}=\langle Y,\nabla_T X\rangle \vert_b.
\]
Hence as we defined $X$ so that $\nabla_T X\vert_b=Y(b)$,
\[
\frac{d^2 \len(\gamma_{s,t})}{dsdt}=\langle Y(b),Y(b)\rangle=\|Q\|=1.
\]
Note next that $\frac{d^2}{ds^2}\len(\gamma_{s,t})=0$ because the geodesics $\gamma_{s,0}$ all have the same length. Similarly, $\frac{d^2}{dt^2}\len(\gamma_{s,t})=0$.
Thus we have the Taylor expansion
\begin{equation}
\frac{d^2}{dsdt}\len(\gamma_{s,t})=d(x,y)+st+O(s^3,t^3).
\end{equation}
There exist $r_0>0$ and $C>0$ such that  for all $0\le s,t<r_0$,
\begin{equation}\label{eqn:taylor_expansion_of_len_done}
\len(\gamma_{s,t})\ge d(x,y)+Cst.
\end{equation}
Consider now the pairs of points $\gamma_{s,0}(a)$ and $\gamma_{0,t}(b)$.
We claim that if $p$ is of the form $p=\gamma_{0,t}(b)$ for some small $t$ then we may take $q=\gamma_{s,0}(a)$, where the choice of $s$ will be dictated by $\beta$. 

Note that
\[
d(\gamma_{s,0}(a),x)=s\|X(a)\|+O(s^2) \text{ and }d(\gamma_{0,t}(b),y)=t\|Y(b)\|+O(t^2).
\]
Hence there exists $s_0$ such that for $0<s,t<s_0$,
\begin{equation}\label{eqn:dgammatless2tXa}
d(\gamma_{s,0}(a),x)<2s\|X(a)\| \text{ and }d(\gamma_{0,t}(b),y)< 2t\|Y(b)\|.
\end{equation}
For any $\beta<\min\{2s_0\|X(a)\|,2r_0\|X(a)\|\}$, by \eqref{eqn:taylor_expansion_of_len_done} taking $s=\beta/2\|X(a)\|$ we obtain 
\[
d(\gamma_{s,0}(a),\gamma_{0,t}(b))\ge d(x,y)+t\beta C/2\|X(a)\|,
\]
which by \eqref{eqn:dgammatless2tXa} implies
\[
d(\gamma_{s,0}(a),\gamma_{0,t}(b))\ge d(x,y)+\frac{C}{4\|X(a)\|\|Y(b)\|}\beta d(\gamma_{0,t}(b),y).
\]
By \eqref{eqn:dgammatless2tXa}  and our choice of $s$
\[
d(\gamma_{s,0}(a),x)<\beta.
\]
Finally, $d(\gamma_{s,0}(a),y)=d(x,y)$ by the construction of the variation. Thus the conclusion of the lemma holds for the points $p=\gamma_{0,t}(b)$ and $q=\gamma_{s,0}(a)$.

We claim that this gives the full result. First, note that for all pairs of points $x$ and $y$ and choices of vectors $Q$ in our construction that $\|X(a)\|$ and $\|Y(b)\|$ are bounded above and below. This is because the distance minimizing geodesic from $X$ to $Y$ does not cross the cut locus. Similarly, the constants $C$, $r_0$, and $s_0$ may be uniformly bounded below over all such choices of $x$ and $y$ by compactness. Thus as all these constants are uniformly bounded independent of $x,y$ and $Q$, the above argument shows that for any pair $x$ and $y$ that there is a neighborhood $N$ of $y$ in $S^{d(x,y)}$ of uniformly bounded size, such that for any $p\in N$ there exists $q$ satisfying the conclusion of the lemma. This gives the result as any $p$ sufficiently close to $y$ such that $d(x,p)=d(x,y)$ lies in such a neighborhood $N$.
\end{proof}

The following lemma shows that if a diffeomorphism with small strain nearly fixes a large region, then that diffeomorphism is close to the identity.

\begin{lem}\label{lem:o1neighborhoodimpliesclose}
Let $(M,g)$ be a closed Riemannian manifold. Then there exists $r_0\in (0,1)$ such that for any $r',\beta\in (0,r_0)$, there exists $C>0$ such that if $f\in \Diff^1(M)$ and
\begin{enumerate}
\item
$d_{C^0}(f,\Id)\le r_0$,
\item
there exists a point $x\in M$ such that all $y$ with $d(x,y)<r'$ satisfy $d(y,f(y))\le \beta\le r_0$, and
\item
$\|f^*g-g\|=\eta\le r_0$,

\end{enumerate}
then 
\begin{equation}\label{eqn:dc0fIdbetaepTwo}
d_{C^0}(f,\Id)<C(\beta+\epTwo).
\end{equation}
\end{lem}

\begin{proof}
Let $r_1,C_1$ denote the $r$ and $C$ in Lemma \ref{lem:geometric_lemma}. Let $C_2$ be the constant in Lemma \ref{lem:lipschitz_constant}. There exists a constant $r_2$ such that for any $x,y\in M$ with $\inj(M)/3<d(y,x)<\inj(M)/2$ and any $z$ such that $d(y,z)<r_2$, then $d(y,\hat{z})<r_1$, where $\hat{z}$ is the radial projection of $z$ onto $S^{d(x,y)}(x)$. Let $r_0=\min\{r_1,r_2,\inj(M)/24\}$.

Suppose that $x\in M$ has the property that $d(x,z)<r$ implies $d(z,f(z))\le \beta$. Suppose that $y$ is a point such that $\inj(M)/3<d(y,x)<\inj(M)/2$. Let $\widehat{f(y)}$ be the radial projection of $f(y)$ onto $S^{d(x,y)}(x)$. 

By choice of $r_0\le r_2$, $d(y,f(y))<r_2$ and so $d(y,\widehat{f(y)})\le r_1$. Hence we may apply Lemma \ref{lem:geometric_lemma} with $\beta=r'$, $x=x$, $y=y$ and $p=\widehat{f(y)}$ to conclude that there exists a point $q\in M$ such that 
\begin{align}
d(q,y)&=d(x,y),\\
d(q,x)&<r',\\\label{eqn:distqhatfy}
d(q,\widehat{f(y)})&\ge d(x,y)+C_1d(y,\widehat{f(y)})r'.
\end{align}

Using the triangle inequality, we bound the left hand side of \eqref{eqn:distqhatfy} to find
\begin{equation}\label{eqn:triangle_bound}
 d(q,f(q))+d(f(q),f(y))+d(f(y),\widehat{f(y)})\ge d(q,\widehat{f(y)})\ge d(x,y)+C_1d(y,\widehat{f(y)})r'.
\end{equation}

First, as $d(q,x)<r'$ and points within $r'$ of $x$ do not move more than $\beta$,
\[
d(q,f(q))\le \beta.
\]
Second, by Lemma \ref{lem:lipschitz_constant}, as the distance between $q$ and $y$ is bounded above by $\inj(M)/2$, there exists $C_3$ such that
\[
d(f(q),f(y))\le d(q,y)(1+C_2\epTwo)=d(x,y)+C_3\epTwo.
\]
Similarly, as $\inj(M)/3<d(x,y)<\inj(M)/2$, Lemma \ref{lem:lipschitz_constant} implies the following two bounds
\begin{equation}\label{eqn:distxfyupper}
d(x,f(y))\le d(x,f(x))+d(f(x),f(y))\le \beta + d(x,y)+C_3\eta
\end{equation}
and similarly
\begin{equation}\label{eqn:distxfylower}
d(x,f(y))\ge d(x,y)-\beta-C_3\eta.
\end{equation}

For $w$ sufficiently close to $S^{d(x,y)}(x)$ we claim that the radial projection $\hat{w}$ is the point in $S^{d(x,y)}(x)$ that minimizes the distance to $w$. 
To see this we use that below the injectivity radius geodesics are the unique distance minimizing path between two points. 
There are two cases: if $d(x,w)>d(x,y)$ and there is some other point $w'\in S^{d(x,y)}(x)$ with $d(w',w)\le d(\hat{w},w)$, then the path from $x$ to $w'$ to $w$ along geodesics must be strictly longer than the geodesic path from $x$ directly to $\hat{w}$. 
If $d(x,w)<d(x,y)$ and $\hat{w}\neq w'\in S^{d(x,y)}(x)$, then one obtains two distance minimizing paths from $x$ to $S^{d(x,y)}(x)$ passing through $w$: the first along a single geodesic and the second from $x$ to $w$ and then from $w$ to $w'$. By the uniqueness of distance minimizing geodesics, the latter path must have length greater than $d(x,y)$ because it is not a geodesic. Thus $d(w,w')>d(w,\hat{w})$; a contradiction.

The estimates \eqref{eqn:distxfyupper} and \eqref{eqn:distxfylower} imply that $\abs{d(f(y),x)-d(x,y)}\le \beta+C_3\eta$. Thus the distance from $f(y)$ to $S^{d(x,y)}(x)$ is at most $\beta+C_3\eta$. By the previous paragraph, $\widehat{f(y)}$ is the point in $S^{d(x,y)}(x)$ that minimizes distance to $f(y)$. Thus
\begin{equation}\label{eqn:f_hat_dist}
d(f(y),\widehat{f(y)})\le \beta+C_3\epTwo.
\end{equation}
Thus, we obtain from equation \eqref{eqn:triangle_bound} 
\[
\beta+d(x,y)+C_3\epTwo+\beta+C_3\epTwo\ge d(x,y)+C_1d(y,\widehat{f(y)})r'.
\]
Thus 
\[
\frac{2\beta+2C_3\epTwo}{C_1r'}\ge d(y,\widehat{f(y)}).
\]
Hence
\[
d(y,f(y))\le 
d(f(y),\widehat{f(y)})+d(y,\widehat{f(y)})\le
\frac{2\beta+2C_3\epTwo}{C_1r'}+\beta+C_3\eta.
\]

Thus by introducing a new constant $C_4\ge 1$, we see that for any $y$ satisfying $\inj(M)/3<d(y,x)<\inj(M)/2$, that
\[
d(y,f(y))\le C_4(\beta+\epTwo).
\]
Note that the constant $C_4$ depends only on $r'$ and $(M,g)$. 

Consider a point $y$ where $(1/3+1/24)\inj(M)<d(x,y)<(1/2-1/24)\inj(M)$. Because  $r'<\inj(M)/24$ such a point $y$ has a neighborhood of size $r'$ on which points are moved at most distance $C_4(\beta+\epTwo)$ by $f$. Hence we may repeat the procedure taking $y$ as the new basepoint. Let $x$ be the given point in the statement of the lemma. Any point $q\in M$ may be connected to $x$ via a finite sequence of points $x=x_0,\ldots,x_n=q$ such that each consecutive pair of points in the sequence are at a distance between $(1/3+1/24)\inj(M)$ and $(1/2-1/24)\inj(M)$ apart. As $M$ is compact there is a uniform upper bound on the length of the shortest such sequence. If $N$ is a uniform upper bound on the length of such a sequence, the above argument shows that for all $q\in M$
\[
d(q,f(q))\le NC_4^N(\beta+\epTwo),
\]
which gives the result.
\end{proof}

The proof of Theorem \ref{thm:near_to_identity} consists of two steps. First a disk of uniform radius is produced on which $f$ nearly fixes points. Then Lemma \ref{lem:o1neighborhoodimpliesclose} is applied to this disk to conclude that $f$ is near to the identity.

\begin{proof}[Proof of Theorem \ref{thm:near_to_identity}.]
Let $r_1,C_1$ be denote the $r$ and $C$ in Lemma \ref{lem:fixed_neighborhood}, and let $r_2,C_2$ denote the $r$ and $c$ in Lemma \ref{lem:geometric_lemma}. 
There will be a constant $r_3>0$ introduced later when it is needed. Let $r_4$ denote the constant $r_0$ appearing in Lemma \ref{lem:o1neighborhoodimpliesclose}. We let $r=\min\{1,r_1,r_2,r_3,r_4,\inj(M)/24\}$. Let $C_3$ be the constant in Lemma \ref{lem:lipschitz_constant}. Let $\gamma\in (0,r)$ be given.

By Lemma \ref{lem:fixed_neighborhood}, for all $z$ such that $d(x,z)<\gamma$, 
\begin{equation}\label{eq:tight_initial_neighborhood}
d(z,f(z))\le C_1(\epOne\gamma +\gamma^2\epThree).
\end{equation}

Suppose that $y$ satisfies $\inj(M)/3<d(x,y)<\inj(M)/2$. Let $\widehat{f(y)}$ be the radial projection of $f(y)$ onto the sphere $S^{d(x,y)}(x)$. 

By Lemma \ref{lem:lipschitz_constant},
\[
d(x,y)(1-C_3\eta)\le d(f(x),f(y))\le d(x,y)(1+C_3\eta).
\]
As $f(x)=x$, this implies
\[
d(x,y)(1-C_3\eta)\le d(x,f(y))\le d(x,y)(1+C_3\eta).
\]
Hence as $d(x,y)$ is uniformly bounded above and below, there exists $C_4$ such that
\begin{equation}\label{eq:est11}
d(f(y),\widehat{f(y)})<C_4\epTwo.
\end{equation}

There exists $r_3>0$ such that if $\epTwo<r_3$, then $C_4\epTwo<r_2$. Hence by our choice of $r$, $d(y,\widehat{f(y)})<r_2$ and we may apply Lemma \ref{lem:geometric_lemma} with $\beta=\gamma$, $x=x$, $y=y$, $p=\widehat{f(y)}$ to deduce that there exists $q$ such that
\begin{align}
d(q,y)&=d(x,y),\\
d(q,x)&<\gamma,\\
d(q,\widehat{f(y)})&\ge d(x,y)+C_2d(y,\widehat{f(y)})\gamma.\label{eq:est14}
\end{align}

By Lemma \ref{lem:lipschitz_constant}, and using that $d(x,y)$ is bounded by $\inj(M)/2$, there exists $C_5$ such that
\begin{equation}\label{eq:est12}
d(f(q),f(y))\le d(q,y)(1+C_3\epTwo)\le d(x,y)+C_5\epTwo.
\end{equation}
By equation \eqref{eq:tight_initial_neighborhood},  as $d(q,x)<\gamma$,
\begin{equation}\label{eq:est13}
d(q,f(q))<C_1(\epOne\gamma+\epThree\gamma^2).
\end{equation}
Using the triangle inequality with \eqref{eq:est11}, \eqref{eq:est12}, \eqref{eq:est13}, to bound the left hand side of equation \eqref{eq:est14},  we obtain that 
\[
C_1(\epOne\gamma+\epThree\gamma^2)+d(x,y)+C_5\epTwo+C_4\epTwo\ge d(q,f(q))+d(f(q),f(y))+d(f(y),\widehat{f(y)})\ge d(x,y)+C_2d(y,\widehat{f(y)})\gamma.
\]
Moreover \eqref{eq:est11} gives the lower bound $d(y,\widehat{f(y)})>d(y,f(y))-C_4\epTwo$. We then obtain that 
\[
C_1(\epOne\gamma+\epThree\gamma^2)+C_5\epTwo+C_4\epTwo\ge C_2d(y,f(y))\gamma-C_2C_4\epTwo\gamma,
\]
and so
\[
\frac{C_1(\epOne\gamma+\epThree\gamma^2)+C_5\epTwo+C_4\epTwo +C_2C_4\epTwo\gamma}{C_2\gamma}\ge d(y,f(y)).
\]
The constants $C_1,\ldots,C_5$ are uniform over all $y$ satisfying $\inj(M)/3<d(x,y)<\inj(M)/2$. Thus there exists $C_6>0$ such that for all such $y$,
\begin{equation}\label{eqn:estimate_points_near_z}
C_6(\epTwo\gamma^{-1}+\epOne+\epThree\gamma)\ge d(y,f(y)).
\end{equation}
Suppose that $y$ is a point at distance $\frac{5}{12}\inj(M)$ from $x$. The above argument shows if $z$ satisfies $d(y,z)<\inj(M)/12$ then \eqref{eqn:estimate_points_near_z} holds with $y$ replaced by $z$, i.e.
\[
C_6(\epTwo\gamma^{-1}+\epOne+\epThree\gamma)\ge d(z,f(z)).
\]

 Define $\alpha$ by 
\begin{equation}
\alpha=C_6(\epTwo\gamma^{-1}+\epOne+\epThree\gamma).
\end{equation}
Assuming that $\alpha<r_4$, $z$ satisfies the second numbered hypothesis of Lemma \ref{lem:o1neighborhoodimpliesclose} with $\beta=\alpha$ and any $r'\le \inj(M)/12$.

There are then two cases depending on whether $\alpha>r_4$ or $\alpha\le r_4$. In the case that $\alpha\le r_4$, we apply Lemma \ref{lem:o1neighborhoodimpliesclose} with $x_0=z$, $r'=r/2$, and $\beta=\alpha$. This gives that there exists a $C_7$ depending only on $r/2$ such that 
\[
d_{C^0}(f,\Id)\le C_7(\epTwo\gamma^{-1}+\epOne+\epThree\gamma).
\]
If $\alpha>r_4$, then as $\kappa\le r_4$,
\[
d_{C^0}(f,\Id)\le \kappa \le r_4 \le \alpha = C_6(\epTwo\gamma^{-1}+\epOne+\epThree\gamma).
\]
Thus letting $C_8=\max\{C_6,C_7\}$, we have that
\[
d_{C^0}(f,\Id)\le C_8(\epTwo\gamma^{-1}+\epOne+\epThree\gamma),
\]
which gives the result.
\end{proof}

\subsection{Application to isotropic spaces: proof of Proposition \ref{lem:H_0_nearby_isometry} }\label{subsec:isotropic_app}

We now prove Proposition \ref{lem:H_0_nearby_isometry}, which is an application of Theorem \ref{thm:near_to_identity} to isotropic spaces.  
The idea of the proof is geometric.
We consider the diffeomorphism $I^{-1}f$. This diffeomorphism is small in $C^0$ norm, so there is an isometry $R_1$ that is close to the identity such that $R_1^{-1}I^{-1}f$ has a fixed point $x$. The differential of $R_1^{-1}I^{-1}f$ at $x$ is very close to preserving both the metric tensor and curvature tensor at $x$. We then use the following lemma to obtain an isometry $R_2$ that is nearby to $R_1^{-1}I^{-1}f$.

\begin{lem}\label{lem:helgason_lemma}
\cite[Ch. IV Ex. A.6]{helgason2001differential} Let $M$ be a simply connected Riemannian globally symmetric space or $\RP^n$. Then if $x\in M$ and $L\colon T_xM\to T_xM$ is a linear map preserving both the metric tensor at $x$ and the curvature tensor at $x$, then there exists $R\in \Isom(M)$ such that $R(x)=x$ and $D_xR=L$.
\end{lem}

We take the diffeomorphism in the conclusion of Proposition \ref{lem:H_0_nearby_isometry} to equal $IR_1R_2$. We then apply Theorem \ref{thm:near_to_identity} to deduce that $R_2^{-1}R_1^{-1}I^{-1}f$ is near the identity diffeomorphism.  It follows that $IR_1R_2$ is near to $f$.
Before beginning the proof, we state some additional lemmas.

\begin{lem}\label{lem:subspace_distances}
Suppose that $V_1$ and $V_2$ are two subspaces of a finite dimensional inner product space $W$. Then there exists $C>0$ such that if $x\in W$, then 
\[
d(x,V_1\cap V_2)<C(d(x,V_1)+d(x,V_2)).
\]
\end{lem}
\begin{lem}\label{lem:near_to_curvature_preserving}
Suppose that $R$ is a tensor on $\R^n$. Let $\stab(R)$ be the subgroup of $\GL(\R^n)$ that stabilizes $R$ under pullback. Then there exist $C,D>0$ such that if $L\colon \R^n\to \R^n$ is an invertible linear map and $\|L-\Id\|<D$, then 
\[
d_{\GL(\R^n)}(L,\stab(R))\le C\|L^*R-R\|.
\]
\end{lem}

\begin{proof}
Let $\mf{s}$ be the Lie algebra to $\stab(R)$. Then consider the map $\phi$ from  $\mf{gl}$ to the tensor algebra on $\R^n$ given by
\[
w\mapsto \exp(w)^*R-R.
\]
We may write $w=v+v^{\perp}$, where $v\in \mf{s}$ and $v\in \mf{s}^{\perp}$.
Because $\phi$ is smooth it has a Taylor expansion of the form
\begin{equation}\label{eqn:taylor_exp_phi}
\phi(tv+tv^{\perp})=0+tAv+tBv^{\perp}+O(t^2).
\end{equation}
Note that $A$ is zero because $v\in \mf{s}$. We claim that $B$ is injective. For the sake of contradiction, suppose $Bv^{\perp}=0$ for some $v^{\perp}\in \mf{s}^{\perp}$. Then $\exp(tv^{\perp})^*R-R=O(t^2)$. But then
\begin{align*}
\exp(v^{\perp})^*R-R&=\sum_{i=0}^{n-1} \exp((i+1)v^{\perp}/n)^*R-\exp(i v^{\perp}/n)R\\
&=\sum_{i=0}^{n-1} \exp(i v^{\perp}/n)^*(\exp(v^{\perp}/n)^*R-R)\\
&=O(1/n).
\end{align*}
And hence $\exp(v^{\perp})^*R-R=0$, which contradicts $v^{\perp}\notin \mf{s}$. Thus $B$ is an injection and hence by Taylor's theorem for small $v^{\perp}$ there exists $C_1$ such that
\begin{equation}\label{eqn:perp_lower_bound_on_tensor_movement}
\|\exp(v^{\perp})^*R-R\|\ge C_1\|v^{\perp}\|.
\end{equation}
By using the Taylor expansion \eqref{eqn:taylor_exp_phi} and noting that $A=0$ there, we obtain from equation \eqref{eqn:perp_lower_bound_on_tensor_movement} that there exists $C_2>0$ such that 
\begin{equation}\label{eqn:lower_bound_on_tensor_movement2}
\|\exp(w)^*R-R\|\ge C_2\|v^{\perp}\|.
\end{equation}
It then follows there exists a neighborhood $N$ of $\Id\in \GL(\R^n)$ such that $\stab(R)\cap N$ is the image of a disc $D\subset \mf{s}$ under $\exp$. Write $\mf{gl}=\mf{s} \oplus \mf{s}^{\perp}$ as a vector space. Thus as $\exp$ is bilipschitz in a neighborhood of  $0\in \mf{gl}$ there exists $C_3$ such that if we write $w\in D$ as $w=v+v^{\perp}$, where $v\in \mf{s}$ and $v^{\perp}\in \mf{s}^{\perp}$, then 
\begin{equation}\label{eqn:nearness_to_stab_exp}
C_3^{-1}\|v^{\perp}\|\le d_{\GL(\R^n)}(\exp(w),\exp(D))\le C_3\|v^{\perp}\|.
\end{equation}
As $\stab(R)\cap N=\exp(D)$, for all $w$ in a smaller neighborhood $D'\subset D$, the middle term above is comparable to $d_{GL(\R^n)}(\exp(w),\stab(R))$.

Thus combining \eqref{eqn:nearness_to_stab_exp} with \eqref{eqn:lower_bound_on_tensor_movement2}, we obtain
\[
d_{\GL(\R^n)}(\exp(w),\stab(R))\le C_2^{-1}C_3\|\exp(w)^*R-R\|.
\]
This gives the result as $\exp$ is a surjection onto a neighborhood of $\Id\in \GL(\R^n)$.
\end{proof}

The following lemma is immediate from \cite[Thm. IV.3.3]{helgason2001differential}, which explicitly describes the isometries of globally symmetric spaces.

\begin{lem}\label{lem:small_distances_small_isometries}
Suppose that $M$ is a closed globally symmetric space. There exists $C>0$ such that if $x,y\in M$, then there exists an isometry $I\in \Isom(M)^{\circ}$ such that $I(x)=y$ and $d_{C^0}(I,\Id)\le Cd(x,y)$. As $\Isom(M)^{\circ}$ is compact, it follows that for each $k$ there exists a constant $C_k$ such that one choose $I$ with $d_{C^k}(I,\Id)\le C_kd(x,y)$.
\end{lem}

We also use the following lemma, which is the specialization of Lemma  \ref{lem:near_to_curvature_preserving} to the metric tensor.

\begin{lem}\label{lem:nearly_isom_implies_near_SO}
Suppose that $V$ is a finite dimensional inner product space with metric $g$ of dimension $d$. There exists a neighborhood $U$ of $\Id\in \GL(V)$ and a constant $C$ such that if $L\in U$  then 
\[
d_{\GL(V)}(L,\SO(V))\le C\|L^*g-g\|,
\]
where $\GL(V)$ is endowed with the right-invariant Riemannian metric it inherits from the inner product space $V$.
\end{lem}

We now prove the proposition.

\begin{proof}[Proof of Proposition \ref{lem:H_0_nearby_isometry}.]

Pick $0<\lambda<1$ and a small $\tau$ such that 
\begin{equation}\label{eqn:choice_of_lambda_and_tau}
\frac{\lambda}{2}-\lambda\tau>\frac{1}{2}-\sigma \text{ and } \sigma>\tau>0.
\end{equation}
We also assume without loss of generality that $\ell\ge 3$.
 By Lemma \ref{lem:bundle_section_interpolation} there exist $k_0$ and $\epsilon_0>0$ such that if $s$ is a smooth section of the bundle of symmetric $2$-tensors over $M$, $\|s\|_{C^{k_0}}\le 4$, and $\|s\|_{H^0}\le \epsilon_0$, then $\|s\|_{C^{\ell}}\le \|s\|_{H^0}^{1-\tau}$. 
Choose $k$ such that 
\begin{equation}\label{eqn:choice_of_k}
k>\max\{k_0,\frac{\ell}{1-\lambda}\}.
\end{equation}
In addition, there are positive numbers $\epsilon_{1},\ldots,\epsilon_7$ that will be introduced when needed in the proof below. We define
\[
r=\min\{\epsilon_0,\epsilon_{1}^{1/(1-\tau)},\epsilon_2,\ldots,\epsilon_7,1\}.
\]

Let $\epsilon_{1}>0$ be small enough that for any $x\in M$, if $L\colon T_xM\to T_xM$ is invertible and $\|L^*g-g\|\le \epsilon_{1}$, then the conclusion of Lemma \ref{lem:nearly_isom_implies_near_SO} holds for $L$.

Let $\eta=\|f^*g-g\|_{H^0}$ and $\varepsilon_2=d_{C^2}(f,I)$. Consider the norm $\|f^*g-g\|_{C^{k_0}}$. As $d_{C^k}(I,f)$ is uniformly bounded, we see that $\|f^*g-g\|_{C^{k-1}}$ is uniformly bounded. In fact, there exists $\epsilon_2>0$ such that if $d_{C^k}(I,f)<\epsilon_2$, then $\|f^*g-g\|_{C^{k-1}}\le 4$. As $r<\epsilon_0$, the discussion in the first paragraph of the proof implies that 
\begin{equation}\label{eqn:f_pullsback_metric_close}
\|f^*g-g\|_{C^3}\le \eta^{1-\tau}.
\end{equation}
Note that this is less than $\epsilon_{1}$ by the choice of $r$.

For $x\in M$, we may consider the Lie group $\GL(T_xM)$ as well as its Lie algebra $\mf{gl}$. There exists $\epsilon_3>0$ such that restricted to the ball of radius $\epsilon_3$ about $0\in \mf{gl}$, the Lie exponential, which we denote by $\exp$, is bilipschitz with constant $2$.

Let $x\in M$ be a point that is moved the maximum distance by $I^{-1}f$. 
By Lemma 
\ref{lem:small_distances_small_isometries}, there exists a constant $D_k>0$ independent of $x$ and an isometry $R_1$ such that $R_1(x)=I^{-1}f(x)$ and $d_{C^k}(R_1,\Id)<D_kd(x,I^{-1}f(x))$. Let $h=R_1^{-1}I^{-1}f$ and note that $h$ fixes $x$. Note that there exists $\epsilon_4>0$ such that if $d_{C^k}(f,I)<\epsilon_4$, then by the previous sentence $R_1$ can be chosen so that $d_{C^k}(R_1,\Id)$ is small enough that 
\begin{equation}\label{eqn:Dxh_nearid}
\|D_xh-\Id\|\le C_0\varepsilon_2.
\end{equation}

We claim that $D_xh$ is near a linear map of $T_xM$ that preserves both the metric tensor and the curvature tensor. Let $\SO(T_xM)$ be the group of linear maps preserving the metric tensor on $T_xM$ and let $G$ be the group of linear maps preserving the curvature tensor on $T_xM$. Both of these are subgroups of $\GL(T_xM)$.
By the sentence after equation \eqref{eqn:f_pullsback_metric_close}, $D_xh$ pulls back the metric on $T_xM$ to be within $\epsilon_{1}$ of itself. 
Thus by Lemma \ref{lem:small_distances_small_isometries}, there exists a uniform constant $C_1$ such that $D_xh$ is within distance $C_1\eta^{1-\tau}$ of $\SO(T_xM)$. Again by equation \eqref{eqn:f_pullsback_metric_close}, we have that $\|h^*g-g\|_{C^3}\le \eta^{1-\tau}$. 
In particular, as the curvature tensor is defined by the second derivatives of the metric, this implies by Lemma \ref{lem:near_to_curvature_preserving} that there exists a constant $C_2$ such that $D_xh$ is within distance $C_2\eta^{1-\tau}$ of $G$.

The previous paragraph shows that there exists $C_3$ such that $D_xh$ is within distance $C_3\eta^{1-\tau}$ of both $\SO(T_xM)$ and $G$. Consider now the exponential map of $\GL(T_xM)$. As before, let $\mathfrak{gl}$ denote the Lie algebra of $\GL(T_xM)$. Let $H=\exp^{-1}(D_xh)\in \GL(T_xM)$.
 Note that this preimage is defined as $D_xh$ is near to the identity. Let $\mathfrak{so}$ be the Lie algebra to $\SO(T_xM)$ and let $\mf{g}$ be the Lie algebra to $G$. As both $SO(T_xM)$ and $G$ are closed subgroups and $\exp$ is bilipschitz we conclude that the distance both between $H$ and each of $\mathfrak{so}$ and $\mathfrak{g}$ is bounded above by $2C_3\eta^{1-\tau}$. Thus by Lemma \ref{lem:subspace_distances}, there exists $C_4$ such that $H$ is at most distance $C_4\eta^{1-\tau}$ from $\mathfrak{g}\cap \mathfrak{so}$. Let $X\in \GL(T_xM)$ be an element of $\mathfrak{g}\cap \mathfrak{so}$ minimizing the distance from $H$ to $\mf{g}\cap \mf{so}$. There exists $\epsilon_5>0$ such that if $\eta\le \epsilon_5$ then $C_4\eta^{1-\tau}<\epsilon_3$.  Hence as $r<\epsilon_5$, the same bilipschitz estimate on the Lie exponential gives
\begin{equation}\label{eqn:dexpX_D_xh}
d(\exp(X),D_xh)\le 2C_4\eta^{1-\tau}.
\end{equation}
 Note that $\exp(X)\in SO(T_xM)\cap G$. By Lemma \ref{lem:helgason_lemma}, there exists an isometry $R_2$ of $M$ such that $R_2$ fixes $x$ and $D_xR_2=\exp(X)$. In fact, because of equation \eqref{eqn:Dxh_nearid} and because $X$ is within distance $C_4\eta^{1-\tau}$ of $H$, we may bound the norm of $X$ and hence deduce that there exists $C_5$ such that
\begin{equation}\label{eqn:sizes_of_R2}
d_{C^k}(R_2,\Id)\le C_5(\varepsilon_2+\eta^{1-\tau}).
\end{equation}

 The map $R$ in the conclusion of the proposition will be $IR_1R_2$. We must now check that $R=IR_1R_2$ satisfies estimates \eqref{eqn:H_0_lemma_low_reg_est} and \eqref{eqn:H_0_lemma_high_reg_est}. The former is straightforward: \eqref{eqn:H_0_lemma_low_reg_est} follows from \eqref{eqn:sizes_of_R2} combined with knowing that $R_1$ was constructed so that $d(R_1,\Id)\le D'\varepsilon_2$ for some uniform $D'>0$.

 Let $h_2=R_2^{-1}h$. The map $h_2$ has $x$ as a fixed point. 
 There exists $C_6>0$ such that the following four estimates hold:
\begin{align}
\|D_xh_2-\Id\|&\le C_6\eta^{1-\tau},\label{eqn:block1}\\
\|h_2^*g-g\|_{C^3}&\le \eta^{1-\tau},\label{eqn:block2}\\
d_{C^2}(h_2,\Id)&\le C_6(\varepsilon_2+\eta^{1-\tau}),\label{eqn:block3}\\
d_{C^k}(h_2,\Id)&\le C_6(\eta^{1-\tau}+d_{C^k}(I,f)).\label{eqn:block4}
\end{align}
The first two estimates above are immediate from equations \eqref{eqn:dexpX_D_xh} and \eqref{eqn:f_pullsback_metric_close}, respectively. The third and fourth follow from an estimate on $C^k$ compositions, Lemma \ref{lem:C^k_composition_estimate}, and equation \eqref{eqn:sizes_of_R2}.

Let $r_0$ be the cutoff $r$ appearing in Theorem \ref{thm:near_to_identity}. Note that there exists $\epsilon_6>0$ such that if $d_{C^k}(f,I)<\epsilon_6$ and $\eta<\epsilon_6$, then the right hand side of each of inequalities \eqref{eqn:block1} through \eqref{eqn:block4} is bounded above by $r_0$.
Hence as $r<\epsilon_6$ we apply Theorem \ref{thm:near_to_identity} to $h_2$ to conclude that there exists $C_7$ such that for all $0<\gamma<r_0$,
\[
d_{C^0}(\Id,h_2)<C_7(\eta^{1-\tau}+C_6(\varepsilon_2+\eta^{1-\tau})\gamma+\eta^{1-\tau}\gamma^{-1}).
\]
But $h_2=R_2^{-1}R_1^{-1}I^{-1}f$, so 
\begin{equation}\label{eq:C0_estimate}
d_{C^0}(R,f)<C_8(\eta^{1-\tau}+C_6(\varepsilon_2+\eta^{1-\tau})\gamma+\eta^{1-\tau}\gamma^{-1}).
\end{equation}

We now obtain the high regularity estimate, equation \eqref{eqn:H_0_lemma_high_reg_est}, via interpolation. By similarly moving the isometries from one slot to the other, \eqref{eqn:block4} gives that
\begin{equation}\label{eq:C_k_construction_estimate}
d_{C^k}(R,f)<C_9(\eta^{1-\tau}+d_{C^k}(I,f)).
\end{equation}
There exists $\epsilon_7>0$ such that if $d_{C^k}(I,f)<\epsilon_7$ and $\eta<\epsilon_7$, then the right hand side of equation \eqref{eq:C_k_construction_estimate} is at most $1$.

We now apply the interpolation inequality in Lemma \ref{lem:ck_interpolation_inequality} and interpolate between the $C^0$ and $C^k$ distance to estimate $d_{C^{\ell}}(R,f)$. Write $\ell=(1-\lambda')k$ for some $\lambda'$ and note that $1>\lambda'>\lambda$ by \eqref{eqn:choice_of_k}. We use the estimate in equation \eqref{eq:C0_estimate} to estimate the $C^0$ norm and use $1$ to estimate the $C^k$ norm, which we may do because $r<\epsilon_7$. Thus there exists $C_{10}$ such that for $0<\gamma<r_0$,
\begin{equation}\label{eqn:c0_ck_interpolated}
d_{C^\ell}(R,f)<C_{10}(\eta^{1-\tau}\gamma^{-1}+\varepsilon_2\gamma)^{\lambda'}.
\end{equation} 

 Note that there exists $C_{11}>0$ such that $\|f^*g-g\|_{H^0}\le C_{11}\varepsilon_2$. Consequently, there exists a constant $C_{13}$ such that $C_{12}\sqrt{\eta/\varepsilon_2}$ is less than the cutoff $r_0$. We take $\gamma$ to equal $C_{12}\sqrt{\eta/\varepsilon_2}$ in equation \eqref{eqn:c0_ck_interpolated}, which gives
\begin{equation}
d_{C^\ell}(R,f)<C_{13}(\eta^{1/2-\tau}\varepsilon_2^{1/2}+\eta^{1/2}\varepsilon_2^{1/2})^{\lambda'}<C_{14}(\eta^{\lambda/2-\lambda\tau}\varepsilon_2^{\lambda/2}+\eta^{\lambda/2}\varepsilon_2^{\lambda/2}).
\end{equation} 
Hence by our choice of $\lambda$ and $\tau$ in equation \eqref{eqn:choice_of_lambda_and_tau} and because $\eta<r<1$,
\begin{equation}
d_{C^\ell}(R,f)<C_{15}\eta^{1/2-\sigma}\varepsilon_2^{1/2-\sigma},
\end{equation} 
which establishes equation \eqref{eqn:H_0_lemma_high_reg_est} and finishes the proof.
\end{proof}

\section{KAM Scheme}\label{sec:KAM_scheme}

In this section we develop the KAM scheme and prove that it converges. A KAM scheme is an iterative approach to constructing a conjugacy between two systems in the $C^{\infty}$ setting. We begin by discussing the smoothing operators that will be used in the scheme. Then we state a lemma, Lemma \ref{lem:KAM_step}, that summarizes the results of performing a step in the scheme. We then prove in Theorem \ref{thm:KAM_converges} that by iterating the single KAM step that we obtain the convergence needed for this theorem. We conclude the section with a final corollary of the KAM scheme which gives an asymptotic relationship between the top exponent, the bottom exponent, and the sum of all the exponents.

\subsection{One step in the KAM scheme}

In the KAM scheme, we begin with a tuple of isometries $(R_1,...,R_m)$ and a nearby tuple of diffeomorphisms $(f_1,...,f_m)$. We want to find a diffeomorphism $\phi$ such that for all $i$, $\phi^{-1}f_i\phi=R_i$. However, such a $\phi$ may not exist.

We will then attempt construct a conjugacy, $\phi$ that has the following property. 
Let $\wt{f}_i$ equal $\phi^{-1}f_i\phi$. 
If we consider the tuple $(\wt{f}_1,...,\wt{f}_m)$ and $(R_1,....,R_m)$, we can arrange that the error term, $\mc{U}$, in Proposition \ref{prop:taylor_expansion_lambda_k}, is small. Once we know that the error term is small, the estimate in Proposition \ref{prop:taylor_expansion_lambda_k} shows that small Lyapunov exponents imply that each $\wt{f}_i$ has small strain. Then using Proposition \ref{lem:H_0_nearby_isometry}, small strain implies that there exist $R_i'$ that each $\wt{f}_i$ is near to that $R_i'$. We then apply the same process to the tuples $(\wt{f}_1,...,\wt{f}_m)$ and $(R_1',\ldots,R_m')$.

The previous paragraph contains the core idea of the KAM scheme. Following this scheme, one encounters a common technical difficulty inherent in KAM arguments: regularity. In our case, this problem is most crucial when we construct the conjugacy $\phi$. There is not a single choice of $\phi$, but rather a family depending on a parameter $\lambda$. The parameter $\lambda$ controls how smooth $\phi$ is. Larger values of $\lambda$ give less regular conjugacies. We refer to this as a \emph{conjugation of cutoff} $\lambda$; the formal construction of the \emph{conjugation of cutoff} $\lambda$ appears in the proof in Lemma \ref{lem:KAM_step} which also gives estimates following from this construction. The $n$th time we iterate this procedure we will use a particular value $\lambda_n$ as our cutoff. The proof of Theorem \ref{thm:KAM_converges} shows how to pick the sequence $\lambda_n$ so that the procedure converges.

We now introduce the smoothing operators. Suppose that $M$ is a closed Riemannian manifold. As before, let $\Delta$ denote the Casimir Laplacian on $M$ as in subsection \ref{subsec:diophantine_sets_tameness}.
As $\Delta$ is self adjoint, it decomposes the space of $L^2$ vector fields into subspaces depending on the particular eigenvalue associated to that subspace. We call these subspaces $H_{\lambda}$. For a vector field $X$, we may write $X=\sum_{\lambda} X_{\lambda}$, where $X_{\lambda}\in H_{\lambda}$ is the projection of $X$ onto the $\lambda$ eigenspace of $\Delta$. All of the eigenvalues of $\Delta$ are positive. By removing the components of $X$ that lie in high eigenvalue subspaces, we are able to smooth $X$. Let $\mc{T}_{\lambda} X=\sum_{\lambda'<\lambda} X_{\lambda'}$ equal the projection onto the modes strictly less than $\lambda$ in magnitude. Let $\mc{R}_{\lambda} X=\sum_{\lambda'\ge \lambda} X_{\lambda'}$ be the projection onto the modes of magnitude greater than or equal to $\lambda$. Then $X=\mc{T}_{\lambda}X+\mc{R}_{\lambda} X$.

We record two standard estimates which may be obtained by application of the Sobolev embedding theorem. For $s\ge 0$, there exists a constant $C_s>0$ such that for any $\overline{s}\ge s$ and any $C^{\infty}$ vector field $X$ on $M$,
\begin{equation}\label{eq:lower_spectrum_regularity_estimate}
\|\mc{T}_{\lambda}X\|_{C^{\overline{s}}}\le C_s\lambda^{k_3+(\overline{s}-s)/2}\|X\|_{C^s},
\end{equation}
\begin{equation}\label{eq:upper_spectrum_regularity_estimate}
\|\mc{R}_{\lambda} X\|_{C^s}\le C_s\lambda^{k_3-(\overline{s}-s)/2}\|X\|_{C^{\overline{s}}}.
\end{equation}

The smoothing operators and the above estimates on them are useful because without smoothing certain estimates appearing in the KAM scheme become unusable. One may see this by considering what happens in the proof of Lemma \ref{lem:KAM_step} if one removes the smoothing operator $\mc{T}_{\lambda}$ from equation \eqref{eqn:V_defn}.

The proof of the following lemma should be compared with \cite[Sec. 3.4]{dolgopyat2007simultaneous}

\begin{lem}\label{lem:KAM_step}
Suppose that $(M^d,g)$ is a closed isotropic Riemannian manifold other than $S^1$.  There exists a natural number $l_0$ such that for $\ell>l_0$ and any $(C,\alpha,n_0)$ the following holds. For any sufficiently small $\sigma>0$, there exist a constant $r_\ell>0$ and numbers $k_0,k_1,k_2$ such that for any $s>\ell$ and any $m$ there exist constants $C_{s,\ell},r_{s,\ell}>0$ such that the following holds. Suppose that $(R_1,...,R_m)$ is a $(C,\alpha,n_0)$-Diophantine tuple with entries in $\Isom(M)$ and $(f_1,...,f_m)$ is a collection of $C^{\infty}$ diffeomorphisms of $M$. Suppose that the random dynamical system generated by $(f_1,...,f_m)$ has stationary measures with arbitrarily small in magnitude bottom exponent. Write $\varepsilon_k$ for $\max_i d_{C^k}(f_i,R_i)$. If $\lambda\ge 1$ is a number such that 
\begin{equation}\label{eq:KAM_step_assumption_2}
\lambda^{k_{0}}\varepsilon_{l_0}\le r_{\ell}
\end{equation}
and
\begin{equation}\label{eq:KAM_step_assumption_1}
\lambda^{k_1-s/4}\varepsilon_s+\varepsilon_{l_0}^{3/2}<r_{s,\ell},
\end{equation}
then there exists a smooth diffeomorphism $\phi$ and a new tuple $(R_1',...,R_m')$ of isometries of $M$ such that for all $i$ setting $\wt{f}_i=\phi f_i\phi^{-1}$, we have
\begin{align}
d_{C^{\ell}}(\wt{f}_i,R_i')&\le C_{s,\ell}(\lambda^{k_1-s/10}\varepsilon_s^{1-\sigma}+\varepsilon_{l_0}^{9/8}),\label{eqn:low_norm_KAM_estimate}\\
d_{C^0}(R_i,R_i')&\le C_{s,\ell}(\varepsilon_{l_0}+(\lambda^{k_1-s/4}\varepsilon_s+\varepsilon_{l_0}^{3/2})^{1-\sigma}),\label{eqn:distance_to_new_rotations}\\
d_{C^s}(\wt{f}_i,R_i')&\le C_{s,\ell}\lambda^{k_2}\varepsilon_s, \text{ and} \label{eqn:high_norm_KAM_estimate}\\
d_{C^s}(\phi,\Id)&\le C_{s,\ell}\lambda^{k_2}\varepsilon_s. \label{eqn:size_of_conj_estimates}
\end{align}

The diffeomorphism $\phi$ is called a conjugation of cutoff $\lambda$.
\end{lem}

\begin{proof}
As in equation \eqref{eqn:linearized_error}, let $Y_i$ be the smallest vector field on $Y_i$ satisfying $\exp_{R(x)}Y_i(x)=f_i(x)$. Let $\mc{L}$ be the operator on vectors fields defined by $\mc{L}(Z)=m^{-1}\sum_{i=1}^m (R_i)_*Z$ as in Proposition \ref{prop:tame_coboundaries2}.
Let
\begin{equation}\label{eqn:V_defn}
V\coloneqq  -(1-\mc{L})^{-1}\left(\frac{1}{m}\sum_i\mc{T}_{\lambda} Y_i\right)
\end{equation}
and let  $\wt{f}_i=\psi_{V}f_i\psi_V^{-1}$. Let $\wt{\varepsilon}_k=\max_{i}d_{C^k}(\wt{f}_i,R_i)$ and let $\wt{Y}_i$ be the pointwise smallest vector field such that $\exp_{R(x)} \wt{Y}_i(x)=\wt{f}_i(x)$. By Proposition \ref{prop:conj_error_field},  for a $C^1$ small vector field $V$,
\begin{equation}\label{eq:conj_error_field}
\wt{Y}_i= Y_i+V-R_iV+Q(Y_i,V),
\end{equation}
where $Q$ is quadratic in the sense of Definition \ref{defn:quadratic}. By Proposition \ref{prop:tame_coboundaries}, we see that $\|V\|_{C^k}\le C_k\varepsilon_{k+\alpha}$ for some fixed $\alpha$. There exist $\beta,D_1$ such that $\|Q(Y_i,V)\|_{C^k}\le D_k\varepsilon_{k+\beta}^2$. 
By estimating the terms in equation \eqref{eq:conj_error_field}, it follows that for each $k>0$ if $\varepsilon_{k+\alpha+\beta}<1$ then there exists a constant $D_{2,k}$ such that 
\begin{equation}\label{eq:wt_f_R_error_Ck}
d_{C^k}(\wt{f}_i,R_i)<D_{2,k}\varepsilon_{k+\alpha+\beta}.
\end{equation}

Let $\mu$ be an ergodic stationary measure on $M$ for the tuple $(\wt{f}_1,...,\wt{f}_m)$ as in the statement of the lemma. We now apply Proposition \ref{prop:taylor_expansion_lambda_k} with $r=d-1,d$ and recall why the hypotheses of that proposition are satisfied. First, by our assumption that $M$ is isotropic, $\Isom(M)^{\circ}$ acts transitively on $M$ and $\Gr_1(M)$. We have also assumed the tuple $(R_1,\ldots,R_m)$ is Diophantine. The nearness of $(\wt{f}_1,\ldots,\wt{f}_m)$ to $(R_1,\ldots,R_m)$ is guaranteed by equation \eqref{eq:wt_f_R_error_Ck}, a sufficiently small choice of $r_{\ell}$, and sufficiently large choice of $l_0$ by equation \eqref{eq:KAM_step_assumption_2} as $\lambda\ge 1$.
Thus by applying Proposition \ref{prop:taylor_expansion_lambda_k} to the conjugated system, there exists $k_1$ such that, in the language of that proposition:
\[
\Lambda_r(\mu)=\frac{-r}{2dm}\sum_{i=1}^m \int_M \|E_C^{\wt{f}_i}\|^2+\frac{r(d-r)}{(d+2)(d-1)m}\sum_{i=1}^m \int_M\|E_{NC}^{\wt{f}_i}\|^2\,d\vol+\int_{G_r(M)}\mc{U}(\psi_r)d\vol+O(\|\wt{Y}\|_{C^{k_{1}}}^3),
\]
where $\psi_r(x)=\frac{1}{m}\sum_{i=1}^m \ln \det (D_x\wt{f}_i\mid E_x)$ and $\mc{U}$ is defined in Proposition \ref{prop:taylor_expansion_of_haar}.

Pick a sequence of ergodic stationary measures $\mu_n$ so that $\abs{\lambda_d(\mu_n)}\to 0$.
Subtracting the expression for $\Lambda_{d-1}(\mu_n)$ from the expression for $\Lambda_d(\mu_n)$, we obtain that 
\begin{align}\label{eqn:lyap_est_1}
\begin{split}
\lambda_d(\mu_n)=\Lambda_{d}(\mu_n)-\Lambda_{d-1}(\mu_n)=&\frac{-1}{2dm}\sum_{i=1}^{m} \int_M \|E_C^{\wt{f}_i}\|^2\,d\vol+\frac{-(d-1)}{(d+2)(d-1)m}\sum_{i=1}^m \int_M \|E_{NC}^{\wt{f}_i}\|^2\,d\vol \\ &-\int_{\Gr_{d-1}(M)} \mc{U}(\psi_{d-1})\,d\vol+\int_{\Gr_d(M)} \mc{U}(\psi_d)\,d\vol+O(\|\wt{Y}\|^3_{C^{k_{1}}}).
\end{split}
\end{align}
Write $\Gr_r(R)$ for the map on $\Gr_r(M)$ induced by $R$. Write $\mathbf{Y}_i$ for the shortest vector field on $\Gr_r(M)$ such that $\exp_{\Gr_r(R_i)(x)} \mathbf{Y}_i=\Gr_r(\wt{f}_i)(x)$. By Lemma \ref{lem:error_on_lifts}, for each $k$ there exists $C_{1,k}$ such that 
\[
\left \|\sum_{i=1}^m \mathbf{Y}_i\right\|_{C^k}\le C_{1,k}\left(\left\|\sum_{i=1}^m \wt{Y}_i\right\|_{C^{k+1}}+\wt{\varepsilon}_{k+1}^2\right).
\]
Hence by the above line and the final estimate in Proposition \ref{prop:taylor_expansion_of_haar}  there exists $k_2$ such that
\begin{equation}\label{eqn:u_psi_r_estimate}
\abs{\int_{\Gr_r(M)} \mc{U}(\psi_r)\,d\vol}\le C_2\|\psi_r\|_{C^{k_2}}\left(\left\|\frac{1}{m}\sum_{i=1}^m \wt{Y}_i\right\|_{C^{k_{2}}}+\|\wt{Y}_i\|_{C^{k_2}}^2\right).
\end{equation}
The term $\|\psi_r\|_{C^{k_2}}$ is bounded by a constant times $\wt{\varepsilon}_{k_2}$. By using equation \eqref{eq:conj_error_field} we may rewrite the second term appearing in the product in equation \eqref{eqn:u_psi_r_estimate}.
\begin{align*}
\frac{1}{m}\sum_{i=1}^m \wt{Y}_i&=\frac{1}{m}\sum_i Y_i +-(1-\mc{L})^{-1}(\frac{1}{m}\sum_{i} \mc{T}_{\lambda} Y_i)-\frac{1}{m}\sum_i (R_i)_*(-(1-\mc{L})^{-1})(\mc{T}_{\lambda}Y_i)+\frac{1}{m}\sum_i Q(Y_i,V)\\
&=\frac{1}{m}\sum_i \mc{R}_{\lambda}Y_i+\frac{1}{m}\sum_i \mc{T}_{\lambda}Y_i - (1-\mc{L})(1-\mc{L})^{-1}(\frac{1}{m}\sum_i \mc{T}_{\lambda} Y_i)+\frac{1}{m}\sum_{i} Q(Y_i,V)\\
&=\frac{1}{m}\sum_i \mc{R}_{\lambda}Y_i+\frac{1}{m}\sum_i \mc{T}_{\lambda}Y_i -\frac{1}{m}\sum_i \mc{T}_{\lambda}Y_i+\frac{1}{m}\sum_i Q(Y_i,V)\\
&=\frac{1}{m}\sum_i \mc{R}_{\lambda}Y_i+\frac{1}{m}\sum_i Q(Y_i,V)\\
\end{align*}
By equation \eqref{eq:upper_spectrum_regularity_estimate}, there exists $k_3$ such that for all $s\ge 0$:
\[
\|R_{\lambda}Y_i\|_{C^1}\le C_{3,s}\lambda^{k_3-s/2}\|Y_i\|_{C^s}.
\]
As the $Q$ term is quadratic, there exist $\ell_2$, $C_4$ such that
\[
\|Q(Y_i,V)\|_{C^{k_2}}\le C_4\|Y_i\|_{C^{\ell_2}}\|V\|_{C^{\ell_2}}=C_4\|Y_i\|_{C^{\ell_2}}\|(1-\mc{L})^{-1}(\mc{T}_{\lambda} Y_i)\|_{C^{\ell_2}}\le C_5\varepsilon_{\ell_3}^2
\]
for some $\ell_3$ by Proposition \ref{prop:tame_coboundaries2}. Thus
\[
\left\|\frac{1}{m}\sum_i \wt{Y}_i\right\|_{C^{k_2}}\le C_{6,s}(\lambda^{k_3-s/2}\varepsilon_s+\varepsilon_{\ell_3}^2).
\]
Finally, by equation \eqref{eq:wt_f_R_error_Ck} we have that $\|\wt{Y}_i\|_{C^{k_2}}\le C_7 \varepsilon_{\ell_3}$ as before. Let $\ell_4=\max\{\ell_3,k_2+\alpha+\beta\}$.
Applying all of these estimates to \eqref{eqn:u_psi_r_estimate} gives
\begin{equation}\label{eqn:U_estimate}
\abs{\int_{\Gr_r(M)} \mc{U}(\psi_r)\,d\vol} 
\le C_{8,s}\varepsilon_{k_2}(\lambda^{k_3-s/2}\varepsilon_s+\varepsilon_{\ell_4}^2).
\end{equation}
By taking $\ell_5>\max\{k_1+\alpha+\beta,k_2,\ell_4\}$, using that $\lambda_d(\mu_n)\to 0$,\footnote{Note that we did not need $\lambda(\mu_n)\to 0$ in order to conclude equation \eqref{eqn:intermediate_nonconformality_estimate}. It suffices to know that there $\mu$ such that $\lambda_d(\mu)$ is comparable to the right hand side of \eqref{eqn:U_estimate}. This observation is the essence of the proof of Theorem \ref{thm:taylor_expansion_of_top_and_bottom}.}  and combining equations \eqref{eqn:U_estimate} and \eqref{eqn:lyap_est_1} we obtain for $s\ge 0$ that there exists $C_{9,s}$ such that
\begin{equation}\label{eqn:intermediate_nonconformality_estimate}
C_{9,s}(\lambda^{k_3-s/2}\varepsilon_s\varepsilon_{\ell_5}+\varepsilon_{\ell_5}^3)\ge\frac{1}{2dm}\sum_{i=1}^{m} \int_M \|E_C^{\wt{f}_i}\|^2\,d\vol+\frac{(d-1)}{(d+2)(d-1)m}\sum_{i=1}^m \int_M \|E_{NC}^{\wt{f}_i}\|^2\,d\vol.
\end{equation}
Note that the coefficients on each of the strain terms are positive. If $s>\ell_5$, then by taking square roots, we see that there exist constants $C_{10,s}$ such that for each $i$
\begin{equation}
C_{10,s}(\lambda^{k_3/2-s/4}\varepsilon_s+\varepsilon_{\ell_5}^{3/2})\ge \|\wt{f}_i^*g-g\|_{H^0}.
\end{equation}

We now give a naive estimate on the higher $C^s$ norms under the assumption that $\varepsilon_1$ is bounded by a constant $\epsilon_1>0$. To begin, by combining equation \eqref{eq:lower_spectrum_regularity_estimate} and Proposition \ref{prop:tame_coboundaries} we see that there exists $\alpha>0$ such that for each $s$ there exists $D_{3,s}$ such that $\|V\|_{C^{s}}\le  D_{3,s}\lambda^{\alpha}\varepsilon_{s}$. 
Hence by Lemma \ref{lem:ck_inverse_est2}, both  $d_{C^{s}}(\psi_V,\Id)$ and $d_{C^{s}}(\psi_V^{-1},\Id)$ are bounded by $D_{4,s}\lambda^{\alpha}\varepsilon_{s}$. This establishes equation \eqref{eqn:size_of_conj_estimates}.

Now applying the composition estimate from Lemma \ref{lem:C^k_composition_estimate}, we find that  assuming $\lambda \ge 1$:
\begin{align*}
d_{C^s}(f\circ \psi_V^{-1},R)&\le C_{11,s}(d_{C^s}(f,R)+d_{C^s}(\psi_V^{-1},\Id))\\
&\le C_{12,s}(\varepsilon_s+\lambda^{\alpha}\varepsilon_s)\\
&\le C_{13,s}(\lambda^{\alpha}\varepsilon_s).
\end{align*}
We then apply the other estimate in Lemma \ref{lem:C^k_composition_estimate}, to find:
\begin{align*}
d_{C^s}(\psi_V\circ f\circ \psi_V^{-1},R)&\le C_{11,s}(d_{C^s}(\psi_V,\Id)+d_{C^s}(f\circ \psi_V^{-1},R))\\
&\le C_{14,s}(\lambda^{\alpha}\varepsilon_s+\lambda^{\alpha}\varepsilon_s)\\
&\le C_{15,s}\lambda^{\alpha}\varepsilon_s.
\end{align*}
Hence under an assumption of the type in equation \eqref{eq:KAM_step_assumption_2}, namely $\varepsilon_1<\epsilon_1$, we may conclude
\begin{equation}\label{eqn:naive_est_wtf_R}
d_{C^s}(\wt{f}_i,R)\le C_{15,s}\lambda^{\alpha}\varepsilon_s,
\end{equation}
which establishes equation \eqref{eqn:high_norm_KAM_estimate}.

We now apply Proposition \ref{lem:H_0_nearby_isometry} to this system. Let $k_{\sigma}$ and $r_{\sigma}$ be the $k$ and $r$ in Proposition \ref{lem:H_0_nearby_isometry} for a given choice of $\sigma$ and our fixed $\ell$.  In preparation for the application of the lemma, we record some basic estimates:
\begin{enumerate}
\item
By combining equation \eqref{eq:lower_spectrum_regularity_estimate} and Proposition \ref{prop:tame_coboundaries2} as before, we see that there exists $\ell_6$  such that
\begin{equation}
d_{C^2}(\wt{f}_i,R_i)\le \varepsilon_{\ell_6}.
\end{equation}
\item
From the previous discussion we also have
\[\|\wt{f}_i^*g-g\|_{H^0}\le C_{10,s}(\lambda^{k_3/2-s/4}\varepsilon_s+\varepsilon_{\ell_5}^{3/2}).
\]
\item
We also need the $C^{k_{\sigma}}$ estimate
\[
d_{C^{k_{\sigma}}}(\wt{f}_i,R)\le C_{15,k_{\sigma}}\lambda^{\alpha}\varepsilon_{k_{\sigma}}.
\]

\end{enumerate}
Hence if
\begin{equation}
C_{15,k_{\sigma}}\lambda^{\alpha}\varepsilon_{k_{\sigma}}<r_{\sigma}
\end{equation}
and 
\begin{equation}
C_{10,s}(\lambda^{k_3/2-s/4}\varepsilon_s+\varepsilon_{\ell_5}^{3/2})\le r_{\sigma},
\end{equation}
then by Proposition \ref{lem:H_0_nearby_isometry} and the previous estimates
there exist $C_6$ and isometries $R_i'$ such that
\begin{equation}\label{eqn:C_ell_est_56}
d_{C^{\ell}}(\wt{f}_i,R_i')\le C_{16,s}(\lambda^{k_3/2-s/4}\varepsilon_s+\varepsilon_{\ell_5}^{3/2})^{1/2-\sigma}\varepsilon_{\ell_6}^{1/2-\sigma}
\end{equation}
and
\begin{equation}\label{eqn:distance_to_new_rotations_2}
d_{C^0}(R_i',R_i)<C_{17,s}(\varepsilon_{\ell_6}+(\lambda^{k_3/2-s/4}\varepsilon_s+\varepsilon_{\ell_5}^{3/2})^{1-\sigma}).
\end{equation}
Let $\ell_7=\max\{\ell_5,\ell_6\}$. If $s>\ell_7$, then equation \eqref{eqn:C_ell_est_56} implies
\[
d_{C^{\ell}}(\wt{f}_i,R_i')\le C_{16,s}(\lambda^{k_4-s/9}\varepsilon_s^{1-2\sigma}+\varepsilon_{\ell_7}^{5/4-(5/2)\sigma}),
\]
which yields equation \eqref{eqn:low_norm_KAM_estimate} under the assumption that $\sigma>0$ is sufficiently small.
Note that equation \eqref{eqn:distance_to_new_rotations_2} establishes equation \eqref{eqn:distance_to_new_rotations}.  Thus we are done as we have established these estimates assuming only bounds of the type appearing in equations \eqref{eq:KAM_step_assumption_2} and \eqref{eq:KAM_step_assumption_1}.
\end{proof}

\begin{rem}
In the above lemma, we could instead have assumed that there exist stationary measures for which both the top exponent and the sum of all the exponents were arbitrarily small and concluded the same result. The reason being if we had considered $\Lambda_1-\Lambda_d$ in equation \eqref{eqn:lyap_est_1}, the coefficients of the strain terms would still have the same sign and so we could conclude the same result. By related modifications, one can produce many other formulations of the main result in \cite{dolgopyat2007simultaneous} that require other hypotheses on the Lyapunov exponents.
\end{rem}

\subsection{Convergence of the KAM scheme}

In this section we prove the main linearization theorem.  It is helpful to note that the approach to this theorem is somewhat different from the classical approach to KAM type results. In a classical argument, one might typically linearize the problem at a target isometric system and then find a solution to the linearized problem. In our case, while we are able to linearize the problem, the resulting linearized problem does not obviously have any solution. Consequently we must give dynamical and geometric arguments that show that a related type of averaged linearized problem can be solved and that solving this averaged problem is indeed helpful. This then allows us to make progress in the KAM scheme by conjugating the system closer to an isometric one. In particular, note that in our case we do not know from the outset which isometric system our random system will ultimately be conjugate to.

\newtheorem*{thm:mainthm}{\bf Theorem \ref{thm:KAM_converges}}
\begin{thm:mainthm}
\emph{Let $M^d$ be a closed isotropic Riemannian manifold other than $S^1$. There exists $k_0$ such that if $(R_1,...,R_m)$ is a tuple of isometries of $M$ such that the subgroup of $\Isom(M)$ generated by this tuple contains $\Isom(M)^{\circ}$, then there exists $\epsilon_{k_0}>0$ such that the following holds. Let $(f_1,...,f_m)$ be a tuple of $C^{\infty}$ diffeomorphisms satisfying $\max_{i} d_{C^{k_0}}(f_i,R_i)<\epsilon_{k_0}$. Suppose that there exists a sequence of ergodic stationary measures $\mu_n$ for the random dynamical system generated by $(f_1,...,f_m)$ such that $\abs{\lambda_d(\mu_n)}\to 0$, then there exists $\psi\in \Diff^{\infty}(M)$ such that for each $i$ the map $\psi f_i\psi^{-1}$ is an isometry of $M$ and lies in the subgroup of $\Isom(M)$ generated by $(R_1,\ldots,R_m)$.}
\end{thm:mainthm}

Before giving the proof, we sketch briefly the argument, which is typical of arguments establishing the convergence of a KAM scheme. In a KAM scheme where one wishes to show that some sequence of objects $h_n$ converges there are often two parts.  The first part of the proof is an inductive argument obtaining a sequence of estimates by the repeated application of  the KAM step, which in our case is Lemma \ref{lem:KAM_step}. The second half of the proof checks that the repeated application of the KAM step is valid by showing that we never leave the neighborhood of its validity and then checks that the procedure is converging in $C^{\infty}$.

In the first part, one inductively produces a sequence of estimates by iterating a KAM step. The estimates produced usually come in two forms: a single good estimate in a low norm and bad estimates in high norms. The low regularity estimate probably looks like $\|h_n\|_{C^0}\le N^{-(1+\tau)^n}$ where $\tau>0$, while for every $s$ one has a high regularity estimate like $\|h_n\|_{C^s}\le N^{(1+\tau)^n}$. A priori, the $h_n$ become superexponentially $C^0$ small, yet might be diverging in higher $C^s$ norms. To remedy this situation one then interpolates between the low and high norms by using an equality derived from Lemma \ref{lem:ck_interpolation_inequality}. In this case such an inequality for the objects $h_n$ might assert something like
\[
\|h_n\|_{C^{\lambda\cdot 0+(1-\lambda)s}}\le C_s\|f\|_{C^0}^{\lambda}\|f\|_{C^{s}}^{1-\lambda}.
\]
If $\lambda$ is sufficiently close to $1$ and $s$ is sufficiently large, a brief calculation then implies that the $C^{(1-\lambda)s}$ norm is also super exponentially small. By changing $s$ and $\lambda$ one then obtains convergence in $C^{\infty}$.

\begin{proof}[Proof of Theorem \ref{thm:KAM_converges}.]
The proof is by a KAM convergence scheme. To begin we introduce the Diophantine condition we will use. By Proposition \ref{prop:diophantine_equivalence2}, $(R_1,...,R_m)$ is $(C',\alpha',n')$-Diophantine for some $C',\alpha'>0$ and is stably so. By stability, there exist $(C,\alpha,n)$ and a $C^0$ neighborhood $\mc{U}$ of 
$(R_1,...,R_m)$ such that any tuple in $\mc{U}$ is also $(C,\alpha,n)$-Diophantine.  Hence if $(R_1',...,R_m')\in \mc{U}$, then the coefficients $C_{i,s}$ appearing in Lemma \ref{lem:KAM_step} are uniform over all of these tuples. Assuming we do not leave the set $\mc{U}$, the constants appearing in Lemma \ref{lem:KAM_step} will be uniform. We check this at the end of the proof in the discussion surrounding equation \eqref{eqn:rotation_distance_estimate}.

We now show that there exists a sequence of cutoffs $\lambda_n$ so that if we repeatedly apply Lemma \ref{lem:KAM_step} with the cutoff $\lambda_n$ on the $n$th time we apply the Lemma, then the resulting sequence of conjugates converges and the hypotheses of Lemma \ref{lem:KAM_step} remain satisfied. Given such a sequence $\lambda_n$ the convergence scheme is run as follows. Let $(f_{1,1},\ldots,f_{m,1})=(f_1,\ldots,f_m)$ and let $(R_{1,1},\ldots,R_{m,1})=(R_1,\ldots,R_m)$. Given $(f_{1,n-1},\ldots,f_{m,n-1})$ and $(R_{1,n-1},\ldots,R_{m,n-1})$ we apply Lemma \ref{lem:KAM_step} with cutoff $\lambda=\lambda_n$ to produce a diffeomorphism $\phi_n$ and a tuple of isometries that we denote by $(R_{1,n},\ldots,R_{m,n})$. We set $f_{i,n}=\phi_nf_{i,n-1} \phi_n^{-1}$ to obtain a new tuple of diffeomorphisms $(f_{1,n},\ldots,f_{m,n})$. We write $\psi_n$ for $\phi_n\circ \phi_{n-1}\circ \cdots \circ \phi_1$, so that $f_{i,n}=\psi_n\circ f_i\circ\psi_n^{-1}$.
Let $\varepsilon_{k,n}=\max_i d_{C^k}(f_{i,n},R_{i,n})$.

We now show that such a sequence of cutoffs $\lambda_n$ exist. Let $\sigma$ be a small positive number and let $l_0$ and $\epsilon_{l_0}$ be as in Lemma \ref{lem:KAM_step}. Let $k_0,k_1,k_2, r_{\ell},C_{s,\ell},r_{s,\ell}$ be as in Lemma \ref{lem:KAM_step} as well.  
To show that such a sequence of cutoffs $\lambda_n$ exists we must also provide a fixed choice of $s,\ell$ for the application of Lemma \ref{lem:KAM_step}. We will first show that the scheme converges in the $C^{l_0}$ norm and then bootstrap to get $C^{\infty}$ convergence. Fix some arbitrary $\ell>l_0$. The choice of $\ell$ does not matter in the sequel because we only will consider estimates on the $l_0$ norm.
We will choose $s$ such that
\begin{equation}\label{eqn:s_lower_bound}
s>\ell.
\end{equation}
Further, if $s$ is sufficiently large and $\tau$ is sufficiently small, then we can pick $\alpha$ such that 
\begin{equation}\label{eqn:defn_of_alpha}
\frac{2+\tau}{s/4-k_1}<\alpha<\min\{1/k_0,\tau/k_2\}
\end{equation}
So, we increase $s$ if needed and choose such a $\tau$ satisfying 
\begin{equation}\label{eqn:tau_bound}
1/8>\tau>0.
\end{equation}
Pick $s,\alpha,\tau$ so that each of equations \eqref{eqn:s_lower_bound}, \eqref{eqn:defn_of_alpha}, \eqref{eqn:tau_bound} is satisfied.

Let $\lambda_n=N^{\alpha(1+\tau)^n}$ for some $N$ we choose later. We will show that with this choice of cutoff at the $n$th step that the KAM scheme converges. In order to show this, we show the following two estimates hold inductively given a choice of  sufficiently large $N$:
\begin{align}
\varepsilon_{l_0,n}&\le N^{-(1+\tau)^n}\tag{H1}\label{eqn:est_epsilonl0}\\
\varepsilon_{s,n}&\le N^{(1+\tau)^n}\tag{H2}\label{eqn:est_epsilon_s}\\
\max_i\{d_{C^0}(R_{i,n},R_{i,1})\}&\le \sum_{i=1}^{n}N^{-\frac{1}{2}(1+\tau)^i}\tag{H3}\label{eqn:R_dist}.
\end{align}
This involves two arguments. The first argument shows that there is a sufficiently large $N$ such that if we have these estimates for $n$, then the hypotheses of Lemma \ref{lem:KAM_step} are satisfied. The second argument is the actual induction, which checks that if equations \eqref{eqn:est_epsilonl0} and \eqref{eqn:est_epsilon_s} hold for $n$ then they also hold for $n+1$, i.e. we apply Lemma \ref{lem:KAM_step} and then deduce \eqref{eqn:est_epsilonl0} and \eqref{eqn:est_epsilon_s} for $n+1$ from this.

We begin by checking that for all sufficiently large $N>0$ and any $n\in \N$ if \eqref{eqn:est_epsilonl0}, \eqref{eqn:est_epsilon_s}, and \eqref{eqn:R_dist} are satisfied, then the hypotheses of Lemma \ref{lem:KAM_step} are satisfied as well. To begin, as the summation in \eqref{eqn:R_dist} is summable, for all sufficiently large $N$, we are assured that $(R_{1,n},\ldots,R_{m,n})$ lies in $\mc{U}$. The first numbered hypothesis of Lemma \ref{lem:KAM_step} is equation \eqref{eq:KAM_step_assumption_2}:
\[
\lambda_n^{k_0}\varepsilon_{l_0,n}\le r_\ell.
\]
Given the choice of $\lambda_n$, if equations \eqref{eqn:est_epsilonl0} and \eqref{eqn:est_epsilon_s} hold it suffices to have
\[
N^{\alpha k_0(1+\tau)^n}N^{-(1+\tau)^n}<r_\ell,
\]
which holds for $N$ sufficiently large and all $n$ by our choice of $\alpha$. The other hypothesis of Lemma \ref{lem:KAM_step}, equation \eqref{eq:KAM_step_assumption_1}, requires
that 
\[
\lambda_n^{k_1-s/4}\varepsilon_{s,n}+\varepsilon_{l_0,n}^{3/2}<r_{s,\ell}.
\]
Given equations \eqref{eqn:est_epsilonl0} and \eqref{eqn:est_epsilon_s} and our choice of $\lambda_n$ it suffices to have
\[
N^{\alpha(k_1-s/4)(1+\tau)^n}N^{(1+\tau)^n}+N^{-\frac{3}{2}(1+\tau)^n}<r_{s,\ell}.
\]
Our choice of $s$ and $\alpha$ implies that $\alpha(k_1-s/4)<-1$, hence the above inequality holds for sufficiently large $N$. Thus the two hypotheses of Lemma \ref{lem:KAM_step} follow from equations \eqref{eqn:est_epsilonl0} and \eqref{eqn:est_epsilon_s}. Thus we may apply Lemma \ref{lem:KAM_step} given \eqref{eqn:est_epsilonl0}, \eqref{eqn:est_epsilon_s}, \eqref{eqn:R_dist}, and our choice of $N$. 

We now proceed to the inductive argument. What we will show is that for all $N$ sufficiently large, if we now require that our perturbation is small enough that \eqref{eqn:est_epsilonl0} and \eqref{eqn:est_epsilon_s} hold for $n=1$ and our choice of $N$ we check that we may continue applying Lemma \ref{lem:KAM_step} and that these estimates as well as \eqref{eqn:R_dist} continue to hold. Note that \eqref{eqn:R_dist} is trivial when $n=1$. We must then check that equations \eqref{eqn:est_epsilonl0}, \eqref{eqn:est_epsilon_s}, and \eqref{eqn:R_dist}  are satisfied for $n+1$ given they hold for $n$. 
 By the previous paragraph, we are free to apply the estimates from Lemma \ref{lem:KAM_step} as long as $N$ is sufficiently large.

We now check that equation \eqref{eqn:est_epsilonl0} holds for $n+1$.
By equation \eqref{eqn:low_norm_KAM_estimate}, we obtain that 
\[
\varepsilon_{l_0,n+1}\le C_{s,\ell}(\lambda_n^{k_1-s/10}\varepsilon_{s,n}^{1-\sigma}+\varepsilon_{l_0,n}^{9/8}).
\]
By applying equations \eqref{eqn:est_epsilonl0} and \eqref{eqn:est_epsilon_s} to 
 each term on the right it suffices to show
\begin{equation}\label{eqn:low_reg_evaluated_est}
C_{s,\ell}(N^{\alpha(k_1-s/10)(1+\tau)^n}N^{(1-\sigma)(1+\tau)^n}+N^{-9/8(1+\tau)^n})<N^{-(1+\tau)^{n+1}}.
\end{equation}
By our choice of $s$, $\alpha$, and $\tau$, the lower bound in equation \eqref{eqn:defn_of_alpha} implies that \[\alpha(k_1-s/10)+(1-\sigma)<-(1+\tau).\] In addition, by equation \eqref{eqn:tau_bound}, $-9/8<-(1+\tau)$. Thus for sufficiently large $N$ the left hand side of equation \eqref{eqn:low_reg_evaluated_est} is bounded above by $N^{-(1+\tau)^{n+1}}$.

Next we check equation \eqref{eqn:est_epsilon_s} holds for $n+1$.  By equation \eqref{eqn:high_norm_KAM_estimate},
\[
\varepsilon_{s,n+1}\le C_{s,\ell}\lambda_n^{k_2}\varepsilon_{s,n}.
\]
Hence, 
\[
\varepsilon_s\le C_{s,\ell}N^{k_2\alpha(1+\tau)^n}N^{(1+\tau)^n},
\]
By equation \eqref{eqn:defn_of_alpha}, $1+k_2\alpha<1+\tau$ and hence, assuming $N$ is sufficiently large, the right hand side is bounded by $N^{(1+\tau)^{n+1}}$, which shows equation \eqref{eqn:est_epsilon_s} is satisfied.

We now check that \eqref{eqn:R_dist}. This follows easily by the application of equation \eqref{eqn:distance_to_new_rotations},  which gives
\begin{equation}\label{eqn:rotation_distance_estimate}
d_{C^0}(R_{i,n},R_{i,n+1})\le C_{s,\ell}(\varepsilon_{l_0,n}+(\lambda_n^{k_1-s/4}\varepsilon_{s,n}+\varepsilon_{l_0,n}^{3/2})^{1-\sigma})
\end{equation}
Applying \eqref{eqn:est_epsilonl0} and \eqref{eqn:est_epsilon_s} and the definition of $\lambda_n$ to estimate the right hand side of equation \eqref{eqn:rotation_distance_estimate}, we find that for the $\gamma$ given in \eqref{eqn:R_dist} and $N$ sufficiently large that
\begin{equation}\label{eqn:isom_seq_cauchy}
d_{C^0}(R_{i,n},R_{i,n+1})\le N^{-\frac{1}{2}(1+\tau)^n},
\end{equation}
and \eqref{eqn:R_dist} holds for $n+1$.

We have now finished the induction but not the proof. We have shown that there exists a sequence $\lambda_n$ and a choice $s,\alpha,\ell,\tau,N$, so that if the initial conditions of the scheme are satisfied then we may iterate indefinitely and be assured of the estimates in equations \eqref{eqn:est_epsilonl0}, \eqref{eqn:est_epsilon_s}, \eqref{eqn:R_dist} at each step. 
We must now check that the conjugacies $\psi_n$ are converging in $C^{\infty}$ and that the tuples $(R_{1,n},\ldots, R_{m,n})$ are converging.  The latter is immediate because by \eqref{eqn:isom_seq_cauchy} this is a Cauchy sequence. In fact, we chose $N$ large enough that we never leave $\mc{U}$, hence the limit is in $\mc{U}$.
As the group of isometries of $M$ is $C^0$ closed and the distance of the tuples $(f_{1,n},\ldots,f_{m,n})$ from a tuple of isometries is converging to $0$, it follows that $(f_{1,n},\ldots,f_{m,n})$ is converging to a tuple of isometries.
To show that the $\psi_n$ converge in $C^{\infty}$, we obtain for every $s$ an estimate on $d_{C^s}(\phi_n,\Id)$. By a similar induction to that just performed, the estimate \eqref{eqn:size_of_conj_estimates} implies
\[
d_{C^s}(\phi_n,\Id)\le C_sN^{(1+\tau)^n}.
\]
Let $j>0$ be  an integer. By Lemma \ref{lem:ck_interpolation_inequality_diffeos}, interpolating with $\lambda=1-1/10$ between the $C^{l_0}$ distance and the $C^{jl_0}$ distance of $\phi_n$ to the identity gives
\[
d_{C^{.9l_0+(j/10)l_0}}(\phi_n,\Id)\le C_{j}N^{-.9(1+\tau)^n}N^{.1 (1+\tau)^n}=C_{j}N^{-.8(1+\tau)^n}.
\]
Thus by increasing $j$, we see that there exists $\tau'>0$ such that for each $C^s$ norm
\[
d_{C^{s}}(\phi_n,\Id)< C_s'N^{-(1+\tau')^n}.
\]
The previous line is summable in $n$. Hence we can apply Lemma \ref{lem:C_k_composition_convergence} to obtain convergence of sequence of the $\psi_n=\phi_n\circ \cdots \circ \phi_1$ in the $C^s$ norm for each $s$ and thus $C^{\infty}$ convergence.

Thus we see that we have simultaneously conjugated each $f_i$ into $\Isom(M)$. In order to obtain the full theorem, we must be assured that $\psi^{-1}f_i\psi$ lies in the subgroup of $\Isom(M)$ generated by $(R_1,\ldots,R_m)$. Note that $\Isom(M)/\Isom(M)^{\circ}$ is a finite group and that $\psi$ is homotopic to the identity by construction. Thus we see that the image of the group generated by $(\psi^{-1}f_1\psi,\ldots,\psi^{-1}f_m\psi)$ in $\Isom(M)/\Isom(M)^{\circ}$ is the same as the image of the group generated by $(R_1,\ldots,R_m)$. By our choice of $N$, $(\psi^{-1}f_1\psi,\ldots,\psi^{-1}f_m\psi)$ is in $\mc{U}$ and thus generates $\Isom(M)^{\circ}$. Thus the original tuple and the new one generate the same subgroup of $\Isom(M)$ and we are done.
\end{proof}

\subsection{Taylor expansion of Lyapunov exponents}

In order to recover Dolgopyat and Krikorian's Taylor expansion in the setting of isotropic manifolds, we would need to apply Proposition \ref{prop:taylor_expansion_lambda_k} for each $0\le r\le \dim M$. However, one of the hypotheses of Proposition \ref{prop:taylor_expansion_lambda_k} is that $\Isom(M)^{\circ}$ acts transitively on $\Gr_r(M)$. In Proposition \ref{prop:isom_not_transitive}, we see that unless $M$ is $S^n$ or $\RP^n$, $\Isom(M)$ does not act transitively on $\Gr_r(M)$ for $r\neq 1$ or $d-1$. Despite Proposition \ref{prop:isom_not_transitive}, we are able to obtain a partial result: the greatest and least Lyapunov exponents are symmetric about the ``average" Lyapunov exponent $\frac{1}{d}\Lambda_d(\mu)$.

\begin{thm}\label{thm:taylor_expansion_of_top_and_bottom}
Suppose that $M^d$ is a closed isotropic manifold other than $S^1$ and that $(R_1,...,R_m)$ is a subset of $\Isom(M)$ that generates a subgroup of $\Isom(M)$ containing $\Isom(M)^{\circ}$. Suppose that $(f_1,...,f_n)$ is a collection of $C^{\infty}$ diffeomorphisms of $M$. Then there exists $k_0$ such that if $\mu$ is an ergodic stationary measure of the random dynamical system generated by the $(f_1,...,f_m)$, then
\begin{equation}\label{eqn:taylor_expansion1d}
\abs{\lambda_1(\mu)-(-\lambda_d(\mu)+\frac{2}{d}\Lambda_d(\mu))}\le o(1)\abs{\lambda_d(\mu)}.
\end{equation}
where the $o(1)$ term goes to $0$ as $\max_{i} d_{C^{k_0}}(f_i,R_i)\to 0$. The $o(1)$ term depends only on $(R_1,...,R_m)$.
\end{thm}

\begin{proof}
By Theorem \ref{thm:KAM_converges}, there are two cases: either $(f_1,\ldots,f_m)$ is conjugate to isometries or it is not. In the isometric case equation \eqref{eqn:taylor_expansion1d} is immediate,  so we may assume that there there is an ergodic stationary measure $\mu$ with $\lambda_d(\mu)$ non-zero. The proof that follows is then essentially an observation about what happens when the KAM scheme is run on a system that has a measure with such a non-zero Lyapunov exponent.  If we run the KAM scheme without assuming that $(f_1,...,f_m)$ has a measure with zero exponents, we can keep running the scheme until the non-trivial exponents prevent us from continuing. At a certain point in the procedure, the non-trivial exponents cause a certain inequality fail. Using the failed inequality then gives the result.

We now give the details. Fix an ergodic stationary measure $\mu$ and consider equation \eqref{eqn:lyap_est_1} appearing in the KAM step:

\begin{align}
\begin{split}
\lambda_d(\mu)=&\frac{-1}{2dm}\sum_{i=1}^{m} \int_M \|E_C^{\wt{f}_i}\|^2\,d\vol+\frac{-(d-1)}{(d+2)(d-1)m}\sum_{i=1}^m \int_M \|E_{NC}^{\wt{f}_i}\|^2\,d\vol \\ &-\int_{\Gr_{d-1}(M)} \mc{U}(\psi_{d-1})\,d\vol+\int_{\Gr_d(M)} \mc{U}(\psi_d)\,d\vol+O(\|\wt{Y}\|^3_{C^{k_{1}}}).
\end{split}
\end{align}

The above equation allows us to use that the exponent $\lambda_d$ is small in magnitude. In the KAM step, we proceed from this estimate by estimating the $\|\wt{Y}\|^3_{C^{k_1}}$ term as well as the $\mc{U}$ terms. Equation \eqref{eqn:U_estimate} and the choice of $\ell_5$ imply that these terms satisfy:
\begin{equation}\label{eqn:U_y_estimate}\abs{\int_{\Gr_{d-1}(M)} \mc{U}(\psi_{d-1})\,d\vol-\int_{\Gr_d(M)} \mc{U}(\psi_d)\,d\vol+O(\|\wt{Y}\|^3_{C^{k_{1}}})}
\le C_{8,s}\varepsilon_{\ell_5}(\lambda^{k_3-s/2}\varepsilon_s+\varepsilon_{\ell_5}^2).
\end{equation}
Hence as long as 
\begin{equation}\label{eqn:lambdad_condition}
\abs{\lambda_d(\mu)}<(C_{9,s}-C_{8,s})(\varepsilon_{\ell_5}(\lambda^{k_3-s/2}\varepsilon_s+\varepsilon_{\ell_5}^2))
\end{equation}
the proof of Lemma \ref{lem:KAM_step} may proceed to equation \eqref{eqn:intermediate_nonconformality_estimate} even if there is not a sequence of measures $\mu_n$ such that $\abs{\lambda_d(\mu_n)}\to 0$. Hence we may continue running the KAM scheme 
until equation \eqref{eqn:lambdad_condition} fails to hold.

Suppose that we iterate the KAM scheme until equation \eqref{eqn:lambdad_condition} fails.  We consider the estimates available in the KAM scheme at the step of failure. By applying Proposition \ref{prop:taylor_expansion_lambda_k} with $r$ equal to $1$, $d$, and $d-1$, we obtain:

\begin{align}
\begin{split}
\Lambda_1(\mu)=&\frac{-1}{2dm}\sum_{i=1}^{m} \int_M \|E_C^{\wt{f}_i}\|^2\,d\vol+\frac{(d-1)}{(d+2)(d-1)m}\sum_{i=1}^m \int_M \|E_{NC}^{\wt{f}_i}\|^2\,d\vol+\int_{G_1(M)} \mc{U}(\psi_1)\,d\vol+\\
&O(\|\wt{Y}\|^3_{C^{k_1}})\\
\Lambda_{d-1}(\mu)=&\frac{-(d-1)}{2dm}\sum_{i=1}^{m} \int_M \|E_C^{\wt{f}_i}\|^2\,d\vol+\frac{(d-1)}{(d+2)(d-1)m}\sum_{i=1}^m \int_M \|E_{NC}^{\wt{f}_i}\|^2\,d\vol+\int_{G_{d-1}(M)} \mc{U}(\psi_{d-1})\,d\vol
+\\
&O(\|\wt{Y}\|^3_{C^{k_1}})\\
\Lambda_{d}(\mu)=&\frac{-d}{2dm}\sum_{i=1}^{m} \int_M \|E_C^{\wt{f}_i}\|^2\,d\vol+\int_{G_d(M)} \mc{U}(\psi_d)\,d\vol+O(\|\wt{Y}\|^3_{C^{k_1}})
\end{split}
\end{align}

Write $\mc{U}_i$ as shorthand for the term $\int_{\Gr_i(M)}\mc{U}(\psi_i)\,d\vol$. Then,

\begin{align}
\lambda_1(\mu)-(-\lambda_d(\mu)+\frac{2}{d}\Lambda_d(\mu))&=\Lambda_1(\mu)-\Lambda_{d-1}(\mu)+\frac{(d-2)}{d}\Lambda_d(\mu)\\
&=\mc{U}_1+\mc{U}_{d-1}+\mc{U}_{d}+O(\|\wt{Y}\|^3_{C^{k_1}}).\label{eqn:sum_of_Us}
\end{align}
Using equations \eqref{eqn:U_estimate}, \eqref{eq:wt_f_R_error_Ck}, and that $\ell_5>k_1+\alpha$, we bound the right hand side of equation \eqref{eqn:sum_of_Us} to find
\[
\abs{\lambda_1(\mu)-(-\lambda_d(\mu)+\frac{2}{d}\Lambda_d(\mu))}\le 4C_{8,s}(\lambda_n^{k_3-s/2}\varepsilon_s\varepsilon_{\ell_5}+\varepsilon_{\ell_5}^3). 
\]
But by the failure of estimate \eqref{eqn:lambdad_condition}, we may bound the right hand side of the previous line to obtain:
\begin{equation}\label{eqn:lambda1d_difference}
\abs{\lambda_1(\mu)-(-\lambda_d(\mu)+\frac{2}{d}\Lambda_d(\mu))}\le \frac{4}{C_{9,s}-C_{8,s}}\abs{\lambda_d(\mu)}.
\end{equation}

Note in the above equation that the larger $C_{9,s}$ is the smaller the left hand side of the equation is. We can take $C_{9,s}$ as large as we like and still run the KAM scheme. Running the KAM scheme while having a larger constant $C_{9,s}$ only requires that we assume our initial perturbation is closer to the original system of rotations in the $C^{k_0}$ norm. Hence by assuming that the initial distance is arbitrarily small in the $C^{k_0}$ norm, we may take $C_{9,s}$ as large as we like.  Thus equation \eqref{eqn:taylor_expansion1d} follows from equation \eqref{eqn:lambda1d_difference}.
\end{proof}

We now check the claim about isotropic manifolds.

\begin{prop}\label{prop:isom_not_transitive}
Suppose that $M$ is a closed isotropic manifold other than $\RP^n$ or $S^n$. Then $\Isom(M)$ does not act transitively on $\Gr_k(M)$ except if $k$ equals $=0,1,\dim M-1$ or $\dim M$.
\end{prop}

\begin{proof}
From subsection \ref{subsec:isotropic_manifolds},
we have a list of all the closed isotropic manifolds, so we may give an argument for each of the families, $\CP^n$, $\HP^n$, and $F_4/\Spin(9)$.

The isometry group of $\CP^n$ is $\PSU(n+1)$. If we fix a point $p$ in $\CP^n$, then the isotropy group is naturally identified with $\SU(n)$. It is then immediate that the action of the isotropy group preserves complex subspaces of $\Gr_k(\CP^n)$. Consequently $\Isom(\CP^n)$ does not act transitively on $\Gr_k(\CP^n)$ as $\CP^n$ has subspaces that are not complex.  In the case of $\HP^n$, which is constructed similarly to $\CP^n$, a similar argument works where we use instead that the isotropy group is $\Sp(k)$, the compact symplectic group.

We now turn to the Cayley plane, for which we give a dimension counting argument. The dimension of $F_4$ is $52$   while  $\dim F_4/\Spin(9)=16$.  Recall that if $M$ is a manifold and $\dim M=d$ then $\dim \Gr_k(M)=(k+1)d+\frac{k(k+1)}{2}$. Hence $\dim \Gr_3(F_4/\Spin(9))>52$.  If $\Isom(M)$ acts transitively on $2$-planes then $M$ must have constant sectional curvature and hence is a sphere. The Cayley plane does not have constant sectional curvature hence $k=2$ is ruled out. Similarly, a dimension count excludes the possibility that $F_4$ acts transitively on $\Gr_k(F_4/\Spin(0))$ when $k\neq 0,1,15,16$.
\end{proof}

\appendix

\section{$C^k$ Estimates}\label{sec:ck_calculus}

In this section of the appendix, we collect some basic results concerning the calculus of $C^k$ functions. Most of the estimates stated here are used to compare constructions coming from Riemannian geometry and constructions coming from a chart.

Most of the estimates we prove below involve the following definition, which is an appropriate form for a second order term in the $C^k$ setting.

\begin{defn}\label{defn:quadratic}
Suppose that $X,Y,Z$ are all vector fields and that $Z=Z(X,Y)$ is a function of $X$ and $Y$. We say that $Z$ is \emph{quadratic} in $X$ and $Y$ if  there exists a fixed $\ell$ such that for each $k$ there is a constant $C_k$ depending only on $Z$ such that:
\begin{equation}\label{eqn:defn_quadratic}
\|Z\|_{C^k}\le C_k(\|X\|_{C^{k+\ell}}^2+\|Y\|_{C^{k+\ell}}^2).
\end{equation}
In addition to quadratic, we may also refer to $Z$ as being \emph{second order} in $X$ and $Y$. In the case that $Z$ depends only on $X$ the definition is analogous. 
\end{defn}

One thinks of equation \eqref{eqn:defn_quadratic} as a quadratic tameness estimate. Our main use of this notion is the following proposition, which allows us compose diffeomorphisms up to a quadratic error. As before, if $Y$ is a vector field on $M$, we write $\psi_Y$ for the map of $M$ that sends $x\mapsto \exp_x(Y(x))$. To emphasize that $\psi$ depends on a metric $g$, we may write $\psi_Y^g$.

The main result from this section is the following, which is used in the KAM scheme to see how the linearized error between $f_i$ and $R_i$ changes when $f_i$ is conjugated by a diffeomorphism $\psi$.

\begin{prop}\label{prop:conj_error_field}\cite[Eq. (8)]{dolgopyat2007simultaneous}
Suppose that $(M,g)$ is a closed Riemannian manifold and that $R$ is an isometry of $M$. Suppose that $f$ is a diffeomorphism of $M$ that is $C^1$ close to $R$. Let $Y(x)=\exp^{-1}_{R(x)} f(x)$. If $C$ is a $C^1$ small vector field on $M$, then the error field $\exp^{-1}_{R(x)}\psi_C f\psi_C^{-1}$ is equal to 
\[
Y+C-R_*C+Q(C,Y),
\]
where $Q$ is quadratic in $C$ and $Y$.
\end{prop}

The proof of Proposition \ref{prop:conj_error_field} is straightforward. It particularly relies on the following proposition, which simplifies working with diffeomorphisms of the form $\psi_X$.

\begin{prop}\label{thm:composition_quadratic_estimate}
Let $M$ be a compact Riemannian manifold. If $X,Y\in \Vect^{\infty}(M)$ are sufficiently $C^1$ small and we define $Z$ by
\[
\psi_Y\circ \psi_X=\psi_{X+Y+Z},
\]
then there exists a fixed $\ell$ such that for each $k$ there exists $C_k$ such that 
\[
\|Z\|_{C^k}\le C_k(\|X\|_{C^{k+\ell}}^2+\|Y\|_{C^{k+\ell}}^2),
\]
i.e. $Z$ is quadratic in $X$ and $Y$.
\end{prop}

The proof of Proposition \ref{thm:composition_quadratic_estimate} uses the following two lemmas concerning maps of $\R^n$.

\begin{lem}\label{lem:hormander_product_estimates}
\cite[Thm. A.7]{hormander1976boundary} Suppose that $B$ is a compact convex domain in $\R^n$ with interior points. Then for $k\ge 0$, there exists $C$ such if $f,g$ are $C^k$ maps from $B$ to $\R$, then
\[
\|fg\|_{C^k}\le C_k(\|f\|_{C^k}\|g\|_{C^0}+\|f\|_{C^0}\|g\|_{C^k}).
\]
\end{lem}

\begin{lem}\label{lem:hormander_composition_estimate}
\cite[Thm. A.8]{hormander1976boundary} For $i\in \{1,2,3\}$, let $B_i$ be a fixed compact convex domain in $\R^{n_i}$ with interior points. Let $k\ge 1$. There exists $C_k>0$ such that  if $f\colon B_1\to B_2$ and $g\colon B_2\to B_3$ are both $C^k$, then $f\circ g$ is $C^k$ and 
\[
\|f\circ g\|_{C^k}\le C_k(\|f\|_{C^k}\|g\|_{C^1}^k+\|f\|_{C^1}\|g\|_{C^k}+\|f\circ g\|_{C^0}).
\]
\end{lem}

Using the previous two lemmas, we prove the following.

\begin{prop}\label{prop:exp_comparison1}
Suppose that $g$ is a metric on $\R^n$.  For a smooth vector field $Y$ such that $\|Y\|_{C^1}<1$, define
\[
Z(x)=\psi^g_{Y}(x)-Y(x)-x.
\]
Let $B$ be a compact convex domain in $\R^n$ with interior points. Then $Z\vert_B$ is quadratic in $Y$. In fact, for each $k$ there exists $C_k$ such that 
\[
\|Z\vert_B\|_{C^k}\le C_k\|Y\|_{C^k}^2.
\]
\end{prop}

\begin{proof}

Let $B$ be as in the statement of the proposition. Define $\gamma(Y(x),t)$ to be the map that sends $x\mapsto \exp_xtY(x)-x$, so that $\gamma(Y(x),1)+x=\psi^g_Y$ and $\gamma(Y(x),0)=0$. We rewrite $Z$.
\begin{align*}
Z=\psi_Y^g(x)-x-Y(x)&=\gamma(Y(x),1)-Y(x)\\
&=\int_0^1 \dot{\gamma}(Y(x),t)-Y(x)\,dt\\
&=\int_0^1 \dot{\gamma}(Y(x),t)-\dot{\gamma}(Y(x),0)\,dt\\
&=\int_0^1 \int_0^1t\ddot{\gamma}(Y(x),st)\,ds\,dt\\
&=\int_0^1 t\int_0^1 \ddot{\gamma}(Y(x),st)\,ds\,dt.
\end{align*}
By differentiating under the integral, we see that the $n$th derivatives of $Z$ are controlled by the maximum of the $n$th derivatives of $\ddot{\gamma}(Y(x),t)$ for each fixed $t$. Hence it suffices to show for each $t\in [0,1]$ that $\ddot{\gamma}(Y(x),t)$ is second order in $Y$.

Dropping the explicit dependence on $x$, we recall the coordinate expression of the geodesic equation. For a coordinate frame $[e_1,\ldots,e_n]$ and indices $1\le \mu,\nu,\lambda\le n$, we define the Christoffel symbols $\Gamma_{\mu\nu}^{\lambda}$ by $\langle \nabla_{e_\mu}e_{\nu},e_{\lambda}\rangle.$ In addition, we write $\dot{\gamma}^{\nu}$ for $\langle \dot{\gamma},e_\nu\rangle$ and similarly for $\ddot{\gamma}$. The coordinate expression for the geodesic equation is then
\[
\ddot{\gamma}^{\lambda}=-\Gamma^{\lambda}_{\mu\nu}\dot{\gamma}^\mu\dot{\gamma}^\nu.
\]

We estimate the $C^k$ norm of the right hand side. Write $\phi^t$ for the geodesic flow on $TB$. For fixed $r>0$ in $TB$, let $TB(r)$ be the set of vectors $v\in TB$ such that $\|v\|<r$. Note that the restriction $\|\phi^t\vert_{TB(t)}\|_{C^k}$ is bounded. Let $\pi$ be the projection from a tangent vector in $T\R^n$ to its basepoint in $\R^n$. Then
\[
\gamma(x,t)=\pi\circ \phi^t\circ Y(x).
\]
Hence, writing $\dot{\phi}$ for the geodesic spray,
\begin{equation}\label{eqn:gamma_composed}
\dot{\gamma}(x,t)=D\pi\circ \dot{\phi}\mid_{\phi^t(Y(x))}.
\end{equation}
$D\pi\circ \dot{\phi}^t\mid_{TB(r)}$  has its $C^k$ norm uniformly bounded in $t$ by some $D_k$. By Lemma \ref{lem:hormander_composition_estimate} because $\|Y\|_{C^1}<1$ it follows that $\|\phi^t(Y(x),t)\|_{C^k}\le C_k\|Y\|_{C^k}$.

Hence by applying Lemma \ref{lem:hormander_composition_estimate} to \eqref{eqn:gamma_composed}, and similarly using that $\|Y\|_{C^1}<1$ and $D\pi\circ \dot{\phi}$ is uniformly bounded we find
\[
\|(D\pi\circ \dot{\phi}^t)\circ Y\|_{C^k}\le C_k'(D_k \|Y\|_{C^1}+D_1\|Y\|_{C^k} +\|Y\|_{C^0}).
\]

Hence
\[
\|\dot{\gamma}(x,t)\|_{C^k}=\|D\pi\circ \dot{\phi}\vert_{\phi^t(Y(x))}\|_{C^k}\le C_k\|Y\|_{C^k}.
\]
The geodesic equation shows that at each point the coordinates of $\ddot{\gamma}$ are a quadratic polynomial in the coordinates of $\dot{\gamma}$. Hence by Lemma \ref{lem:hormander_product_estimates} 
\[
\|\ddot{\gamma}(x,t)\|_{C^k}\le C_k''\|Y\|_{C^k}^2
\]
for all $t\in [0,1]$. Thus we obtain a uniform estimate on $Z$.
\end{proof}

\begin{proof}[Proof of Proposition \ref{thm:composition_quadratic_estimate}.]
As before, it suffices to prove the estimate in a chart. So, we are reduced to working in a neighborhood of $0\in \R^n$. Fix some $k$, then by Proposition \ref{prop:exp_comparison1} we may write
\[
\psi_Y(x)=x+Y(x)+Z_Y(x),
\]
 where $Z_Y(x)$ is quadratic in $Y$. Similarly define $Z_X(x)$ and $Z_{X+Y}(x)$. Then

\begin{align*}
\psi_Y\circ \psi_X&=\psi_Y(x+X(x)+Z_X(x))\\
&=x+X(x)+Z_X(x)+Y(x+X(x)+Z_X(x))+Z_Y(x+X(x)+Z_X(x)).
\end{align*}
To prove this proposition, we compare the previous line with
\[
\psi_{X+Y}=x+X(x)+Y(x)+Z_{X+Y}(x).
\]
The difference is 
\[
\psi_Y\circ\psi_X-\psi_{X+Y}=Z_X(x)-Z_{X+Y}(x)+Y(x+X(x)+Z_X(x))-Y(x)+Z_Y(x+X(x)+Z_X(x))
\]
The first and second terms satisfy the appropriate quadratic $C^k$ estimate already. For the last term, we apply Lemma \ref{lem:hormander_composition_estimate}. Hence by assuming that $\|X\|_{C^1}$ is sufficiently small, we conclude that the $Z_Y$ term is quadratic. We now turn to the $Y$ terms:
\[
Y(x+X(x)+Z_X(x))-Y(x).
\]
For this we apply the same trick as before. Write
\[
Y(x+X(x)+Z_X(x))-Y(x)=
\int_0^1 Y'(x+t(X(x)+Z_X(x)))\|X(x)+Z_X(x)\|\,dt.
\]
By differentiating under the integral, it suffices to show that the integrand is quadratic in $X$ and $Y$. By Lemma \ref{lem:hormander_product_estimates}, the integrand will be quadratic if there exists $\ell$ such that for each $k$ there is a constant $C_k$ such that both of $\|Y'(x+t(X(x)+Z_X(x)))\|_{C^k}$ and $\|X(x)+Z_X(x)\|_{C^k}$ are bounded by $C_k(\|X\|_{C^{k+{\ell}}}+\|Y\|_{C^{k+{\ell}}})$. This follows for both terms by the application of Lemma \ref{lem:hormander_composition_estimate}, so we are done.
\end{proof}

We now show another basic fact: near to the identity map a diffeomorphism and its inverse have comparable size.

\begin{lem}\label{lem:ck_inverse_est2}
Suppose that $M$ is a closed Riemannian manifold. Then there exists $\epsilon>0$ such that for all $k\ge 0$ then there exists $C_k$, such that if $f\in \Diff^k(M)$ and $d_{C^1}(f,\Id)<\epsilon$ then 
\[
d_{C^k}(f^{-1},\Id)\le C_k d_{C^k}(f,\Id).
\]
\end{lem}

\begin{proof}
This proof follows the outline of the similar estimate in \cite[Lem. 2.3.6]{hamilton1982inverse}.
For convenience, write $g=f^{-1}$. In a chart, we write $f(x)=x+X(x)$ where the $C^k$ norm of $X$ is bounded by $d_{C^k}(f,\Id)$. Similarly write $g(x)=x+Y(x)$. We now apply the chainrule to differentiate $g\circ f$. The case where $n=1$ is immediate by differentiating $g\circ f=x+X(x)+Y(x+X(x))$, which gives that 
\[
DX+DY(\Id+DX)=0.
\]
Hence
\[
DY=-DX(\Id+DX)^{-1},
\]
which is uniformly comparable to $\|DX\|$ because $d_{C^1}(f,\Id)$ is uniformly bounded.

For $k>1$, we must estimate the higher order derivatives of $Y$. Note that for $k>1$ that $D^kg=D^kY$ and $D^kf=D^kX$.

Applying the chain rule to $f\circ g=\Id$ to calculate the $k$th derivative gives:
\[
0=\sum_{l=1}^k \sum_{j_1+\cdots+j_l=k} C_{l,j_1,...,j_l}D_{g(x)}^lf\{D^{j_1}_xg,\ldots,D_x^{j_l}g\},
\]
and hence 
\begin{equation}\label{eqn:Ck_inverse_lemma}
D^{k}_xg=-(D_{g(x)}f)^{-1}\sum_{l=2}^{k}\sum_{j_1+\cdots+j_l=k}C_{l,j_1,\ldots,j_l}D^l_{g(x)}f\{D^{j_1}_xg(x),\ldots,D^{j_l}_xg(x)\}.
\end{equation}
As $(Df)^{-1}$ has uniformly bounded norm, it suffices to show that the each term in the sum has norm bounded by $\|X\|_{C^n}$.

We use the interpolation estimate in Lemma \ref{lem:ck_interpolation_inequality}. If $j>1$, then 
\[
\|D^{j}g\|=\|D^jY\|,
\]
By interpolation between the $C^1$ and $C^{n-1}$ norms, for $1\le j\le n-1$, 
\[
\|Y\|_{C^j}\le C_{1,n-1} \|Y\|_{C^1}^{\frac{n-j-1}{n-2}}\|Y\|_{C^{n-1}}^{\frac{j-1}{n-2}}.
\]
By interpolation between the $C^1$ and $C^n$ norms, for $1\le j\le n$,
\[
\|X\|_{C^j}\le C_{1,n}\|X\|_{C^1}^{\frac{n-j}{n-1}}\|X\|_{C^{n}}^{\frac{j-1}{n-1}}.
\]

We now estimate the terms in the right hand side of equation \eqref{eqn:Ck_inverse_lemma}. In the case that some $j_i=1$, then $D^{j_i}g=\Id+DY$.  Hence the right hand side of equation \eqref{eqn:Ck_inverse_lemma}, may be rewritten as the sum of terms of the form
\[
D^l_{g(x)}X\{A_1,...,A_l\},
\]
where each $A_i$ is either equal to $\Id$ or $D^{j_i}Y$ and the sum of the $j_i$ is less than or equal to $k$. If $\|Y\|_{C^{k-1}}\le 1$, then we are immediately done as the norm of this expression is at most $\|D^k f\|$. Otherwise, we may suppose that $\|Y\|_{C^{k-1}}\ge 1$. The $C^1$ norms of $X$ and $Y$ are uniformly bounded. Hence by interpolating between the $C^1$ and $C^{k}$ norm to estimate the $D^lX$ term and the $C^1$ and the $C^{k-1}$ norm to estimate the $A_i$ terms, we find that
\[
\|D^k_{g(x)}X\{A_1,...,A_k\}\|\le C'\|X\|_{C^k}^{\frac{l-1}{k-1}}\|Y\|_{C^{k-1}}^{\frac{k-r}{n-2}},
\]
where $r\ge l$. But as $\|Y\|_{C^{k-1}}>1$, this bounded above by 
\[
C'\|X\|_{C^k}^{\frac{l-1}{k-1}}\|Y\|_{C^{k-1}}^{\frac{k-l}{k-2}}.
\]

Thus 
\[
\|D^kY\|_{C^0}\le C''\sum_{l=2}^k \|X\|_{C^k}^{\frac{l-1}{k-1}}\|Y\|_{C^{k-1}}^{\frac{k-l}{k-2}}.
\]
We may now proceed by induction on $k$. We already established the theorem for $k=1$. Now, given that $\|Y\|_{C^{k-1}}\le C_{k-1}\|X\|_{C^{k-1}}$, it follows that
\[
\|D^kY\|_{C^0}\le C'''\sum_{l=2}^k \|X\|_{C^k}^{\frac{l-1}{k-1}}\|X\|_{C^{k-1}}^{\frac{k-l}{k-2}}.
\]
By interpolation between the $1$ and $k$ norms, the uniform bound on the $C^1$ norm, we find that $\|X\|_{C^{k-1}}\le D_k\|X\|_{C^k}^{\frac{k-2}{k-1}}$. This yields
\[
\|D^kY\|_{C^0}\le D'\sum_{l=2}^k \|X\|_{C^k}^{\frac{l-1}{k-1}}\|X\|_{C^{k}}^{\frac{k-l}{k-1}}\le D''\|X\|_{C^k},
\]
which is the desired result.
\end{proof}

We now obtain the following corollary.

\begin{cor}\label{cor:inverses_quadratic_estimate}
Suppose that $M$ is a closed Riemannian manifold. For smooth $C^1$ small vector fields $X$ on $M$, we may write
\[
\psi_X^{-1}=\psi_{-X+Z},
\]
where $Z$ is quadratic in $X$.
\end{cor}

\begin{proof}
To begin we know by Proposition \ref{thm:composition_quadratic_estimate} that  
\[
\psi_{X}\circ \psi_{-X}=\psi_{Z},
\]
where $Z$ is quadratic in $X$. Note that $\psi_{-X}\circ \psi_{Z}^{-1}=\psi_X^{-1}$. By Lemma \ref{lem:ck_inverse_est2}, $\psi_Z^{-1}=\psi_{Z'}$ where $Z'$ is quadratic in $X$. Hence $\psi_{X}^{-1}=\psi_{-X}\circ \psi_{Z'}$. By Proposition \ref{thm:composition_quadratic_estimate}, this gives that $\psi_X^{-1}=\psi_{-X+Z'+Q}$, where $Q$ is quadratic in $X$ and $Z'$. Hence as $Z'$ is quadratic in $X$ and the corollary follows.
\end{proof}

We can now complete the proof of the estimate on the error field of the conjugated system.

\begin{proof}[Proof of Prop. \ref{prop:conj_error_field}.]
To show this, we repeatedly apply Proposition \ref{thm:composition_quadratic_estimate} and Corollary \ref{cor:inverses_quadratic_estimate}. Writing $Z$ for anything second order in $C$ and $Y$, we find:
\begin{align*}
\psi_Cf\psi_C^{-1}&=\psi_C\psi_YR\psi_C^{-1}\\
&=\psi_{C+Y+Z}R\psi_C^{-1}\\
&=\psi_{C+Y+Z}R\psi_{-C+Z}\\
&=\psi_{C+Y+Z+R_*(-C+Z)}R\\
&=\psi_{C+Y-R_*C+Z}R.
\end{align*}
\end{proof}

We now show two additional lemmas that we use in the KAM scheme.

\begin{lem}\label{lem:C^k_composition_estimate}
Let $M$ be a closed Riemannian manifold. Fix $k\ge 1$. There exist $C,\epsilon>0$ such that if $R\in \Isom(M)$ and $f,g\in \Diff^k(M)$ satisfy $d_{C^1}(f,R)<\epsilon$, and $d_{C^1}(g,\Id)<\epsilon$, then 
\[
d_{C^k}(f\circ g,R)\le C_{k}(d_{C^k}(f,R)+d_{C^k}(g,\Id)),
\]
and
\[
d_{C^k}(g\circ f,R)\le  C_{k}(d_{C^k}(f,R)+d_{C^k}(g,\Id)).
\]
\end{lem}

\begin{proof}
We begin with a proof for the first inequality. In coordinates write $f(x)=R(x)+Y(x)$ and $g(x)=x+X(x)$. Then we just need to estimate 
\[
f\circ g(x)-R(x)=R(x+X(x))-R(x)+Y(x+X(x)).
\]
The last term is controlled by $d_{C^k}(f,R)+d_{C^k}(g,\Id)$ by Lemma \ref{lem:hormander_composition_estimate}. So, it suffices to estimate the first term. The $k$th derivative of $R(x+X(x))-R(x)$ is then
\[
\sum_{l=1}^k \sum_{j_1+\cdots+j_l=k} C_{l,j_1,...,j_l}D_{x+X(x)}^lR\{D^{j_1}_xg,\ldots,D_x^{j_l}g\}-D_x^lR.
\]
For all the terms with $l<k$, the same interpolation approach as in Lemma \ref{lem:ck_inverse_est2} gives the appropriate estimate, i.e. they are bounded by
\[
C\sum_{l=1}^{k-1} \|X\|_{C^k}^{\frac{l-1}{k-1}}\|X\|_{C^{k-1}}^{\frac{k-l}{k-2}}.
\] 
There are two remaining terms which are unaccounted for: $D^kR_{x+X(x)}-D^kR_x$. This is bounded by a constant time $\|X\|_{C^0}$ and the result follows.

We now consider the second inequality. As before we must estimate
\[
g\circ f(x)-R(x)=X(x)+Y(R(x)+X(x)).
\]
The important term is the second one. A similar argument to before then gives the result as all derviatives of $R$ are uniformly bounded independent of $R$.
\end{proof}

\begin{lem}\label{lem:C_k_composition_convergence}
Let $M$ be a closed Riemannian manifold and $k\ge 0$. If $g_n\in \Diff^k(M)$ is a sequence of diffeomorphisms and $\sum_n d_{C^k}(g_n,\Id)<\infty$, then the sequence of compositions of diffeomorphisms $h_n=g_ng_{n-1}\cdots g_2g_1$ converges in $C^k$ to a diffeomorphism.
\end{lem}

\begin{proof}
As before, we check in charts. 
Having fixed a chart, write $g_n(x)=x+X_n(x)$. Write $h_n(x)=1+Y_n(x)$. 
Let $a_n=\|X_n\|_{C^k}$ and let $b_n=\|Y_n\|_{C^k}$. Note that
\begin{equation}\label{eqn:coord_expr_hn}
h_{n}(x)=x+Y_{n-1}(x)+X_n(x+Y_n(x)).
\end{equation}
Suppose for the moment that $\|Y_{n-1}\|_{C^k}\le 1$.
Using Lemma \ref{lem:hormander_composition_estimate}  and that $\|Y_{n}\|_{C^k}\le 1$,
\begin{align}
\|X_n(x+Y_{n-1})\|_{C^k}&\le C_k(\|X_n\|_{C^k}\|x+Y_{n-1}\|_{C^1}^k+\|X_n\|_{C^1}\|x+Y_{n-1}\|_{C^k}+\|X\|_{C^0})\\
&\le C_k'(a_n+a_nb_{n-1})\label{eqn:est_on_bn}
\end{align}
Hence it follows from equation \eqref{eqn:coord_expr_hn} that there exists $D_k$ such that if $b_{n-1}\le 1$ then
\[
b_n \le b_{n-1}+D_ka_n(1+b_{n-1}).\\
\]
By induction, under the same assumption that $\|Y_j\|_{C^k}\le 1$ for $j< n$, it follows that 
\[
b_n\le -1+\prod_{i=1}^n(1+D_ka_i).
\]
By noting that $\prod_{i=1}^{\infty} (1+x_n)\le \exp\left(\sum_{i=1}^{\infty} x_n\right)$ for $x_n\ge 0$, we can conclude that a tail of the sequence converges. This follows because as $\sum_n a_n$ converges we can inductively check that these inequalities hold starting the argument from an index $N$ satisfying $\exp(\sum_{i=N}^{\infty} D_k a_i) -1<1$. Hence as a tail of the infinite composition converges so does the whole composition.
\end{proof}

\section{Interpolation Inequalities}\label{sec:interpolation}

There is a basic $C^k$ interpolation inequality, which may be found in the appendix of \cite[Thm A.5]{hormander1976boundary}. It states that:
\begin{lem}\label{lem:ck_interpolation_inequality}
Suppose that $M$ is a closed Riemannian manifold. For $0\le a\le b<\infty$ and $0<\lambda<1$ there exists a constant $C(a,b,\lambda)$ such that for any real valued $C^b$ function $f$ defined on $M$,
\[
\|f\|_{C^{\lambda a+(1-\lambda)b}}\le C\|f\|_{C^a}^{\lambda}\|f\|_{C^{b}}^{1-\lambda}.
\]
\end{lem}

The following is an immediate consequence of Lemma \ref{lem:ck_interpolation_inequality}.

\begin{lem}\label{lem:ck_interpolation_inequality_diffeos}
Suppose that $M$ is a closed Riemannian manifold. There exists $\epsilon>0$ such that for $0\le a\le b<\infty$ and $0<\lambda<1$ there exists a constant $C(a,b,\lambda)$ such that for any $f\in \Diff^{\infty}(M)$ such that $d_{C^0}(f,\Id)<\epsilon$, then
\[
d_{C^{\lambda a+(1-\lambda)b}}(f,\Id)\le Cd_{C^a}(f,\Id)^{\lambda}d_{C^{b}}(f,\Id)^{1-\lambda}.
\]
\end{lem}

\begin{lem}
Consider the space $C^{\infty}(M,N)$ where $M$ and $N$ are Riemannian manifolds and $M$ and $N$ are closed.
For all $j,\sigma>0$, there exists a natural number $k$ and a number $\epsilon_0>0$ such that if $f,g\in C^{\infty}(M,N)$, $\|f-g\|_{H^j} <\epsilon_0 <1$, and $\|f-g\|_{C^k}\le 1/2$ then $\|f-g\|_{C^j}\le \|f-g\|_{H^j}^{1-\sigma}$.
\end{lem}

\begin{proof}
The proof is a relatively straightforward application of the Sobolev embedding theorem and interpolation inequalities.
First, we recall an interpolation inequality for Sobolev norms, see \cite[Thm. 6.5.4]{bergh1976interpolation}. For each $0<\theta<1$, $s_0,s_1$, there exists a constant $C$ such that if we let $s=(1-\theta)s_0+\theta s_1$, then we have 
\[
\|f-g\|_{H^s}\le C\|f-g\|_{H^{s_0}}^{1-\theta}\|f-g\|_{H^{s_1}}^{\theta}.
\]

To begin the proof, note that it suffices to estimate $\|f-g\|_{C^{j+1}}$.  Fix $\ell$ large enough that $H^{\ell}$ embeds compactly in $C^{j+1}$ by a Sobolev embedding. Then pick $k$ large enough that 
\[
\|f-g\|_{H^{\ell}}\le C_{\lambda,\ell}\|f-g\|_{H^{j}}^{1-\theta}\|f-g\|_{H^k}^{\theta},
\]
where $0<\theta<\sigma$. The term $\|f-g\|_{H^k}^{\theta}$ is uniformly bounded by $C_k\|f-g\|_{C^k}^{\theta}$. Hence as $H^{\ell}$ compactly embeds in $C^{j+1}$, there exists $C'>0$ such that
\[
\|f-g\|_{C^{j+1}}\le C'\|f-g\|_{H^{j}}^{1-\theta}=C'\|f-g\|^{\sigma-\theta}_{H^j}\|f-g\|_{H^{j}}^{1-\sigma}.
\]
If we choose $\epsilon_0$ sufficiently small that $C'\|f-g\|^{\sigma-\theta}_{H^j}\le 1$, then the result follows.
\end{proof}

A similar argument shows the following:

\begin{lem}\label{lem:bundle_section_interpolation}
Suppose that $E$ is a smooth Riemannian vector bundle over a closed Riemannian manifold $M$. For all choices $j,\ell, \sigma, D>0$ there exist $k,\epsilon_0$ such that if $f$ is a smooth section of $E$ and $\|f\|_{H^j}\le \epsilon_0<1$ and $\|f\|_{C^k}\le D$ then $\|f\|_{C^\ell}\le \|f\|_{H^j}^{1-\sigma}$.
\end{lem}

\section{Estimate on Lifted Error Fields}

The goal of this section is to prove a technical estimate on the error fields of a lifted system. The proof is a computation in charts.

\begin{lem}\label{lem:error_on_lifts}
Suppose that $M$ is a closed Riemannian manifold. Fix numbers $m,k\ge 0$ and $d$ such that $0\le d\le \dim M$. There exists a constant $C$ such that the following holds. For any tuple $(f_1,...,f_m)$ of diffeomorphisms of $M$ and $(r_1,...,r_m)$ a $C^1$ close tuple of isometries of $M$, let $Y_i$ be the shortest vector field such that $\exp_{r_i(x)}Y_i(x)=f_i(x)$. Let $F_i$ be the lift of $f_i$ to $\Gr_d(M)$ and $R_i$ be the lift of $r_i$ to $\Gr_d(M)$. Let $\wt{Y}_i$ be the shortest vector field on $\Gr_d(M)$ such that $\exp_{R_i(x)}\wt{Y}_i(x)=F_i(x)$. If $\|\sum_i Y_i\|_{C^k}=\epsilon$ and $\max_i \|Y_i\|_{C^k}=\eta$, then
\[
\left\|\sum_{i=1}^m\wt{Y}_i\right\|_{C^{k-1}}\le C(\epsilon+\eta^2).
\]
\end{lem}

\begin{proof}
The proof is straightforward but tedious. We give the proof in the case that each $R_i$ is the identity. Removing this assumption both complicates the argument in purely technical ways and substantially obscures why the lemma is true. At the end of the argument, we indicate the modifications needed for the general proof.

For readability we redevelop some of the basic notions concerning Grassmannians. First we recall the charts on $\Gr_d(V)$, the Grassmannian of $d$-planes in a vector space $V$. Recall that given a vector space $V$ and a pair of complementary subspaces $P$ and $Q$ of $V$ that if $\dim P=d$ we obtain a chart on $\Gr_d(V)$ in the following manner. Let $L(P,Q)$ denote the space of linear maps from $P$ to $Q$. For $A\in L(P,Q)$, we send $A$ to the subspace $\{x+Ax\mid x\in P\}\in \Gr_d(V)$. This gives a smooth parametrization of a subset of $\Gr_d(V)$. Having fixed a complementary pair of subspaces $P$ and $Q$, let $\pi_P$ denote the projection to $P$ along $Q$.

Suppose that $U$ is a chart on $M$ and let $\partial_1,...,\partial_n$ denote the coordinate vector fields. We use the usual coordinate framing of $TU$ to give coordinates on the Grassmannian bundle $\Gr_d(M)$. 
The tangent bundle to $U$ naturally splits into sub-bundles spanned by $\{\partial_1,...,\partial_{d}\}$ and $\{\partial_{d+1},\ldots,\partial_n\}$. Call these sub-bundles $P$ and $Q$, respectively. Let $\End(P,Q)$ denote the bundle of maps from $P$ to $Q$.
We obtain a coordinate chart via associating an element of $A\in \End(P,Q)$ and a point $x\in U$ with the graph of $A$ in the tangent space over $x$. 

As we have assumed that each $r_i$ is the identity, in charts we write $f_i(x)=x+X_i(x)$. As the $f_i$ are $C^1$ small, we work in a single chart. It now suffices to prove the corresponding estimate on the field $X_i$ because $X_i$ and $Y_i$ are equal up to an error that is quadratic in the sense of Definition \ref{defn:quadratic}. We now calculate the action of $f$ on $\Gr_d(U)$. Suppose that $A\in \End(P,Q)$. Then we have that $\{Df(v+Av)\}$ is a subspace of $T_{f(x)}M$. We must find the map $A'$ whose graph gives the same subspace. Let $I_A$ be the $n\times d$ matrix with top block $I$ and bottom block $A$. Then the action of $Df$ sends $A$ to $A'$ which is equal to
\[
 A'=Df I_A(\pi_P Df I_A)^{-1}-\Id.
\]
To see that this is true, we must check that $A'V\subseteq Q$ and that $\{Dfv+DfAv\mid v\in V\}$ is the same as $\{v+A'v\}$. The second condition is evident from the definition of $A'$.
If $v\in P$, then $(\pi_P Df I_A)^{-1}v=w$ is an element of $P$ satisfying $\pi_P DfI_Aw=v$. Thus $A'v=DfI_A(\pi_PDfI_A)^{-1}v-v\in Q$ and hence $A'V\subseteq Q$. 
Write $F$ for the induced map on $\Gr_d(U)$. In coordinates $F$ is the map that sends
\begin{equation}
(x,A)\mapsto(x,Df I_A(\pi_P Df I_A)^{-1}-\Id).
\end{equation}
Write $I_d$ for the $d$ by $d$ identity matrix. Let $\widehat{DX_i}$ be the matrix comprised of the first $d$ rows of the matrix $DX_i$. In the estimates below, we will assume that the size of $A$ is uniformly bounded. This does not restrict the generality as any subspace may be represented by such a uniformly bounded $A$.  Then note that 
\begin{align*}
(\pi_PD_f\left[\frac{I_d}{A}\right])^{-1}&=(I_d+\widehat{DX}\left[\frac{I_d}{A}\right])^{-1}\\
&=I_d-\widehat{DX}\left[\frac{I_d}{A}\right]+O(DX^2),
\end{align*}
where the $O(DX^2)$ is quadratic in the sense of Definition \ref{defn:quadratic}.
Write $X_A$ for the second term above.

We then have that 
\begin{align*}
Df I_A(\pi_P Df I_A)^{-1}-\Id&=(\Id+DX)\left[\frac{I_d}{A}\right](I_d-X_A)-\left[\frac{I_d}{0}\right]+O(DX^2)\\
&=\left[\frac{I_d}{A}\right]-\left[\frac{I_d}{A}\right]X_A+DX\left[\frac{I_d}{A}\right]+DX\left[\frac{I_d}{A}\right]X_A-\left[\frac{I_d}{0}\right]+O(DX^2)\\
&=\left[\frac{0}{A}\right]-\left[\frac{I_d}{A}\right]X_A+DX\left[\frac{I_d}{A}\right]+O(DX^2)\\
&=\left[\frac{0}{A}\right]+H(A,DX)+O(DX^2),
\end{align*}
where $H(A,DX)$ is the sum of the second and third terms two lines above. Note that $H$ is linear in $DX$ and that $\|H(A,DX)\le C\|DX\|$ given our uniform boundedness assumption on $A$.

Thus we see that in this chart on $\Gr_d(U)$ that 
\begin{equation}
F(x,A)-(x,A)=(f(x)-x,H(A,DX)+O(DX^2)).
\end{equation}

In this chart, $\|\sum_i f_i(x)-x\|_{C^k}\le \epsilon$. Hence writing $f_i(x)=x+X_i(x)$ as before, $\|\sum_i DX_i(x)\|_{C^{k-1}}\le \epsilon$. Thus 
\begin{align*}
\left\|\sum_i F_i(x,A)-(x,A)\right\|_{C^{k-1}}&=\left\|\sum_i (f_i(x)-x,H(A,DX_i)+O(DX^2))\right\|_{C^{k-1}}\\
&\le C\pez{\left\|\sum_i X_i\right\|_{C^k}+\max_i \|X_i\|_{C^{k}}^2}
\end{align*}
by the linearity of $H$. This completes the proof in the special case where $r_i=\Id$ for each $i$.

In the general setting one follows the same sequence of steps. One writes $f_i(x)=r_i(x)+X_i(r_i(x))$. One then does the same computation to determine the action on the Grassmannian bundle. This is complicated by additional terms related to $R$. Having finished this computation, one finds a natural analog of $H(A,DX)$, which now comprises eight terms instead of two, and also depends on $r_i$. Recognizing the cancellation is then somewhat complicated because of the dependence on $r_i$. However, this dependence does not cause an issue because the terms that would potentially cause trouble satisfy some useful relations. These relations emerge when one keeps in mind the base points, which is crucial when the isometries are non-trivial.
\end{proof}

\section{Determinants}\label{sec:determinants}

Suppose that $V$ and $W$ are finite dimensional inner product spaces. Consider a linear map $L\colon V\to W$. The determinant of the map $L$ is defined as follows. If $\{v_i\}$ is an orthonormal basis for $V$, one may measure the size of the tensor $Lv_1\wedge \cdots \wedge Lv_n$ with respect to the norm on tensors induced by the metric on $W$.   If $\{v_1,...,v_n\}$ is a basis for $V$, then we define
\[
\det(L,g_1,g_2)\coloneqq \sqrt{\frac{\Det\left( \langle Lv_i,Lv_j\rangle_{g_2}\right) }{\Det \left(\langle v_i,v_j\rangle_{g_1}\right)}},
\]
where $\Det$ is the usual determinant of a square matrix. Sometimes we have a map $L\colon V\to W$ and a subspace $E\subset V$. We then define 
\begin{equation}
\det(L,g_1,g_2\mid E)=\det(L\vert_E,g_1\vert_E,g_2).
\end{equation}
When the spaces $V$ and $W$ are understood, we may write $\det(L\mid E)$.

There are some properties of $\det$ that we will record for later use.
\begin{lem}
Fix a basis and suppose that $V=W$. Working with respect to this basis, the determinant has the following properties:
\begin{align}
\det(L,g_1,g_2)&=\det(\Id,g_1,L^*g_2),\\
\det(\Id,\Id,A)&=\sqrt{\det(A,\Id,\Id)}=\sqrt{\abs{\Det(A)}}.
\end{align}
\end{lem}

\begin{proof}
For the first equality, let $\{v_i\}$ be a basis of $(V,g_1)$, then
\[
\det(L,g_1,g_2)=\sqrt{\frac{\Det \langle Lv_i,Lv_j\rangle_{g_2} }{\Det \langle v_i,v_j\rangle_{g_1}}}
\]
But, $\langle v_i,v_j\rangle_{L^*g_2}=\langle Lv_i,Lv_j\rangle_{g_2}$, so, this is equal to
\[
\sqrt{\frac{\Det \langle v_i,v_j\rangle_{L^*g_2} }{\Det \langle v_i,v_j\rangle_{g_1}}},
\]
which is the definition of $\det(\Id,g_1,L^*g_2)$.

For the second equality, fix an orthonormal basis $\{e_i\}$, then
\[
\det(\Id,\Id,A)=\sqrt{\Det\langle e_i,e_j\rangle_{A}}=\sqrt{\Det A_{ij}}
\]
whereas,
\[
\det(A,\Id,\Id)=\sqrt{\Det\langle Ae_i,Ae_j\rangle_{\Id}}=\sqrt{\Det A^TA}=\sqrt{\abs{\Det A}^2}=\abs{\Det A}.
\]

\end{proof}

We record the following estimate which is used in the proof.

\begin{lem}\label{lem:log_det_on_grassmannian}
Let $M$ be a closed manifold and let $0\le r\le \dim M$. If $g$ is an isometry of $M$, then $\ln \det(Df\vert E_x)$, which is defined on $\Gr_r(M)$, satisfies the following estimate:
\[
\left\| \ln\det(Df\mid E_x)\right\|_{C^k}=O(d_{C^{k+1}}(f,g)),
\]
as $f\to g$ in $C^{k+1}$.
The big-O is uniform over all isometries $g$.
\end{lem}

\begin{proof}
It suffices to show that this estimate holds in charts. So, fix a pair of charts $U$ and $V$ on $M$ such that $f(U)$ has compact closure inside of $V$. We define a map $H\colon \Gr_d(U)\times U\times V\times \R^{n^2}\to \R$ by sending the point $(E,x,y,A)$ to the $\ln \det(A, g_x,g_y\vert E)$, where $g_x$ and $g_y$ denote the pullback metric from $M$. Using $f$ we define a map $\wt{f}\colon \Gr_d\times U\to \Gr_d(U)\times U\times V\times \R^{n^2}$ by 
\[
(E,x)\mapsto (E,x,f(x),Df),
\]
where we are using the coordinates to express $Df$ as a matrix. Then the quantity we wish to estimate the $C^k$ norm of is $H\circ \wt{f}$. If we analogously define $\wt{g}$, then note that $H\circ \wt{g}\equiv 0$ because $g$ is an isometry. By writing out the derivatives using the chain rule and using that $f$ is uniformly close to $g$, one sees that $\|H\circ \wt{g}-H\circ \wt{f}\|_{C^k}=O(d_{C^{k+1}}(f,g))$, and the result follows.
\end{proof}

\section{Taylor Expansions}\label{sec:taylor_expansions}

\subsection{Taylor expansion of the log Jacobian}
\begin{prop}\label{prop:taylor_expansion_of_log_jacobian}
For $C^1$ small vector fields $Y$ on a Riemannian manifold $M$, the following approximation holds
\[
\int_{\Gr_r(M)} \ln\det(D_x\psi_Y,\Id,g_{\psi_{Y(x)}}\mid E_x)\,d\vol=
-\frac{r}{2d} \int_M \|E_C\|^2\,d\vol +\frac{r(d-r)}{(d+2)(d-1)}\int_M\|E_{NC}\|^2\,d\vol+O(\|Y\|_{C^1}^3),
\]
where $E_C$ and $E_{NC}$ are the conformal and non-conformal strain tensors associated to $\psi_Y$ as defined in subsection \ref{subsec:strain}. In addition, $\det$ is defined in Appendix \ref{sec:determinants} and $\psi_Y$ is defined in equation \eqref{eq:psi_defn}.
\end{prop}

The proof of this proposition is a lengthy computation with several subordinate lemmas.

\begin{proof}

In order to estimate the integral over $M$, we will first obtain a pointwise estimate on:
\[
\int_{\Gr_r(T_xM)} \ln\det (D_x\psi_Y\mid E)\,dE.
\]
To estimate this we work in an exponential chart on $M$ centered at $x$. In this chart, $x$ is $0$ and $\psi_{Y}(0)=Y(0)$. Then
\[
\int_{\Gr_r(T_xM)} \ln\det (D_x\psi_Y\mid E)\,dE=\int_{\Gr_r(T_xM)} \ln\det(D_0\psi_Y, \Id, g_{Y(0)}\mid E)\,dE.
\]
We now rewrite the above line so that we can apply the Taylor approximation 
 in Proposition \ref{prop:det_integral_expansion}.

Write the metric as $\Id+\hat{g}$. As we are in an exponential chart, $\|\hat{g}_{Y(0)}\|=O(\|Y\|^2_{C^0})$. Write $D\psi_Y=\Id+\hat{\psi}$. The integral we are calculating only involves $\hat{\psi}_0$ and $\hat{g}_{Y(0)}$, so below we drop the subscripts. Then
\[
\int_{\Gr_r(T_xM)} \ln\det (D_x\psi_Y\mid E)\,dE=
\int_{\Gr_r(T_xM)} \ln\det(\Id+\hat{\psi},\Id,\Id+\hat{g}\mid E)\,dE.
\]
Now applying the Taylor expansions in Propositions \ref{prop:det_integral_expansion} and \ref{prop:det_taylor_expansion2}, we obtain the following expansion. For convenience let 
\begin{equation}\label{eqn:defn_of_K}
K=(\hat{\psi}+\hat{\psi}^T)/2-\frac{\Tr \hat{\psi}}{d}\Id.
\end{equation}
Then
\begin{align}\label{align:log_det_1}
\int_{\Gr_r(T_xM)} \ln\det(D\psi_Y,\Id,g_{Y(0)}\mid E)\,dE=\\
\frac{r}{d}\Tr(\hat{\psi})+\left[-\frac{r}{2d}\Tr(\hat{\psi}^2)+\frac{r(d-r)}{(d+2)(d-1)}\Tr(K^2)\right]+O(\|\hat{\psi}^3\|)+\frac{r}{2d}\Tr(\hat{g})+O(\|\hat{g}\|^2)
\end{align}

Note that $\|\hat{\psi}\|=O(\|Y\|_{C^1})$ and $\|\hat{g}\|=O(\|Y\|^2_{C^0})$, hence the fourth and sixth terms in the above expression are each $O(\|Y\|^3_{C^1})$.

 We now eliminate the two trace terms that are not quadratic in their arguments.
For this, we use a Taylor expansion of the determinant.\footnote{Recall the usual Taylor expansion $\Det(\Id+A)=\Id+\Tr(A)+\frac{(\Tr(A))^2-\Tr(A^2)}{2}+O(\|A\|^3)$. We combine this with the first order Taylor expansion 
\[
\det(\Id,\Id,\Id+G)=\sqrt{\Det(1+G)}=\sqrt{1+\Tr(G)+O(\|G\|^2)}=1+\frac{\Tr(G)}{2}+O(\|G\|^2).
\]
}
Thus 
\[
\det(D\psi,\Id,g_{Y(0)})=1+\Tr\hat{\psi} +\frac{(\Tr(\hat{\psi}))^2-\Tr(\hat{\psi}^2)}{2}+\frac{\Tr(\hat{g})}{2}+O(\|Y\|^3_{C^1})
\]

The integral of the Jacobian is $1$, so integrating the previous line over $M$ against volume we obtain
\[
1=1+\int_M \Tr\hat{\psi} +\frac{(\Tr(\hat{\psi}))^2-\Tr(\hat{\psi}^2)}{2}+\frac{\Tr(\hat{g})}{2}\,d\vol+O(\|Y\|^3_{C^1}).
\]
Thus
\[
\int_M \Tr(\hat{\psi})+\frac{\Tr(\hat{g})}{2}-\frac{\Tr(\hat{\psi}^2)}{2}\,d\vol=-\int_M \frac{(\Tr(\hat{\psi}))^2}{2}\,d\vol+O(\|Y\|^3_{C^1}).
\]
Now, we integrate equation \eqref{align:log_det_1} over $M$ and apply the previous line to eliminate the non-quadratic terms. This gives
\begin{align}\label{align:log_det_2}
\int_{\Gr_r(M)} \ln\det(D_x\psi_Y,\Id,g_{\psi_{Y(x)}}\mid E_x)\,dE_x=
\int_M -\frac{r}{2d} (\Tr(\hat{\psi}_x))^2 +\frac{r(d-r)}{(d+2)(d-1)}\Tr(K_x^2)\,d\vol+O(\|Y\|^3_{C^1}),
\end{align}
where we have written $\hat{\psi}_x$ and $K_x$ to emphasize the basepoint.
The formula above is not yet very usable as both $K_x$ and $\hat{\psi}_x$ are defined in terms of exponential charts. We now obtain an intrinsic expression for these terms. Recall that pointwise we use the $L^2$ norm on tensors. Below we suppress the $x$ in $\|E_C(x)\|$ and $\hat{\psi}_x$.

\begin{lem}\label{lem:taylor_expansion_trace_psi}
Let $E_C$ be the conformal strain tensor associated to $\psi_Y$. Then
\[
\int_M (\Tr(\hat{\psi}_x))^2\,d\vol=\int \|E_C\|^2\,d\vol+O(\|Y\|^3_{C^1}).
\]
\end{lem}

\begin{proof}
We use an exponential chart and compute a coordinate expression for $\|E_C\|^2$ in the center of this chart. As before, write $D\psi_Y=\Id+\hat{\psi}$, where $\hat{\psi}=O(\|Y\|_{C^1})$. Then working in exponential coordinates, 
\begin{align*}
\Tr(\psi^*_Yg-g)&=\Tr((\Id+\hat{\psi})^T(\Id+O(\|Y\|^2_{C^0}))(\Id+\hat{\psi})-\Id)\\
&=\Tr(\Id+\hat{\psi}^T+\hat{\psi}-\Id+O(\|Y\|^2_{C^1})\\
&=2\Tr(\hat{\psi})+O(\|Y\|^2_{C^1}).
\end{align*}
Thus since $\hat{\psi}=O(\|Y\|_{C^1})$, by definition of $E_C$, we have
\begin{align*}
\|E_C\|^2&=\left\|\frac{\Tr(\psi_Y^*g-g)}{2d}\Id\right\|^2\\
&=\left\|\frac{2\Tr(\hat{\psi})}{2d}\Id\right\|^2\\
&=\frac{\Tr(\hat{\psi})}{d}\|\Id\|^2\\
&=\Tr(\hat{\psi}).
\end{align*}
Integrating over $M$, we obtain the result.
\end{proof}

\begin{lem}\label{lem:taylor_expansion_trace_K}
Let $E_{NC}$ be the non-conformal strain tensor associated to $\psi_Y$ and let $K_x$ be as in equation \eqref{eqn:defn_of_K}, then
\[
\int_M \Tr\pez{K^2_x}\,d\vol=\int_M\|E_{NC}\|^2\,d\vol+O(\|Y\|^3_{C^1}).
\]
\end{lem}
\begin{proof}
As before, we first compute a local expression for the integrand and check that this expression is comparable to the local expression for the non-conformal strain tensor. We compute at the center of an exponential chart. As before, write $D\psi_Y=\Id+\hat{\psi}$ where $\hat{\psi}=O(\|Y\|_{C^1})$. In this case
\[
\psi^*_Yg=(\Id+\hat{\psi})^T(\Id+O(\|Y\|^2_{C^0}))(\Id+\hat{\psi})=\Id+\hat{\psi}^T+\hat{\psi}+O(\|Y\|^2_{C^1}).
\]
Using the above line and the definition of $E_{NC}$ we then compute:
\begin{align*}
\|E_{NC}\|^2&=\left\|\frac{1}{2}\pez{\psi^*_Yg-g-\frac{\Tr(\psi^*_Yg-g)}{d}g}\right\|^2\\
&=\frac{1}{4}\|(\Id+\hat{\psi})^T(\Id+O(\|Y\|^2_{C^0})(\Id+\hat{\psi})-\Id-2\frac{\Tr\hat{\psi}}{d}\Id+O(\|Y\|^2_{C^1})\|^2\\
&=\frac{1}{4}\|\hat{\psi}^T+\hat{\psi}-2\frac{\Tr\hat{\psi}}{d}\Id+O(\|Y\|^2_{C^1})\|^2\\
&=\frac{1}{4}\Tr\pez{\pez{\hat{\psi}^T+\hat{\psi}-2\frac{\Tr\hat{\psi}}{d}\Id+O(\|Y\|^2_{C^1})}^2}\\
&=\Tr\pez{\pez{\frac{\hat{\psi}^T+\hat{\psi}}{2}-\frac{\Tr\hat{\psi}}{d}\Id}^2}+O(\|Y\|^3_{C^1})\\
&=\Tr(K^2)+O(\|Y\|^3_{C^1}).
\end{align*}
By integrating the above equality over $M$, the result follows.
\end{proof}

Finally, the proof of Proposition \ref{prop:taylor_expansion_of_log_jacobian} follows by applying Lemma \ref{lem:taylor_expansion_trace_psi} and Lemma \ref{lem:taylor_expansion_trace_K} to equation \eqref{align:log_det_2}, which gives

\begin{align}
\int_{\Gr_r(M)} \ln\det(D_x\psi_Y,\Id,g_{\psi_{Y(x)}}\mid E_x)\,dE_x=
-\frac{r}{2d} \int_M \|E_C\|^2\,d\vol +\frac{r(d-r)}{(d+2)(d-1)}\int_M \|E_{NC}\|^2\,d\vol+O(\|Y\|_{C^1}^3).
\end{align}

\end{proof}

\subsection{Approximation of integrals over Grassmanians}
Let $\mb{G}_{r,d}$ be the Grassmanian of $r$-planes in $\R^d$. In this subsection, we prove the following simple estimate.

\begin{prop}\label{prop:det_integral_expansion}
For $1\le r\le d$, let $\Lambda_r\colon \End(\R^d)\rightarrow \R$ be defined by 
\[
\Lambda_r(L):=\int_{\mb{G}_{r,d}} \ln \det (\Id+L,\Id,\Id\mid E)\,dE,
\]
where $dE$ denotes the Haar measure on $\mb{G}_{r,d}$. Then the second order Taylor approximation for $\Lambda_r$ at $0$ is 
\[
\Lambda_r(L)=\frac{r}{d}\Tr L+\left[-\frac{r}{2d}\Tr(L^2)+\frac{r(d-r)}{(d+2)(d-1)}\Tr(K^2)\right]+O(\|L\|^3),
\]
where 
\[
K=\frac{L+L^T}{2}-\frac{\Tr L}{d}\Id.
\]

Let $\lambda_r(L)=\Lambda_{r}(L)-\Lambda_{r-1}(L)$. Then the above expansion implies
\[
\lambda_r(L)=\frac{1}{d}\Tr L+\left[-\frac{1}{2d}\Tr(L^2)+\frac{d-2r+1}{(d+2)(d-1)}\Tr(K^2) \right]+O(\|L\|^3).
\]

\end{prop}

\begin{proof}
Before beginning, note from the definition of $\Lambda_r$ that if $U$ is an orthogonal transformation, $\Lambda_r(U^TLU)=\Lambda_r(L)$. Consequently, if $\alpha_i$ is the $i$th term in the Taylor expansion of $\Lambda_r$, then $\alpha_i$ is invariant under conjugation by isometries. 

The map $\Lambda_r$ is smooth, so it admits a Taylor expansion:
\[
\Lambda_r(L)=\alpha_1(L)+\alpha_2(L)+O(\|L\|^3),
\]
where $\alpha_1$ is linear in $L$ and $\alpha_2$ is quadratic in $L$. The rest of the proof is a calculation of $\alpha_1$ and $\alpha_2$. Before we begin this calculation we describe the approach. In each case, we reduce to the case of a symmetric matrix $L$. Then restricted to symmetric matrices, we diagonalize. There are few linear or quadratic maps from $\End(\R^n)$ to $\R$ that are invariant under conjugation by an orthogonal matrix. We then write $\alpha_i$ as a linear combination of such invariant maps from $\End(\R^n)$ to $\R$ and then solve for the coefficients of this linear combination.

We begin by calculating $\alpha_1$.
\begin{claim}
With notation as above, 
\[
\alpha_1(L)=\frac{r}{d}\Tr L.
\]
\end{claim}

\begin{proof}

Let $\wt{\Lambda}_r(\Id+L)=\Lambda_r(L)$. Then from the definition, note that if $U$ is an isometry then $\wt{\Lambda}_r(U(\Id+L))=\wt{\Lambda}_r((\Id+L)U))=\Lambda_r(L)$. Suppose that $O_t$ is some path tangent to $O(n)\subset \End(\R^n)$ such that $O_0=\Id$. Then $\wt{\Lambda}_r(O_t)=0$. Write $O_t=\Id+tS+O(t^2)$ where $S$ is skew symmetric. Then we see that 
\[
\wt{\Lambda}_r(\Id+tS+O(t^2))=O(t^2),
\]
So, $\Lambda_r(tS)=O(t^2)$. Hence $\alpha_1$ vanishes on skew symmetric matrices. 

Thus it suffices to evaluate $\alpha_1$ restricted to symmetric matrices.
Suppose that $A$ is a symmetric matrix, then there exists an orthogonal matrix $U$ so that $U^TAU$ is diagonal. Restricted to the space of diagonal matrices, which we identify with $\R^d$ in the natural way, observe that $\alpha_1\colon \R^d \to \R$ is invariant under permutation of the coordinates in $\R^d$ because it is invariant under conjugation by isometries. There is a one dimensional space of maps having this property, and it is spanned by the trace, $\Tr$. So, $\alpha_1(A)=\alpha_1(U^TAU)=a_1\Tr(A)$ for some constant $C$. To compute the constant $c$ it suffices to consider a specific matrix, e.g. $A=\Id$. 
\begin{align*}
\alpha_1(\Id)&= \frac{d}{d\epsilon} \int \ln\det(\Id+\epsilon \Id\mid E)\,dE\\
&=\frac{d}{d\epsilon}\int \ln (\Id+\epsilon)^r\,dE\\
&=\frac{d}{d\epsilon} r\ln(1+\epsilon)\\
&=r.
\end{align*}
So, $a_1=r/d$. Thus for $L\in \End(\R^d)$, $\alpha_1(L)=\frac{r}{d}\Tr((L+L^T)/2)=\frac{r}{d}\Tr(L)$.
\end{proof}

We now compute $\alpha_2$.

\begin{claim}
With notation as in the statement of Proposition \ref{prop:det_integral_expansion},
\[
\alpha_2(L)=-\frac{r}{2d}\Tr(L^2)+\frac{r(d-r)}{(d+2)(d-1)}\Tr(K^2).
\]
\end{claim}

\begin{proof}
Let $\wt{\Lambda}_r(\Id+L)=\Lambda_r(L)$. From the definition, note that for an isometry $U$, that $\wt{\Lambda}_r((\Id+L)U)=\wt{\Lambda}_r(U)$.
Fix $L$ and let $J=(L-L^T)/2$. Observe that 
\[
(\Id+L)e^{-J}=\Id+(L-J)+(J^2/2-LJ)+O(\abs{L}^3).
\]
Thus we see that 
\begin{align*}
\Lambda_r(L)&=\wt{\Lambda}_r(\Id+L)\\
&=\wt{\Lambda}_r((L+\Id)e^{-J})\\
&=\wt{\Lambda}_r(\Id+(L-J)+(J^2/2-LJ)+O(\abs{L}^3))\\
&=\Lambda_r((L-J)+(J^2/2-LJ))+O(\abs{L}^3).
\end{align*}
Now comparing the two Taylor expansions of $\wt{\Lambda}_r(\Id+L)$, we find:
\[
\alpha_2(L)=\alpha_2(L-J)+\alpha_1(J^2/2-LJ).
\]
Thus  as we have already determined $\alpha_1$:
\[
\alpha_2(L)=\alpha_2((L+L^T)/2)+\frac{r}{d}\Tr(J^2/2-LJ).
\]

So, we are again reduced to the case of a symmetric matrix $S$. In fact, by invariance of $\alpha_2$ under conjugation by isometries, we are reduced to determining $\alpha_2$ on the space of diagonal matrices. Identify $\R^d$ with diagonal matrices as before. We see that $\alpha_2$ is a symmetric polynomial of degree $2$ in $d$ variables. The space of such polynomials is spanned by $\sum x_i^2$  and $\sum_{i,j} x_ix_j$. It is convenient to observe that for a diagonal matrix, $D$, $\Tr(D^2)$ and $\Tr(D)^2$ span this space as well. Hence 
\[
\alpha_2(S)=b_1\Tr(S)^2+b_2\Tr(S^2)
\]

Now in order to calculate $b_1$ and $b_2$ we will explicitly calculate $\alpha_2(\Id)$ and $\alpha_2(P)$, where $P$ is the orthogonal projection onto a coordinate axis.

In the first case,

\begin{align*}
2\alpha_2(\Id)=\frac{d}{d\epsilon_1}\frac{d}{d\epsilon_2}\int_{\mathbb{G}_{r,d}} \ln \det((\Id+\epsilon_1+\epsilon_2)\mid E)\,dE\mid_{\epsilon_1=0,\epsilon_2=0}
&=\frac{d^2}{d\epsilon}\ln (1+\epsilon)^r\mid_0=-r.
\end{align*}
So, $\alpha_2(\Id)=-r/2$.

Next suppose that $P$ is projection onto a fixed vector $e$. Suppose that $\angle(e,E)=\theta$. We now compute $\ln\det(\Id+\epsilon P\mid E)$. We fix a useful basis of $E$. Let $v$ be a unit vector making angle $\angle(e,E)$ with $e$. Then let $e_2,...,e_r$ be unit vectors in $E$ that are orthogonal to $e$ and $v$.
Then using the basis $v,e_2,...,e_r$, we see that 
\[
\det(\Id+\epsilon P\mid E)=\frac{\|(\Id+\epsilon P)v\wedge (\Id+\epsilon P)e_2\wedge \cdots\wedge (\Id+\epsilon P)e_r\| }{\|v\wedge\cdots\wedge e_r\|}=\sqrt{\langle (\Id+\epsilon P)v,(\Id+\epsilon P)v\rangle},
\]
by considering the determinant defining the wedge product. But then as $Pv=\cos(\theta)e$,
\[
\sqrt{\langle v+\epsilon\cos(\theta)e,v+\epsilon\cos(\theta)e\rangle}=\sqrt{\langle v,v\rangle+2\epsilon\cos\theta\langle v,e\rangle +\epsilon^2\langle Pv,Pv\rangle}=\sqrt{1+2\epsilon \cos^2\theta+\epsilon^2\cos^2\theta}.
\]

Now, the Taylor approximation for $\ln\sqrt{1+x}$ at $x=0$ is $x/2-x^2/4+O(x^3)$, so
\[
\ln \det (\Id+\epsilon P\mid E)=\epsilon \cos\angle (E,e)+\epsilon^2\left[\frac{\cos\angle(E,e)}{2}-\cos^4\angle(E,e)\right]+O(\epsilon^3).
\]
Hence, as this estimate is uniform over $E$, by integrating,
\[
\int_{\mb{G}_{r,d}} \ln \det(\Id+\epsilon P\mid E)\,dE=\epsilon \int_{\mb{G}_{r,d}} \cos^2\angle(E,e)\,dE+\epsilon^2\int_{\mb{G}_{r,d}}\left[\frac{\cos^2\angle(E,e)}{2}-\cos^4\angle(E,e)\right]\,dE+O(\epsilon^3).
\]

So, we are reduced to calculating the coefficient of $\epsilon^2$ in the above expression. One may rewrite the above integrals in the following manner, by definition of the Haar measure as $\mb{G}_{r,d}$ is a homogeneous space of $\SO(d)$. Write $x_1,...,x_d$ for the restriction of the Euclidean coordinates to the sphere. By fixing the coordinate plane $E_0=\langle e_1,...,e_r\rangle$, and letting $\theta=\angle((x_1,...,x_d),E)$ we then have that $\cos(\theta)=\sqrt{\sum_{i=1}^r x_i^2}$. Thus
\begin{align*}
\int_{\mb{G}_{r,d}} \cos^2\angle (E,e)\,dE&=\int_{\SO_d}\cos^2\angle(gE_0,e)\,dg\\
&=\int_{\SO_d}\cos^2\angle(E_0,ge)\,dg\\
&=\int_{S^{d-1}}\cos^2\angle (E_0,x)\,dx\\
&=\int_{S^{d-1}}\sum_{i=1}^r x_i^2\,dx,
\end{align*}

Similarly, fixing the plane $E_0=\langle e_1,...,e_r\rangle$, we see that as $\cos^4\angle (E_0,x)=\pez{\sum_{i=1}^r x_i^2}^2$
\[
\int_{\mb{G}_{r,d}} \cos^4\angle (E,e)=\int_{S^{d-1}}\pez{\sum_{i=1}^r x_i^2}^2\,dx.
\]

The evaluation of these integrals is immediate by using the following standard formulas:
\[
\int_{S^{d-1}} x_1^2\,dx=\frac{1}{d},\quad \int_{S^{d-1}} x_1^4\,dx=\frac{3}{d(d+2)},\quad \int_{S^{d-1}} x_1^2x_2^2\,dx=\frac{1}{d(d+2)}.
\]
Thus we see that 
\[
\int_{\mb{G}_{r,d}} \frac{\cos^2\angle(E,e)}{2}-\cos^4\angle(E,e)\,dE=\frac{r}{2d}-\frac{r(r+2)}{d(d+2)}.
\]
Thus
\[
\alpha_2(P)=\frac{r}{2d}-\frac{r(r+2)}{d(d+2)}.
\]

Returning to $b_1,b_2$, the coefficients of $(\Tr(S))^2$ and $\Tr(S^2)$, respectively, combining the cases of $\Id$ and $P$ gives
\[
-\frac{r}{2}=b_1d^2+b_2d.
\]
and
\[
\frac{r}{2d}-\frac{r(r+2)}{d(d+2)}=b_1+b_2.
\]
We can now solve for $b_1$ and $b_2$ with respect to this basis of the space of conjugation invariant quadratic functionals. However, the computation will be more direct if instead we we use a different basis and write write $\alpha_2(S)$ as
\[
b_1(\Tr(S))^2+b_2\Tr(\pez{S-\frac{\Tr S}{d}}^2),
\]
so that the second term is trace $0$. Our computations from before now show that:
\[
-\frac{r}{2}=b_1d^2+0,
\]
and
\[
\frac{r}{2d}-\frac{r(r+2)}{d(d+2)}=b_1+\frac{d-1}{d}b_2\left(=b_1(\Tr(P))^2+b_2\Tr\pez{(P-\frac{\Tr P}{d}\Id)^2}\right).
\]

The first equation implies that 
\[
b_1=-\frac{r}{2d^2},
\]
The left hand side of the second equation of the pair is equal to
\[
\frac{r(d-r)}{d(d+2)}-\frac{r}{2d}.
\]
This gives
\[
b_2=\frac{r(d-r)}{(d-1)(d+2)}-\frac{r}{2d}.
\]
So, for symmetric $L$, we have
\begin{equation}
\alpha_2(S)=\frac{-r}{2d^2}(\Tr(S))^2+\pez{\frac{r(d-r)}{(d-1)(d+2)}-\frac{r}{2d}}\Tr((S-\frac{\Tr S}{d}\Id)^2).
\end{equation}
Recall that we specialized to the case of a symmetric matrix, and that for a non-symmetric matrix there is another term. For $L\in \End{\R^d}$, setting $J=(L-L^T)/2$, as before,
\[
\alpha_2(L)=\alpha_2\pez{\frac{L+L^T}{2}}+\frac{r}{d}\Tr\pez{\frac{J^2}{2}-LJ}.
\]
To simplify this we compute that:
\begin{align*}
\Tr\pez{\frac{J^2}{2}-LJ}&=\Tr\pez{\frac{L^2-LL^T-L^TL+(L^T)^2}{8}-L\frac{L-L^T}{2}}\\
&=\Tr\pez{\frac{LL^T-L^2}{4}}.
\end{align*}
Write
\[
S=\frac{L+L^T}{2}.
\]
Observe that for an arbitrary matrix $X$, $\Tr((X-(\Tr X)/d\Id)^2)=\Tr(X^2)-(\Tr(X))^2/d$. Thus
\begin{align*}
&-\frac{r}{2d^2}\pez{\Tr(S)}^2-\frac{r}{2d}\Tr(\pez{S-(\Tr S)/d\Id)^2}+\frac{r}{d}\Tr(\frac{LL^T-L^2}{4})\\
&=-\frac{r}{2d^2}\pez{\Tr(S)}^2-\frac{r}{2d}\pez{\Tr(S^2)}-\frac{-r}{2d^2}(\Tr(S))^2+\frac{r}{d}\pez{\Tr\pez{\frac{LL^T-L^2}{4}}}\\
&=-\frac{r}{2d}\pez{\Tr(S^2)}+\frac{r}{d}\pez{\Tr\pez{\frac{LL^T-L^2}{4}}}\\
&=\frac{r}{d}\left[\frac{-1}{2}\Tr(((L+L^T)/2)^2)+\Tr(\frac{LL^T-L^2}{4})\right]\\
&=\frac{r}{d}\left[\frac{-1}{2}(\Tr(\frac{L^2+(L^T)^2+2LL^T}{4}))+\Tr(\frac{LL^T-L^2}{4})\right]\\
&=-\frac{r}{2d}\Tr(L^2).
\end{align*}
From before, we have that 
\[
\alpha_2(L)=-\frac{r}{2d^2}(\Tr(S))^2+\pez{\frac{r(d-r)}{(d-1)(d-2)}-\frac{r}{2d}}\Tr((S-\frac{\Tr S}{d}\Id)^2)+\frac{r}{d}\Tr(\frac{LL^T-L^2}{4}).
\]
So substituting the previous calculation we obtain:
\[
\alpha_2(L)=-\frac{r}{2d}\Tr(L^2)+\pez{\frac{r(d-r)}{(d-1)(d-2)}}\Tr\pez{\pez{\frac{L+L^T}{2}-\frac{\Tr L}{d}\Id}^2},
\]
which is the desired formula.
\end{proof}
We have now calculated $\alpha_1$ and $\alpha_2$. This concludes the proof of Proposition \ref{prop:det_integral_expansion}.
\end{proof}

We will also use a first order Taylor expansion as well with respect to the metric.

\begin{prop}\label{prop:det_taylor_expansion2}

Let $\Lambda_r(G)$ be defined for symmetric matrices $G$ by
\[
\Lambda_r(G)\coloneqq \int_{\mb{G}_{r,d}} \ln \det(\Id,\Id,\Id+G\mid E)\,dE.
\]
Then $\Lambda_r(G)$ admits the following Taylor development:
\[
\Lambda_r(G)=\frac{r}{2d}\Tr G+O(\|G\|^2).
\]
\end{prop}

\begin{proof}
The proof of this proposition is substantially similar to that of the previous proposition. Let $\alpha_1$ denote the first term in the Taylor expansion. Note that if $U$ is an isometry that $\Lambda_r(U^TGU)=\Lambda_r(G)$. Thus $\alpha_1$ is invariant under conjugation by isometries. Thus by conjugating by an orthogonal matrix, we are reduced to the case of $G$ and diagonal matrix. As before, we see that $\alpha_1(D)$ is a multiple of $\Tr(D)$ as $\Tr$ spans the linear forms on $\R^d$ that are invariant under permutation of coordinates.

Thus it suffices to calculate the derivative in the case of $D=\Id$. So, we see that 
\[
\alpha_1(\Id)=\frac{d}{d\epsilon}\int_E \ln\det(\Id,\Id,\Id+\epsilon \Id\mid E)\,dE.
\]
Thus the integral is equal to $\ln \sqrt{(1+\epsilon)^r}$ on every plane $E$. Thus the derivative is $r/2$ and so
\[
\alpha_1(\Id)=\frac{r}{2}=\frac{r}{2d}\Tr(\Id).
\]
And so the result follows.
\end{proof}

\providecommand{\bysame}{\leavevmode\hbox to3em{\hrulefill}\thinspace}
\providecommand{\MR}{\relax\ifhmode\unskip\space\fi MR }
\providecommand{\MRhref}[2]{%
  \href{http://www.ams.org/mathscinet-getitem?mr=#1}{#2}
}
\providecommand{\href}[2]{#2}

\end{document}